%% file: HeisenbergHomology.tex
\documentclass[10pt]{article}
\usepackage{authblk}
\usepackage[colorlinks,citecolor=blue,linkcolor=blue,urlcolor=blue,filecolor=blue,breaklinks]{hyperref}
\usepackage{footnotebackref}

\usepackage[utf8]{inputenc}
\usepackage{amssymb}
\usepackage{amsfonts}
\usepackage{amsmath}
\usepackage{amsthm}
\usepackage{mathtools}
\usepackage{mathrsfs}
\usepackage{cleveref}
\usepackage{relsize}
\usepackage{tikz-cd}
\usepackage{geometry}
\usepackage{graphicx}
\usepackage{booktabs}
\usepackage{array}
\usepackage{paralist}
\usepackage{verbatim}
\usepackage{subfig}
\usepackage{epsfig}
\usepackage{changepage}

\usepackage{fancyhdr}
\usepackage{sectsty}
\allsectionsfont{\sffamily\mdseries\upshape}
\usepackage[nottoc,notlof,notlot]{tocbibind}
\usepackage[titles,subfigure]{tocloft}

\definecolor{lightblue}{rgb}{0.8,0.8,1}

\theoremstyle{plain}
\newtheorem{theorem}{Theorem}
\newtheorem{atheorem}{Theorem}
\newtheorem{aproposition}[atheorem]{Proposition}

\newtheorem{corollary}[theorem]{Corollary}

\newtheorem{lemma}[theorem]{Lemma}

\newtheorem{proposition}[theorem]{Proposition}

\theoremstyle{definition}
\newtheorem{definition}[theorem]{Definition}
\newtheorem{remark}[theorem]{Remark}
\newtheorem{observation}[theorem]{Observation}
\newtheorem{notation}[theorem]{Notation}

\newtheoremstyle{named}%
    {}{}{\itshape}{}{\bfseries}{.}{.5em}{\thmnote{#3}}
\theoremstyle{named}

\newcommand{\R}{\mathbb{R}}
\newcommand{\Z}{\mathbb{Z}}
\newcommand{\Q}{\mathbb{Q}}
\newcommand{\C}{\mathbb{C}}

\newcommand{\id}{\mathrm{Id}}
\renewenvironment{proof}[1][Proof]{\noindent\textbf{#1.} }{\ \rule{0.5em}{0.5em}\par\addvspace{\baselineskip}}
\newenvironment{proofnobox}[1][Proof]{\noindent\textbf{#1.} }{}
\newcommand{\Heis}{\mathcal{H}}
\newcommand{\Heisr}{\mathcal{H}_{\mathbb{R}}}
\newcommand{\Diff}{\mathrm{Diff}}
\newcommand{\cL}{\mathcal{L}}
\newcommand{\Hom}{\mathrm{Hom}}
\newcommand{\bimod}[2]{\ensuremath{{}_{#1}\mathrm{Mod}_{#2}}}
\newcommand{\hlf}{\ensuremath{\mathrm{hlf}}}
\newcommand{\actiongroupoidleft}{\ensuremath{\mathbin{\setminus\!\!\setminus}}}
\newcommand{\actiongroupoidright}{\ensuremath{\mathbin{/\!\!/}}}

\renewcommand{\geq}{\geqslant}
\renewcommand{\leq}{\leqslant}

\definecolor{cycleBM}{RGB}{0,150,0}
\definecolor{cycleD}{RGB}{255,0,0}

\input{sage-preamble.tex}

\begin{document}
\title{Heisenberg homology on surface configurations}
\author{Christian Blanchet\protect\footnote{Universit{\'e} Paris Cit\'e and Sorbonne Universit{\'e}, CNRS, IMJ-PRG, F-75006 Paris, France}, Martin Palmer\protect\footnote{Simion Stoilow Mathematical Institute of the Romanian Academy, Bucharest, Romania; School of Mathematics, University of Leeds, Leeds, LS2 9JT, UK}, Awais Shaukat\protect\footnote{Department of Mathematics, Namal University, 30 km Talagang Road, Mianwali, 42250, Pakistan }}
\date{10 March 2025}
\maketitle
{
\makeatletter
\renewcommand*{\BHFN@OldMakefntext}{}
\makeatother
\footnotetext{\hspace{1em} ORCIDs: 0000-0002-8948-5993 (CB), 0000-0002-1449-5767 (MP), 0000-0002-0249-2227 (AS)}
\footnotetext{\hspace{1em} Acknowledgements: This paper forms a part of the PhD thesis of the third author. The first and third authors are thankful for the support of the Abdus Salam School of Mathematical Sciences, Lahore. The third author is also thankful for the support of Namal University, Mianwali. The second author is grateful to Arthur Souli{\'e} for several enlightening discussions about the Moriyama and Magnus representations and their kernels, and for pointing out the reference \cite{Suzuki2003}. The second author was partially supported by a grant of the Romanian Ministry of Education and Research, CNCS - UEFISCDI, project number PN-III-P4-ID-PCE-2020-2798, within PNCDI III, as well as a grant of the Ministry of Research, Innovation and Digitization, CNCS - UEFISCDI, project number PN-IV-P1-PCE-2023-2001, within PNCDI IV.}
}

\begin{abstract}
Motivated by the Lawrence-Krammer-Bigelow representations of the classical braid groups, we study the homology of unordered configurations in an orientable genus-$g$ surface with one boundary component, over non-commutative local systems defined from representations of the discrete Heisenberg group. Mapping classes act on the local systems and for a general representation of the Heisenberg group we obtain a representation of the mapping class group that is twisted by this action. For the linearisation of the affine translation action of the Heisenberg group we obtain a genuine, \emph{untwisted} representation of the mapping class group. In the case of the generic Schr{\"o}dinger representation, by composing with a Stone-von Neumann isomorphism we obtain a projective representation by bounded operators on a Hilbert space, which lifts to a representation of the \emph{stably universal} central extension of the mapping class group. We also discuss the finite dimensional  Schr{\"o}dinger representations, especially in the even case. Based on a natural intersection pairing, we show that our representations preserve a sesquilinear form.

\end{abstract}

\section*{Introduction}

The braid group $B_m$ was defined by Artin in terms of geometric braids in $\R^3$; equivalently, it is the fundamental group of the configuration space $\mathcal{C}_m(\R^2)$ of $m$ unordered points in the plane. Another equivalent description is as the mapping class group $\mathfrak{M}(\mathbb{D}_m)=\Diff(\mathbb{D}_m,S^1)/\Diff_0(\mathbb{D}_m,S^1)$ of the closed $2$-disc with $m$ interior points removed. (The \emph{mapping class group} of a surface is the group of isotopy classes of self-diffeomorphisms fixing the boundary pointwise.)

There is also a natural action of $\Diff(\mathbb{D}_m,S^1)$ on configuration spaces $\mathcal{C}_n(\mathbb{D}_m)$; considering the induced action on the homology of these configuration spaces, Lawrence \cite{Lawrence1990}  defined a representation of $B_m$ for each $n \geq 1$. The $n=2$ version is known as the \emph{Lawrence-Krammer-Bigelow representation}, and a celebrated result of Bigelow~\cite{Bigelow2001} and Krammer~\cite{Krammer2002} states that this representation of $B_m$ is \emph{faithful}, i.e.~injective. 

On the other hand, for almost all other surfaces $\Sigma$, the question of whether the mapping class group $\mathfrak{M}(\Sigma)$ admits a faithful, finite-dimensional representation over a field (whether it is \emph{linear}) is open. The mapping class group of the torus is $SL_2(\Z)$, which is evidently linear, and the mapping class group of the closed orientable surface of genus $2$ was shown to be linear by Bigelow and Budney~\cite{Bigelow-Budney}, as a corollary of the linearity of $B_5$. However, nothing is known in genus $g \geq 3$.

Our programme is to study the action of the positive-genus and connected-boundary mapping class groups $\mathfrak{M}(\Sigma_{g,1})$ on the homology of the configuration spaces $\mathcal{C}_n(\Sigma_{g,1})$, equipped with local systems that are similar to the Lawrence-Krammer-Bigelow construction. We first argue that \emph{abelian} local systems would not retain enough information, in the sense that they cannot faithfully encode the ``writhe'' of loops of configurations.
In general, for any surface $\Sigma$ and $n \geq 2$, the abelianisation of $\pi_1(\mathcal{C}_n(\Sigma))$ is canonically isomorphic to $H_1(\Sigma) \times C$, where $C$ is a cyclic group of order $\infty$ if $\Sigma$ is planar (embeds into $\R^2$), of order $2n-2$ if $\Sigma = \mathbb{S}^2$ and of order $2$ in all other cases (see for example \cite[Proposition 6.32]{DPS}). In the case $\Sigma = \mathbb{D}_m$, the abelianisation is $\Z^m \times \Z$, and the Lawrence representations are defined using the local system given by the quotient $\pi_1(\mathcal{C}_n(\mathbb{D}_m)) \twoheadrightarrow \Z^m \times \Z \twoheadrightarrow \Z \times \Z$, where the second map is addition of the first $m$ factors. However, in the non-planar case (in particular if $\Sigma = \Sigma_{g,1}$), we \emph{lose information} by passing to the abelianisation, since the cyclic factor $C$ -- which counts the self-winding or ``writhe'' of a loop of configurations -- has order $2$ rather than order $\infty$.

To obtain a better analogue of the Lawrence representations in the setting $\Sigma = \Sigma_{g,1}$ for $g>0$, we consider instead a larger, non-abelian quotient of $\pi_1(\mathcal{C}_n(\Sigma))$, which is isomorphic to the discrete Heisenberg group $\Heis = \Heis(\Sigma)$, defined as the central extension of the first homology $H=H_1(\Sigma,\Z)$ associated to the intersection $2$-cocycle, which concretely means $\Heis=\Z\times H$ as a set, with group law $(k,x)(l,y)=(k+l+x.y,x+y)$. This is a $2$-nilpotent group that arises very naturally as a quotient of the surface braid group $\pi_1(\mathcal{C}_n(\Sigma))$ by forcing a single element to be central. It may also be realised concretely as a group of $(g+2) \times (g+2)$ matrices, as explained in Remark \ref{rmk:Heisenberg-matrices} below. In the case $n\geq 3$ it is known by \cite{B_al2008} to be the \emph{$2$-nilpotentisation} of the surface braid group (in fact it is the maximal nilpotent quotient of the surface braid group), but for $n=2$ it differs from the $2$-nilpotentisation. A key property of this Heisenberg quotient is that it still detects the self-winding (or ``writhe'') of a loop of configurations \emph{without reducing modulo two}. Any representation $V$ of the discrete Heisenberg group $\Heis(\Sigma)$ defines a local system on the configuration space $\mathcal{C}_n(\Sigma)$. 

An and Ko studied in \cite{An} extensions of the Lawrence-Krammer-Bigelow representations to homological representations of surface braid groups; see also \cite{BGG2017}. Their purpose was to extend the homological representation of the classical braid group to some homology of configurations in an $n$-punctured surface and produce representations of the surface braid groups. In our case the surface has no punctures, and the goal is to represent the full mapping class group. Our constructions based on the Heisenberg quotient of the surface braid group have a similar flavour but are significantly simpler; moreover we obtain strong improvements by specialising to explicit representations.

We speculate about faithfulness results for our representations and linearity results for the mapping class group. This would involve two steps.

\begin{enumerate}
\item Prove that the action on the homology of the Heisenberg covering space of $\mathcal{C}_n(\Sigma)$ is faithful. Following Bigelow's strategy, this would follow from a key lemma showing that an algebraic intersection form on homology  detects the geometric intersection of curves on the surface.
\item Find a good finite-dimensional representation of the Heisenberg group that retains faithfulness.
\end{enumerate}

It was shown in  \cite{Crivelli93,Crivelli94} that the adjoint representation of quantum $sl(2)$ at roots of $1$ has a topological realisation as homology of configurations  with local coefficients in the once-punctured torus. Following this programme, De Renzi and Martel \cite{DeRenziMartel} have recently produced a homological model for non-semisimple TQFT representations derived from quantum $sl(2)$. They use the local system on surface configurations given by the Schr{\"o}dinger representation at an odd root of $1$, which is a special case in our construction. We believe that our work contributes to a promising programme for topological interpretations of quantum constructions and possible classical constructions of quantum invariants and TQFTs. 

\begin{notation}
Henceforth we will use the abbreviation $\Sigma = \Sigma_{g,1}$ for an integer $g\geq 1$.
\end{notation}

\paragraph{General representations.}
Our first main result is a calculation of a Borel--Moore relative homology group with coefficients twisted by any representation of the Heisenberg group, together with a twisted action of the mapping class group. In the following, $H_*^{BM}$ denotes Borel--Moore homology and $\mathcal{C}_n(\Sigma,\partial^-(\Sigma))$ is the properly embedded subspace of $\mathcal{C}_n(\Sigma)$ consisting of all configurations intersecting a given closed arc  $\partial^-\Sigma\subset \partial\Sigma$. The twisted action is formulated as a representation of an \emph{action groupoid}. The key point is that the mapping class group acts on the Heisenberg group, which induces an action on our local systems. We denote by $f_\Heis \in \mathrm{Aut(\Heis)}$ the automorphism induced by $f\in \mathfrak{M}(\Sigma)$. For a left representation $\rho \colon \Heis \rightarrow GL(V)$ and $\tau \in  \mathrm{Aut}(\Heis)$, the $\tau$-twisted representation $\rho\circ \tau$ is denoted by ${}_\tau\!V$.

\begin{atheorem}[Theorems \ref{basis} and \ref{thm:MCG}]
\label{athm:twisted}
Let $n\geq 2$, $g\geq 1$ and let $V$ be a left representation of the discrete Heisenberg group $\Heis = \Heis(\Sigma)$ in $(R,S)$-bimodules, for unital rings $R$ and $S$.\\
\textup{(a)} The Borel--Moore homology $H_n^{BM}(\mathcal{C}_{n}(\Sigma),\mathcal{C}_{n}(\Sigma,\partial^-(\Sigma)) ;V)$ is isomorphic, as an $(R,S)$-bimodule, to the direct sum of 
$\bigl(
\begin{smallmatrix}
2g+n-1 \\
n \\
\end{smallmatrix}
\bigr)$
copies of $V$.
Furthermore, it is the only non-zero bimodule in the graded bimodule $H_*^{BM}(\mathcal{C}_{n}(\Sigma),\mathcal{C}_{n}(\Sigma,\partial^-(\Sigma)) ;V)$.
The corresponding statements are also true if Borel--Moore homology $H_*^{BM}$ is replaced with compactly-supported cohomology $H^*_c$. \\
\textup{(b)} There is a natural twisted representation of the mapping class group $\mathfrak{M}(\Sigma)$ on the collection of $(R,S)$-bimodules
\begin{equation*}
H_n^{BM}\bigl( \mathcal{C}_n(\Sigma) , \mathcal{C}_n(\Sigma,\partial^-(\Sigma)) ;{}_\tau\!V \bigr) \ , \quad \tau \in \mathrm{Aut}(\Heis)\ ,\ 
\end{equation*}
where the action of $f\in \mathfrak{M}(\Sigma)$ is
\begin{equation}
\label{eq:athm:twisted}
\mathcal{C}_n(f)_* \colon H_n^{BM}\bigl( \mathcal{C}_n(\Sigma) , \mathcal{C}_n(\Sigma,\partial^-(\Sigma)) ;{}_{\tau\circ f_\Heis}\!V \bigr) \longrightarrow H_n^{BM}\bigl( \mathcal{C}_n(\Sigma) , \mathcal{C}_n(\Sigma,\partial^-(\Sigma)) ;{}_{\tau}\!V \bigr) .
\end{equation}
\end{atheorem}

\begin{remark}
The case of a trivial representation of $\Heis$ is already something interesting; indeed, connecting with Moriyama's work \cite{Moriyama}, we show that the Johnson filtration is recovered; see \S\ref{relation-to-Moriyama}.
\end{remark}

\begin{remark}
The Heisenberg group $\Heis(\Sigma)$ can be realised as a group of $(g+2) \times (g+2)$ matrices (see Remark \ref{rmk:Heisenberg-matrices}); this gives a $(g+2)$-dimensional representation, which we refer to as its \emph{tautological} representation. We then obtain, for each $n\geq 2$, a family of twisted representations with polynomially growing dimension equal to
$(g+2) \bigl(
\begin{smallmatrix}
2g+n-1 \\
n \\
\end{smallmatrix}
\bigr)$.
\end{remark}

\paragraph{The linearised translation action.}

The discrete Heisenberg group $\Heis$ has a natural affine structure over $\Z$ for which the left translation action $\Heis \curvearrowright \Heis$ is by affine automorphisms. The linearisation functor, from the category of affine spaces over $\Z$ to the category of $\Z$-modules, applied to this affine action, gives a representation $L = \Heis \oplus \Z \cong \Z^{2g+2}$ of $\Heis$ over $\Z$. A key feature of this representation is that, for an automorphism $\tau$ of $\Heis$, the twisted representation ${}_\tau L$ is \emph{canonically} isomorphic to $L$.
We deduce a genuine (i.e.~\emph{un}twisted) representation of the mapping class group.

\begin{atheorem}[Theorem \ref{thm:l-regular}]
\label{athm:l-regular}
For each $n\geq 2$ and $g\geq 1$ there is a representation of the mapping class group ${\mathfrak{M}}(\Sigma)$ on the free $\Z$-module of rank
$(2g+2) \bigl(
\begin{smallmatrix}
2g+n-1 \\
n \\
\end{smallmatrix}
\bigr)$,
\begin{equation}
\label{eq:athm:l-regular}
H_n^{BM}\bigl( \mathcal{C}_n(\Sigma) , \mathcal{C}_n(\Sigma,\partial^-(\Sigma)) ; L\bigr) .
\end{equation}
\end{atheorem}

\paragraph{The Schr\"odinger representation.}

The centre of the real Heisenberg group $\Heisr(\Sigma)$ is one-dimensional, and acts by scalars on the Hilbert space $W = L^2(\R^g)$ by $t \mapsto e^{it\hbar/2}$, where $\hbar$ is a fixed non-zero real number (the Planck constant, for physicists). The famous Stone-von Neumann Theorem (see for example \cite[page 19]{LionVergne}, recalled as Theorem \ref{th:LionVergne} below) states that there is, up to isomorphism, a unique irreducible unitary representation of $\Heisr(\Sigma)$ on $W$ extending this action of its centre; this is the \emph{Schr\"odinger representation}. We also denote by $W$ this representation restricted to the discrete subgroup $\Heis = \Heis(\Sigma) \subset \Heisr(\Sigma)$. It depends on the parameter $\hbar$, so that we have a continuous family of Schrödinger representations $W=W(\hbar)$. For $\tau \in \mathrm{Aut}(\Heis)$ the twisted representation ${}_\tau W$ is isomorphic to $W$ as a unitary representation and this isomorphism is unique up to a unit complex number. Using such isomorphisms we may identify the twisted local system with the original one and obtain an untwisted representation of the mapping class group to the projective group of bounded operators on the  homology with local coefficients $W$. Here, the Hilbert structure on homology is specified by a choice of CW-complex structure. We build a linear lift of this projective action to the stably universal central extension $\widetilde{\mathfrak{M}}(\Sigma)$.

\begin{atheorem}[Theorem \ref{thm:Schroedinger}]
\label{athm:universal-central}
For each $n\geq 2$ and $g\geq 1$ there is a representation of $\widetilde{\mathfrak{M}}(\Sigma)$ on the complex Hilbert space
\begin{equation}
\label{eq:athm:universal-central}
\mathcal{V}_n=H_n^{BM}\bigl( \mathcal{C}_n(\Sigma) , \mathcal{C}_n(\Sigma,\partial^-(\Sigma)) ; W\bigr) 
\end{equation}
by bounded operators, which lifts the natural projective action of
$\mathfrak{M}(\Sigma)$.
\end{atheorem}

The group $\widetilde{\mathfrak{M}}(\Sigma)$ on which we construct our linear representation is a central extension of the mapping class group $\mathfrak{M}(\Sigma)$ of the form:
\begin{equation}
\label{eq:univ-central-ext}
0 \to \Z \longrightarrow \widetilde{\mathfrak{M}}(\Sigma) \longrightarrow \mathfrak{M}(\Sigma) \to 1,
\end{equation}
and is the \emph{stably universal central extension} of $\mathfrak{M}(\Sigma)$, which we explain next.

\paragraph{The stably universal central extension.}

A group $G$ has a \emph{universal central extension} (an initial object in the category of central extensions of $G$) if and only if $H_1(G;\Z)=0$, and it is of the form $0 \to H_2(G;\Z) \to \widetilde{G} \to G \to 1$ when it exists (see \cite[Theorem 6.9.5]{Weibel}). For genus $g\geq 4$, we have $H_1(\mathfrak{M}(\Sigma_{g,1});\Z) = 0$ and $H_2(\mathfrak{M}(\Sigma_{g,1});\Z) \cong \Z$ (see \cite[Theorems 5.1 and 6.1]{Korkmaz}). Moreover, there are natural inclusion maps
\begin{equation}
\label{eq:stabilisation-maps}
\mathfrak{M}(\Sigma_{1,1}) \longrightarrow \mathfrak{M}(\Sigma_{2,1}) \longrightarrow \cdots \longrightarrow \mathfrak{M}(\Sigma_{g,1}) \longrightarrow \mathfrak{M}(\Sigma_{g+1,1}) \longrightarrow \cdots ,
\end{equation}
which induce isomorphisms on $H_1(-;\Z)$ and $H_2(-;\Z)$ for $g\geq 4$ (by homological stability for mapping class groups of surfaces, due originally to Harer~\cite{Harer1985}; see \cite[Theorem 1.1]{Wahl2013} for the optimal stability range). This implies that, for $g\geq 4$, the pullback along \eqref{eq:stabilisation-maps} of the universal central extension of $\mathfrak{M}(\Sigma_{g+1,1})$ to $\mathfrak{M}(\Sigma_{g,1})$ is the universal central extension of $\mathfrak{M}(\Sigma_{g,1})$. Hence we may define, for all $g\geq 1$, the \emph{stably universal central extension} of $\mathfrak{M}(\Sigma_{g,1})$ to be the pullback along \eqref{eq:stabilisation-maps} of the universal central extension of $\mathfrak{M}(\Sigma_{h,1})$ for any $h \geq \mathrm{max}(g,4)$.

\paragraph{A finite-dimensional Schr\"odinger representation.}
When the parameter $\hbar$ controlling the action of the centre is $2\pi$ times a rational number, the discrete Heisenberg group has finite-dimensional Schr\"odinger representations, which may be realised either by theta functions, by induction or by an abelian TQFT. For a positive even integer $N$, we will follow \cite{Gelca_Hamilton,Gelca_Uribe_TQFT,Gelca_Uribe}, which connect nicely the different approaches when $\hbar=\frac{2\pi}{N}$.  We denote by $W_N = L^2((\Z/N)^g)$ the $N^g$-dimensional representation that is the unique unitary irreducible representation of the finite quotient $\Heis_N = \Heis/I_N$ of $\Heis$ by the normal subgroup $I_N=\{(2Nk,Nx) \mid k\in \Z , x\in H\}\subset \Heis=\Z\times H$, where each central element $(k,0)$ acts by $e^\frac{i\pi k}{N}$. The analogue of the Stone-von Neumann Theorem in this context \cite[Theorem~2.4]{Gelca_Uribe_TQFT} allows us to construct an untwisted representation of a finite-index subgroup of the mapping class group to a projective linear group. We identify this subgroup as the stabiliser subgroup $\mathfrak{M}(\Sigma,q_0)$ for the spin structure represented by the quadratic form $q_0 \colon H_1(\Sigma;\Z/2) \to \Z/2$ that is zero on the preferred basis.

\begin{atheorem}[Theorem \ref{thm:Schroedinger-fd}]
\label{athm:universal-central-fd}
For each $g\geq 1$, $n\geq 2$ and $N\geq 2$ with $N$ even, there is a complex projective  representation of $\mathfrak{M}(\Sigma,q_0)$ on the
$\bigl(
\begin{smallmatrix}
2g+n-1 \\
n \\
\end{smallmatrix}
\bigr) N^{g}$-dimensional
complex Hilbert space
\begin{equation}
\label{eq:athm:universal-central-fd}
\mathcal{V}_{N,n}=H_n^{BM}\bigl( \mathcal{C}_n(\Sigma) , \mathcal{C}_n(\Sigma,\partial^-(\Sigma)) ; W_N\bigr) 
\end{equation}
\end{atheorem}

\begin{remark}
A similar construction for odd $N$ is used in \cite{DeRenziMartel}. In this case the Stone-von Neumann Theorem applies to the quotient $\Heis_N=\Heis/I_N$, $I_N=\{(Nk,Nx) \mid k\in \Z , x\in H\}\subset \Heis=\Z\times H$, and produces a projective  action of the full mapping class group on the homology spaces $\mathcal{V}_{N,n}$.
\end{remark}

For any complex vector space $V$, the adjoint action of $GL(V)$ on $\mathrm{End}_{\C}(V)$ induces a canonical embedding $PGL(V) \hookrightarrow GL(\mathrm{End}_{\C}(V))$. Applying this to the natural projective action $\mathfrak{M}(\Sigma)\rightarrow PGL(\mathcal{V}_{N,n})$ for odd $N$, we obtain an \emph{untwisted complex representation}
\begin{equation}
\label{eq:fd-rep-of-mcg}
\mathfrak{M}(\Sigma) \longrightarrow GL (\mathrm{End}_{\C}(\mathcal{V}_{N,n}))
\end{equation}
of dimension
$\bigl(
\begin{smallmatrix}
2g+n-1 \\
n \\
\end{smallmatrix}
\bigr)^2 N^{2g}$.
Because $PGL(\mathcal{V}_{N,n}) \hookrightarrow GL(\mathrm{End}_{\C}(\mathcal{V}_{N,n}))$ is injective, we see that:

\begin{observation}
Injectivity of the representation \eqref{eq:fd-rep-of-mcg} is equivalent to injectivity of the projective representation $\mathfrak{M}(\Sigma)\rightarrow PGL(\mathcal{V}_{N,n})$. Thus a proof of injectivity of $\mathfrak{M}(\Sigma)\rightarrow PGL(\mathcal{V}_{N,n})$ for any $(N,n)$ with $N,n\geq 2$ and $N$ odd would imply that the mapping class group $\mathfrak{M}(\Sigma)$ is linear. The same observation holds for $N$ even, with $\mathfrak{M}(\Sigma)$ replaced by its finite-index subgroup $\mathfrak{M}(\Sigma,q_0)$.
\end{observation}

\paragraph{Unitarity.}

When using a Hilbert space as local coefficients, although a CW-complex structure can be used to specify a Hilbert structure on (cellular) homology, it is not true that mapping classes will act as unitary operators on cellular chains. This is due to the fact that the cellular approximation theorem produces a homotopic map that may fail to be a homeomorphism. It is nevertheless possible to find a kind of unitarity property similar to the one stated for the Burau and Gassner representations in \cite{Squier,BarNatan}. We will state this as the property that a certain perfect sesquilinear form on homology is preserved; see \S\ref{ss:sesquilinear-form}, in particular Proposition \ref{prop:perfect-sesquilinear-form}.

\paragraph{Kernels.}

To describe an upper bound on the kernels of our representations, we first recall the \emph{Johnson filtration} of the mapping class group.

The mapping class group $\mathfrak{M}(\Sigma)$ acts naturally on the fundamental group $\pi_1(\Sigma) =: \Gamma_1$ of the surface. Denote by $\Gamma_i$, $i\geq 2$, the subgroups of the lower central series defined recursively by $\Gamma_i := [\Gamma_1,\Gamma_{i-1}]$.  Each term of the lower central series of a group is fully invariant, so there is a well-defined induced action of $\mathfrak{M}(\Sigma)$ on the quotient $\pi_1(\Sigma) / \Gamma_{i+1}$, which is the largest $(i+1)$-step nilpotent quotient of $\pi_1(\Sigma)$. The \emph{Johnson filtration} $\mathfrak{J}(*)$ is then defined by setting $\mathfrak{J}(i)$ to be the kernel of this induced action. Thus $\mathfrak{J}(0)$ is the whole mapping class group and $\mathfrak{J}(1)$ is the Torelli group. The intersection of all terms in the filtration is trivial, i.e., it is an \emph{exhaustive} filtration of the mapping class group \cite{Johnson-survey}.

One may also consider the induced action of the mapping class group $\mathfrak{M}(\Sigma)$ on the universal metabelian quotient $\pi_1(\Sigma) / \pi_1(\Sigma)^{(2)}$ of the fundamental group of the surface (the quotient by its second derived subgroup); its kernel is the \emph{Magnus kernel} of $\mathfrak{M}(\Sigma)$, which we denote by $\mathrm{Mag}(\Sigma) \subseteq \mathfrak{M}(\Sigma)$. In \S\ref{relation-to-Moriyama} (Proposition \ref{prop:kernel}) we prove:

\begin{aproposition}[Proposition \ref{prop:kernel}]
\label{aprop:kernel}
For each $n\geq 2$, $g\geq 1$, considering the regular representation $V = \Z[\Heis]$ of the discrete Heisenberg group $\Heis = \Heis(\Sigma)$, the kernel of the representation constructed in Theorem \ref{athm:twisted} is contained in $\mathfrak{J}(n) \cap \mathrm{Mag}(\Sigma)$.
\end{aproposition}

\paragraph{Computability.}

We emphasise that our representations are explicit and computable. First, the underlying $(R,S)$-bimodule in Theorem \ref{athm:twisted} is a direct sum of finitely many copies of the $(R,S)$-bimodule $V$ that underlies the chosen representation of the discrete Heisenberg group $\Heis(\Sigma)$. This is Theorem \ref{athm:twisted}(a); an explicit basis is described in Theorem \ref{basis}.

Moreover, the actions of elements of the mapping class group on the canonical basis provided by Theorem \ref{basis} may be explicitly computed. To demonstrate this, we calculate in \S\ref{s:computations} explicit matrices for our representations in the case when $n=2$ and $V = \Z[\Heis]$ is the regular representation of $\Heis = \Heis(\Sigma)$. For example, when $g=1$, the Dehn twist around the boundary of $\Sigma_{1,1}$ acts by the $3 \times 3$ matrix over $\Z[\Heis] = \Z[u^{\pm 1}]\langle a^{\pm 1},b^{\pm 1} \rangle / (ab=u^2ba)$ depicted in Figure \ref{fig:bigmatrix} (page \pageref{fig:bigmatrix}).

\paragraph{Outline.}

In \S\ref{s:Heisenberg-system} we define and study the quotient $\Heis$ of the surface braid group. In \S\ref{s:Heisenberg-homology} we study the Borel--Moore homology with local coefficients of configuration spaces on $\Sigma$, proving Theorem \ref{athm:twisted}(a) and showing in particular that, with coefficients in $V=\Z[\Heis]$, it is a free module with an explicit free generating set. Next, in \S\ref{s:action}, we show that the action of the mapping class group on the surface braid group descends to the Heisenberg quotient $\Heis$.

In \S\ref{rep-MCG} we construct twisted representations (Theorem \ref{athm:twisted}(b)) of the full mapping class group, as well as the \emph{un}twisted representations associated to the linearised translation action $L = \Heis \oplus \Z$ (Theorem \ref{athm:l-regular}).
In \S\ref{Schroedinger} we prove Theorems \ref{athm:universal-central} and \ref{athm:universal-central-fd} for the Schr{\"o}dinger representation of $\Heis$ and its finite-dimensional analogues. In \S\ref{relation-to-Moriyama} we discuss connections with the Moriyama and Magnus representations of mapping class groups and deduce that the kernels of our twisted representations of $\mathfrak{M}(\Sigma)$ from Theorem \ref{athm:twisted}, with coefficients in $V = \Z[\Heis]$, are contained in the intersection of the Johnson filtration with the Magnus kernel.

In \S\ref{s:computations} we explain how to compute explicit matrices for our representations with respect to the free basis coming from \S\ref{s:Heisenberg-homology}. We carry out this computation in the case of configurations of $n=2$ points and where $V = \Z[\Heis]$ is the regular representation of $\Heis$; this special case of our construction is a direct analogue of the Lawrence--Krammer--Bigelow representations of the braid groups.

The first version of this paper also contained further results about untwisted representations of subgroups of the mapping class group on Heisenberg homology. In order to improve readability, we have moved this part to a separate article \cite{BPS2}.

\clearpage
\tableofcontents

\section{A non-commutative local system on configuration spaces of surfaces}
\label{s:Heisenberg-system}

Let $\Sigma=\Sigma_{g,1}$ be a compact, connected, orientable surface of genus $g\geq 1$ with one boundary component. For $n\geq 2$, the $n$-point unordered configuration space of $\Sigma$ is
\[
\mathcal{C}_{n}(\Sigma )= \{ \{c_{1},c_{2},...,c_{n}\} \subset \Sigma \mid c_i\neq c_j\text{ for }i\neq j\},
\]
topologised as a quotient of a subspace of $\Sigma^n$. The surface braid group $\mathbb{B}_{n}(\Sigma)$ is then defined as $\mathbb{B}_{n}(\Sigma)=\pi_{1}(\mathcal{C}_{n}(\Sigma))$. We will use the presentation of this group given by Bellingeri and Godelle \cite{BellingeriGodelle2007}, which in turn follows from Bellingeri's presentation \cite{Bellingeri}. We fix based loops, $\alpha_1,\dots,\alpha_g,\beta_1,\dots,\beta_g$ on $\Sigma$, as depicted in Figure \ref{modelSurf}. The basepoint $*_1$ on $\Sigma$ belongs to the base configuration $*$ in $\mathcal{C}_{n}(\Sigma)$. We use the same notation $\alpha_r$, $\beta_r$ for the $\pi_1$-type generators of $\mathbb{B}_{n}(\Sigma)$, which are loops in $\mathcal{C}_{n}(\Sigma)$ where only the first point moves.
\begin{figure}
\centering
\includegraphics[scale=0.7]{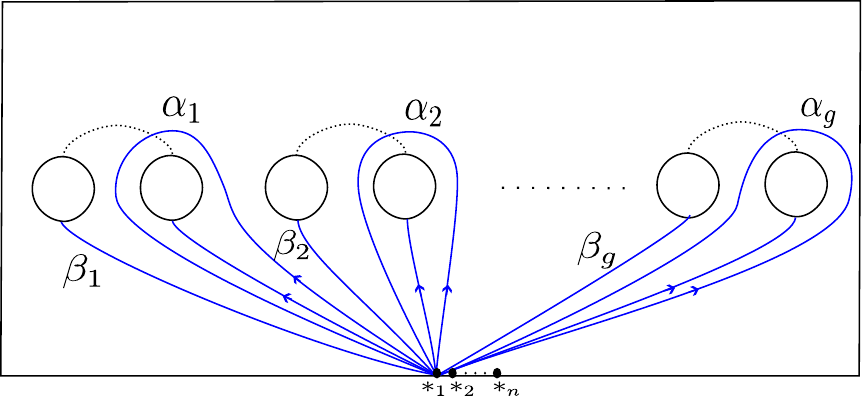}
\caption{Model surface; inner circles are identified in pairs according to the dotted arcs.}
\label{modelSurf}
\end{figure}
The braid group $\mathbb{B}_{n}(\Sigma)$ has generators $\alpha_1,\ldots,\alpha_g$, $\beta_1,\ldots,\beta_g$, together with the classical braid generators $\sigma_1,\ldots,\sigma_{n-1}$ obtained from embedding a disc around the base configuration, and relations:
\[
\begin{cases}
\,\text{(\textbf{BR1}) }\, [\sigma_{i},\sigma_{j}] = 1 & \text{for } \lvert i-j \rvert \geq 2, \\
\,\text{(\textbf{BR2}) }\, \sigma_{i}\sigma_{j}\sigma_{i}=\sigma_{j}\sigma_{i}\sigma_{j} & \text{for } \lvert i-j \rvert = 1, \\
\,\text{(\textbf{CR1}) }\, [\alpha_{r},\sigma_{i}] = [\beta_{r},\sigma_{i}] = 1 & \text{for } i>1 \text{ and all } r, \\
\,\text{(\textbf{CR2}) }\, [\alpha_{r},\sigma_{1}\alpha_{r}\sigma_{1}] = [\beta_{r},\sigma_{1}\beta_{r}\sigma_{1}] = 1 & \text{for all } r, \\
\,\text{(\textbf{CR3}) }\, [\alpha_{r},\sigma^{-1}_{1}\alpha_{s}\sigma_{1}] = [\alpha_{r},\sigma^{-1}_{1}\beta_{s}\sigma_{1}] = & \\
\qquad\qquad\qquad = [\beta_{r},\sigma^{-1}_{1}\alpha_{s}\sigma_{1}] = [\beta_{r},\sigma^{-1}_{1}\beta_{s}\sigma_{1}] = 1 & \text{for all } r<s, \\
\,\text{(\textbf{SCR}) }\, \sigma_{1}\beta_{r}\sigma_{1}\alpha_{r}\sigma_{1}=\alpha_{r}\sigma_{1}\beta_{r} & \text{for all } r.
\end{cases}
\]
We note that composition of loops is written from right to left. Our relation (\textbf{CR3}) is a slight modification of the relation (\textbf{CR3}) of \cite{BellingeriGodelle2007}, but it is equivalent to it via the relation (\textbf{CR2}).

The first homology group $H_1(\Sigma)= H_1(\Sigma;\Z)$ is equipped with a symplectic intersection form $H_1(\Sigma)\times H_1(\Sigma) \to \Z$, denoted by $x.y$, and the \emph{Heisenberg group} $\Heis=\Heis(\Sigma)$ is defined to be the central extension of $H_1(\Sigma)$ determined by this intersection $2$-cocycle. Concretely, it is the set-theoretic product $\Z \times H_{1}(\Sigma)$ with the operation
\begin{equation}
\label{eq:Heisenberg-product}
(k,x)(l,y)=(k+l+x.y,x+y).
\end{equation}
Denote by $\psi \colon \Heis \twoheadrightarrow H_1(\Sigma)$ the projection onto the second factor and by $i \colon \Z \hookrightarrow \Heis$ the inclusion of the first factor; the central extension may then be written as:
\[
\begin{tikzcd}
0 \ar[r] & \Z \ar[r,"i"] & \Heis \ar[r,"\psi"] & H_1(\Sigma) \ar[r] & 0
\end{tikzcd}
\]

\begin{remark}
\label{rmk:Heisenberg-matrices}
The Heisenberg group $\Heis$ may be realised as a group of matrices, which gives a faithful finite-dimensional representation, defined as follows:
\[
\left( k,x=\sum_{i=1}^gp_ia_i+q_ib_i \right) \longmapsto \left(\begin{array}{ccc}
1&p&\frac{k+p\cdot q}{2}\\
0&I_g&q\\
0&0&1
\end{array}\right)\ ,
\]
where $p=(p_i)$ is a row vector and $q=(q_i)$ is a column vector. This matrix form is often given as the definition of the Heisenberg group; we therefore refer to this representation of $\Heis$ as its \emph{tautological representation}.
\end{remark}

There is a general recipe for computing a presentation of an extension of two groups, given presentations of these two groups and some information about the structure of the extension (we will use the formulation of \cite[Appendix B]{DPS}; an alternative reference is \cite[\S 2.4.3]{HBE}). In particular, for a \emph{central} extension $1 \to H \to G \to K \to 1$ with $H = \langle X|R \rangle$ and $K = \langle Y|S \rangle$, we have $G = \langle X \sqcup Y | R \sqcup \widetilde{S} \sqcup T \rangle$, where $\widetilde{S}$ is any collection of relations that are true in $G$ and that project to the relations $S$ in $K$ and where $T$ is a collection of relations saying that the generators $X$ are central in $G$.

Applying this to our setting, we obtain the following presentation of $\Heis$, where we write $u = (1,0)$ and where $a_1,\ldots,a_g$, $b_1,\ldots,b_g$ is a symplectic basis of $H_1(\Sigma)$.

\begin{proposition}
\label{p:presentation}
The Heisenberg group $\Heis = \Heis(\Sigma)$ admits a presentation with generators $u$, $\tilde{a}_i = (0,a_i)$, $\tilde{b}_i = (0,b_i)$ for $1\leq i\leq g$ and relations:
\begin{equation}
\label{eq:relations}
\begin{cases}
\,\text{all pairs of generators commute, except:} \\
\,\tilde{a}_i \tilde{b}_i = u^2 \tilde{b}_i \tilde{a}_i \qquad \text{for each } i.
\end{cases}
\end{equation}
\end{proposition}
\begin{proof}
We apply the above procedure to the presentations $\Z = \langle X|R \rangle$ and $H_1(\Sigma) = \langle Y|S \rangle$ where $X = \{u\}$, $Y = \{\tilde{a}_1,\ldots,\tilde{a}_g,\tilde{b}_1,\ldots,\tilde{b}_g\}$, the relations $R$ are empty and the relations $S$ say that all pairs of elements of $Y$ commute. The relations $T$ say that $u$ commutes with each of $\{\tilde{a}_1,\ldots,\tilde{a}_g,\tilde{b}_1,\ldots,\tilde{b}_g\}$, so to show that \eqref{eq:relations} is a correct presentation of $\Heis$ it will suffice to show that the relations $\tilde{a}_i \tilde{b}_i = u^2 \tilde{b}_i \tilde{a}_i$ and $\tilde{a}_i \tilde{b}_j = \tilde{b}_j \tilde{a}_i$ for $i\neq j$ are true in $\Heis$, because we may then take $\widetilde{S}$ to be this collection of relations, since it projects to $S$. To verify these, we compute that
\[
\tilde{a}_i \tilde{b}_j = (0,a_i+b_j) = (0,b_j+a_i) = \tilde{b}_i \tilde{a}_j
\]
since $a_i.b_j =0$ when $i\neq j$, and
\[
\tilde{a}_i \tilde{b}_i = (1,a_i+b_i) = (1,b_i+a_i) = (2,0)(-1,b_i+a_i) = u^2 \tilde{b}_i \tilde{a}_i,
\]
since $a_i.b_i = 1$ and $b_i.a_i = -1$.
\end{proof}

It follows immediately from this presentation that:

\begin{corollary}
\label{hom_phi}
For each $g\geq 1$ and $n\geq 2$, there is a surjective homomorphism
\[
\phi \colon \mathbb{B}_{n}(\Sigma) \relbar\joinrel\twoheadrightarrow \Heis(\Sigma)
\]
sending each $\sigma_i$ to $u$ and sending $\alpha_i \mapsto \tilde{a}_i$, $\beta_i \mapsto \tilde{b}_i$.
\end{corollary}

In the case $n\geq 3$, this quotient of the surface braid group has previously been considered in \cite{B_al2008,B_al2011,BGG2017}, which also consider the more general setting where $\Sigma$ is closed or has several boundary components. The alternative approach in these articles allows one to identify the kernel of $\phi$ as a characteristic subgroup. We include below a description of the kernel valid for all $n\geq 2$.

\begin{proposition}
\label{kernel}
\textup{(a)} For $n\geq 2$, the kernel of $\phi$ is the normal subgroup generated by the commutators $[\sigma_1,x]$ for $x\in \mathbb{B}_{n}(\Sigma)$.\\
\textup{(b)} For $n\geq 3$, the kernel of $\phi$ is the subgroup of $3$-commutators $\Gamma_3(\mathbb{B}_{n}(\Sigma))$.
\end{proposition}

For a proof of statement (b), we refer to \cite[Theorem 2]{B_al2008}. More precisely, statement (10) on page 1416 of \cite{B_al2008} is the analogous fact for the closed surface $\Sigma_g$: that there is a surjective homomorphism $\mathbb{B}_n(\Sigma_g) \twoheadrightarrow \Heis_g / \langle u^{2(n+g-1)} \rangle$ whose kernel is exactly $\Gamma_3(\mathbb{B}_n(\Sigma_g))$. The proof given there works also in our case where the surface has one boundary component and we do not quotient by $\langle u^{2(n+g-1)} \rangle$. In this paper we will use statement (a) and focus on the case $n=2$ in our explicit computations.

\begin{proof}
Let $K_n \subseteq \mathbb{B}_{n}(\Sigma)$ be the normal subgroup generated by the commutators $[\sigma_1,x]$ for $x\in \mathbb{B}_{n}(\Sigma)$. The image $\phi(\sigma_1)$ being central, we have $K_n \subseteq \ker(\phi)$, hence we see that $\phi$ may be factored through a surjective  homomorphism $\overline \phi \colon \mathbb{B}_{n}(\Sigma)/K_n \rightarrow \Heis$. If we add centrality of $\sigma_1$ to the defining relations for $\mathbb{B}_{n}(\Sigma)$, we may:
\begin{compactitem}
\item replace (\textbf{BR2}) by $\sigma_i=\sigma_1$ for all $i$,
\item remove (\textbf{BR1}), (\textbf{CR1}) and (\textbf{CR2}),
\item replace (\textbf{CR3}) by commutators of all pairs of generators except for $(\alpha_r,\beta_r)$,
\item replace (\textbf{SCR}) with $\alpha_r\beta_r=\sigma_1^2\beta_r\alpha_r$.
\end{compactitem}
Finally the presentations of $\mathbb{B}_{n}(\Sigma)/K_n$ and $\Heis$ coincide and $\overline \phi$ is an isomorphism, which proves (a).
\end{proof}

In contrast to the case of $n \geq 3$, the kernel $\mathrm{ker}(\phi)$ when $n=2$ lies strictly between the terms $\Gamma_2$ and $\Gamma_3$ of the lower central series of $\mathbb{B}_2(\Sigma)$.

\begin{proposition}
There are proper inclusions
\[
\Gamma_3(\mathbb{B}_2(\Sigma)) \hookrightarrow \mathrm{ker}(\phi) \hookrightarrow \Gamma_2(\mathbb{B}_2(\Sigma)).
\]
\end{proposition}
\begin{proof}
By the above proposition, $\mathrm{ker}(\phi)$ is normally generated by commutators, so it must lie inside $\Gamma_2(\mathbb{B}_2(\Sigma))$. On the other hand, the Heisenberg group $\Heis = \Heis_g$ is a central extension of an abelian group, hence $2$-nilpotent. The kernel of any homomorphism $G \to H$ with target a $2$-nilpotent group contains $\Gamma_3(G)$, so $\mathrm{ker}(\phi)$ contains $\Gamma_3(\mathbb{B}_2(\Sigma))$. To see that $\mathrm{ker}(\phi)$ is not equal to $\Gamma_2$, it suffices to note that the Heisenberg group is not abelian. To see that $\mathrm{ker}(\phi)$ is not equal to $\Gamma_3$, we will construct a quotient
\[
\psi \colon \mathbb{B}_2(\Sigma) \longrightarrow Q
\]
where $Q$ is $2$-nilpotent and $[\sigma_1,\alpha_1] \not\in \mathrm{ker}(\psi)$. Given this for the moment, suppose for a contradiction that $\mathrm{ker}(\phi) = \Gamma_3$. Then we have $[\sigma_1,\alpha_1] \in \mathrm{ker}(\phi) = \Gamma_3 \subseteq \mathrm{ker}(\psi)$, due to the fact that $Q$ is $2$-nilpotent, which is a contradiction.

It therefore remains to show that there exists a quotient $Q$ with the claimed properties. We will take $Q = D_4 = \langle \tau,\tau' \mid \tau^2 = (\tau')^2 = (\tau\tau')^4 = 1 \rangle$, the dihedral group with $8$ elements. Let us set $\psi(\alpha_i) = \psi(\beta_i) = \tau'$ and $\psi(\sigma_1) = \tau$. It is easy to verify from the presentations that this is a well-defined surjective homomorphism. The dihedral group $D_4$ is $2$-nilpotent (its centre is generated by $(\tau\tau')^2$ and the quotient by this element is isomorphic to the abelian group $(\Z/2)^2$), and we compute that $\psi([\sigma_1,\alpha_1]) = (\tau \tau')^2 \neq 1$, which completes the proof.
\end{proof}

\section{Heisenberg homology}
\label{s:Heisenberg-homology}

Using the homomorphism $\phi$, any left representation $V$ of the Heisenberg group $\Heis$ over a ring $R$ becomes a left module over $R[\mathbb{B}_{n}(\Sigma)]$. If $V$ also has a right module structure over another ring $S$, i.e.\ if it is a left representation of $\Heis$ in $(R,S)$-bimodules, then it becomes an $(R[\mathbb{B}_{n}(\Sigma)],S)$-bimodule. Following for example \cite[Ch.~3.H]{Hatcher} or \cite[Ch.~5]{Davis} we then have homology groups with local coefficients $H_*(\mathcal{C}_{n}(\Sigma);V)$, which are again $(R,S)$-bimodules.
Let $\widetilde{\mathcal{C}}_{n}(\Sigma)$ be the regular covering of $\mathcal{C}_{n}(\Sigma)$ associated with the kernel of $\phi$. When $V$ is the regular representation $R[\Heis]$, the homology $H_*(\mathcal{C}_{n}(\Sigma);R[\Heis])$ is the homology of the singular chain complex $\mathcal{S}_*(\widetilde{\mathcal{C}}_{n}(\Sigma))$ considered as a right $R[\Heis]$-module by deck transformations. In general, given any $V$ as above, the $(R,S)$-bimodule $H_*(\mathcal{C}_{n}(\Sigma);V)$ is the homology of the chain complex $\mathcal{S}_*(\widetilde{\mathcal{C}}_{n}(\Sigma)) \otimes_{R[\Heis]} V$.

Relative homology with local coefficients is defined in the usual way. We also use \textit{Borel--Moore homology}, defined by
\begin{equation}
\label{eq:BM-homology-limit-definition}
H_n^{BM}(\mathcal{C}_{n}(\Sigma);V) = 
{\varprojlim_T}\, H_n(\mathcal{C}_{n}(\Sigma), \mathcal{C}_{n}(\Sigma) \setminus T ; V),
\end{equation}
where the inverse limit is taken over all compact subsets  $T\subset\mathcal{C}_{n}(\Sigma)$. In general, writing $\mathcal{K}(X)$ for the poset of compact subsets of a space $X$, the Borel--Moore homology module $H_n^{BM}(X,A;V)$ is the limit of the functor $H_n(X,A \cup (X\setminus -);V) \colon \mathcal{K}(X)^{\mathrm{op}} \to \bimod{R}{S}$ for any local system $V$ on $X$ and any properly embedded subspace $A \subseteq X$.

All of the properties concerning Borel--Moore homology that will be used here may be checked by elementary arguments. The interested reader will also find a general exposition based on cosheaves in \cite[Chapter 5]{Bredon}, which includes the case of local coefficients.
One could also work with locally finite singular homology. From \cite[Theorem 7.3]{Spanier} this gives homology groups isomorphic to the inverse limit \eqref{eq:BM-homology-limit-definition} provided that there exists an exhausting sequence of compact subsets $T$ for which the $\lim^1$ contribution vanishes. This is satisfied in our case. Indeed, the configuration space $\mathcal{C}_n(\Sigma)$ is the complement of the big diagonal in the symmetric power $\mathrm{Sym}^n(\Sigma)$. By removing an open tubular neighbourhood of the big diagonal in $\mathrm{Sym}^n(\Sigma)$ we obtain a manifold with boundary that is a compactification of $\mathcal{C}_n(\Sigma)$. This shows that the limit process is stationary with limit the homology of the compactification relative to its boundary.

Borel--Moore homology is functorial with respect to proper maps: If $f \colon Y \to X$ is a proper map taking $B \subseteq Y$ into $A \subseteq X$, then there is an induced functor $f^{-1} \colon \mathcal{K}(X) \to \mathcal{K}(Y)$ by taking pre-images, and a natural transformation $H_n(Y,B \cup (Y\setminus -);f^*(V)) \circ f^{-1} \Rightarrow H_n(X,A \cup (X\setminus -);V)$ (where $f^*(V)$ denotes the pullback of the local system $V$ on $X$ to $Y$) arising from the naturality of singular homology. Taking limits, these induce maps
\begin{align*}
H_n^{BM}(Y,B;f^*(V)) &= \mathrm{lim}\, H_n(Y,B \cup (Y\setminus -);f^*(V)) \\
&\longrightarrow \mathrm{lim}\, \bigl( H_n(Y,B \cup (Y\setminus -);f^*(V)) \circ f^{-1} \bigr) \\
&\longrightarrow \mathrm{lim}\, H_n(X,A \cup (X\setminus -);V) = H_n^{BM}(X,A;V).
\end{align*}
In particular, homeomorphisms are proper maps, so self-homeomorphisms of a space act on its Borel--Moore homology.

We will adapt a method used by Bigelow in the genus-zero case \cite{Bigelow_Hecke} (see also \cite{An,Martel,AP}) for computing the relative Borel--Moore homology
\[
H_*^{BM}(\mathcal{C}_{n}(\Sigma),\mathcal{C}_{n}(\Sigma,\partial^-(\Sigma)) ;V) = {\varprojlim_T}(\mathcal{C}_{n}(\Sigma),\mathcal{C}_{n}(\Sigma,\partial^-(\Sigma)) \cup (\mathcal{C}_{n}(\Sigma)\setminus T) ; V),
\]
where $\mathcal{C}_{n}(\Sigma,\partial^-(\Sigma))$ is the closed (thus properly embedded) subspace of configurations containing at least one point in a fixed closed interval $\partial^-(\Sigma)\subset \partial \Sigma$. In general for a pair $(X,Y)$ the notation $\mathcal{C}_{n}(X,Y)$ will be used for configurations of $n$ points in $X$ containing at least one point in $Y$.

The surface $\Sigma$ can be represented as a thickened interval $[0,1]\times I$ with $2g$ handles, whose cores are attached along $\{1\}\times \{ w_1,w_2,w'_1,w'_2,\dots,w_{2g-1},w_{2g},w'_{2g-1},w'_{2g} \}$ as depicted in Figure~\ref{fig:model-surface}. We view $\Sigma$ as a relative cobordism from $\partial^-(\Sigma)=\{0\}\times I$ (in blue below) to  $\partial^+(\Sigma)$ (in green below), where $\partial^+(\Sigma)$ is the closure of the complement of $\partial^-(\Sigma)$ in $\partial(\Sigma)$.
For $1\leq i\leq 2g$, we denote by $\gamma_i$ the union of the core of the $i$-th handle with $[0,1]\times \{w_i,w'_i\}$, oriented from $w_i$ to $w'_i$, and we set $\Gamma = \amalg_i \gamma_i$ (in red in Figure~\ref{fig:model-surface}).

\begin{figure}[h]
    \centering
    \includegraphics[scale=0.7]{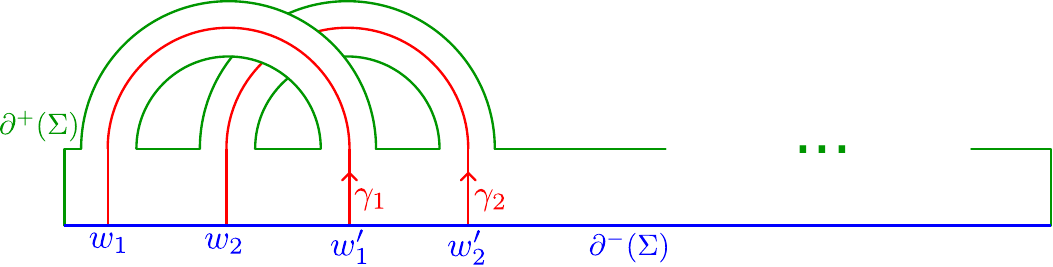}
    \caption{The surface $\Sigma$ together with the decomposition of its boundary into $\partial^+(\Sigma)$ and $\partial^-(\Sigma)$.}
    \label{fig:model-surface}
\end{figure}

Let $\mathcal{K}$ be the set of sequences $k=(k_1,k_2,...,k_{2g})$ such that $k_i$ is a non-negative integer and $\sum_i k_i = n$. 
We will associate to each $k \in \mathcal{K}$ an element of the relative Borel--Moore homology $H_n^{BM}(\mathcal{C}_{n}(\Sigma),\mathcal{C}_{n}(\Sigma,\partial^-(\Sigma)) ;V)$, as follows.

For $k\in \mathcal{K}$ we consider the submanifold $E_k \subset \mathcal{C}_{n}(\Sigma)$ consisting of all configurations having $k_i$ points on $\gamma_i$. This manifold inherits an \emph{orientation} from the orientations of the arcs $\gamma_i$ together with the ordering of the points on $\Gamma$ defined by declaring that $x<y$ for $x \in \gamma_i$, $y \in \gamma_j$ if either $i<j$ or $i=j$ and $x$ comes before $y$ according to the orientation of $\gamma_i$. Moreover, it is a properly embedded Euclidean half-space $\R^n_+$ in $\mathcal{C}_n(\Sigma)$ with boundary in $\mathcal{C}_n(\Sigma,\partial^-(\Sigma))$.
After choosing a path connecting it to the basepoint in $\mathcal{C}_n(\Sigma)$, $E_k$ represents a homology class in $H_n^{BM}(\mathcal{C}_{n}(\Sigma),\mathcal{C}_{n}(\Sigma,\partial^-(\Sigma));V)$, which we also denote by $E_k$.

\begin{theorem}[Theorem \ref{athm:twisted}(a)]
\label{basis}
Let $V$ be any left representation of the discrete Heisenberg group $\Heis$ in $(R,S)$-bimodules, for two rings $R,S$. Then, for $n\geq 2$, there is an isomorphism of $(R,S)$-bimodules
\begin{equation}
\label{eq:module-decomposition}
H_n^{BM}(\mathcal{C}_{n}(\Sigma),\mathcal{C}_{n}(\Sigma,\partial^-(\Sigma)) ;V) \;\cong\; \bigoplus_{k\in \mathcal{K}} V.
\end{equation}
Furthermore, this is the only non-zero bimodule in the graded bimodule $H_*^{BM}(\mathcal{C}_{n}(\Sigma),\mathcal{C}_{n}(\Sigma,\partial^-(\Sigma)) ;V)$.
In particular, in the case when $(R,S) = (\Z,\Z[\Heis])$ and $V = \Z[\Heis]$, the graded right $\Z[\Heis]$-module $H_*^{BM}(\mathcal{C}_{n}(\Sigma),\mathcal{C}_{n}(\Sigma,\partial^-(\Sigma)) ;\Z[\Heis])$ is concentrated in degree $n$ and free of dimension
$\bigl(
\begin{smallmatrix}
2g+n-1 \\
n \\
\end{smallmatrix}
\bigr)$
with basis $\{ E_k \}_{k\in \mathcal{K}}$.
The corresponding statements are also true if Borel--Moore homology $H_*^{BM}$ is replaced by compactly-supported cohomology $H^*_c$ and $V$ is a right representation of $\Heis$ in $(R,S)$-bimodules.
\end{theorem}

\begin{remark}
\label{rmk:more-general-coefficients}
Theorem \ref{basis} is true (with the same proof) more generally for Borel--Moore homology (or compactly-supported cohomology) with coefficients in any representation $V$ of the surface braid group $\mathbb{B}_n(\Sigma) = \pi_1(\mathcal{C}_n(\Sigma))$, not necessarily factoring through the quotient $\mathbb{B}_n(\Sigma) \twoheadrightarrow \Heis$. However, we will only need Theorem \ref{basis} for representations of the Heisenberg group.
\end{remark}

The isomorphism \eqref{eq:module-decomposition} of Theorem \ref{basis} is natural in $V$ in the following sense.

\begin{proposition}
\label{prop:basis-naturality}
The decomposition \eqref{eq:module-decomposition} is natural in the following sense: for any morphism $\xi \colon V \to V'$ of left representations of $\Heis$ over a pair of rings $(R,S)$, the induced map on homology
\[
H_n^{BM}(\mathcal{C}_{n}(\Sigma),\mathcal{C}_{n}(\Sigma,\partial^-(\Sigma)) ;V) \longrightarrow H_n^{BM}(\mathcal{C}_{n}(\Sigma),\mathcal{C}_{n}(\Sigma,\partial^-(\Sigma)) ;V') ,
\]
under the identifications \eqref{eq:module-decomposition}, is equal to $\bigoplus_{k\in \mathcal{K}} \xi$. Thus we have an isomorphism
\[
H_n^{BM}(\mathcal{C}_{n}(\Sigma),\mathcal{C}_{n}(\Sigma,\partial^-(\Sigma));-) \;\cong\; (-)^{\oplus \mathcal{K}}
\]
of functors $\bimod{R[\Heis]}{S} \to \bimod{R}{S}$ for any pair of rings $(R,S)$.
\end{proposition}

In order to prove Theorem \ref{basis} (and Proposition \ref{prop:basis-naturality}), we need a preliminary lemma. To state it, we recall that a deformation retraction $h \colon [0,1] \times \Sigma \rightarrow \Sigma $ from $\Sigma$ to $Y\subset \Sigma$ is a continuous map $(t,x) \mapsto h(t,x)=h_t (x)$ such that $h_0 = \id_\Sigma$, $h_1(\Sigma) = Y$ and $(h_t)_{|_Y} = \id_Y$ for all $0\leq t\leq 1$.

\begin{lemma}\label{deformation}
There exists a metric $d$ on $\Sigma$, inducing the standard topology, and a deformation retraction $h$ from  $\Sigma$ to $\Gamma\cup \partial^-(\Sigma)$, such that for all  $0 \leq t < 1$, the map $h_t \colon \Sigma \to \Sigma$ is a $1$-Lipschitz embedding.
\end{lemma}
\begin{proof}
We have a model for $(\Sigma,\Gamma)$ by gluing $2g$ bands $b_j=[-1,1] \times [-l,l]$, $1\leq j\leq 2g$ and $4g+1$ squares $c_\nu=[0,1]\times [0,1]$, $0\leq j\leq 4g$ according to the identifications depicted in Figure \ref{model_sigma}.
We obtain a  deformation retraction $h$ which is defined on each band by the formula $h_t(u,v) = ((1-t)u,v)$ and on each square by $h_t(u,v)=(u,(1-t)v)$. It remains to show that for an appropriate metric $d$ the map $h_t$, $0\leq t<1$, is a $1$-Lipschitz embedding. 
On each band and square we use the standard Euclidean metric. Then for points $x,y \in \Sigma$, the distance $d(x,y)$ is defined as the shortest length of a path from $x$ to $y$. 
It is convenient to assume that $l$ is big enough so that no shortest path can go across a handle. Then $d$ is a metric which is flat outside $4g$ boundary points where the curvature is concentrated.
We have that $h_t$, for $0\leq t<1$, is a $1$-Lipshitz embedding in each band or square, from which we deduce that $h_t$, for $0\leq t<1$, is globally a $1$-Lipschitz embedding.
\end{proof}

\begin{figure}[t]
\begin{center}
\includegraphics[scale=1]{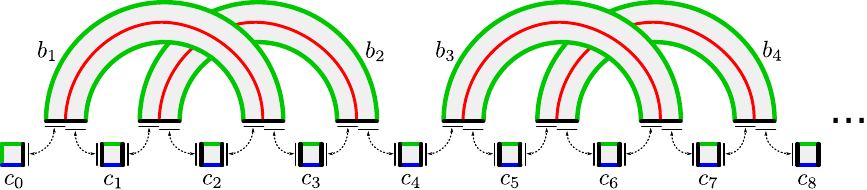}
\end{center}
\caption{A model for $\Sigma$}
\label{model_sigma}
\end{figure}

\begin{proof}[Proof of Theorem \ref{basis}]
We use a metric $d$ and a deformation retraction $h$ from Lemma \ref{deformation}.
For $\epsilon > 0$ and $Y\subset \Sigma$ we denote by $\mathcal{C}_n^\epsilon(Y)$ the subspace of configurations ${x}=\{ x_1,x_2,...,x_n\}\subset Y$ such that $d(x_i , x_j) <\epsilon$ for some $i\ne j$. For $0\leq t\leq 1$, let us write $\Sigma_t=h_t(\Sigma)$. Also, in order to shorten the notation in this proof, we will abbreviate $\mathcal{C}_n^{\epsilon,-}(\Sigma_t) := \mathcal{C}_{n}(\Sigma_t,\partial^-(\Sigma)) \cup \mathcal{C}_n^{\epsilon}(\Sigma_t)$ (in particular when $t=0$, in which case $\Sigma_t = \Sigma$).

For $0\leq t<1$ we have an inclusion
\begin{equation}
\label{eq:inc1}
\left( \mathcal{C}_{n}(\Sigma_t),\mathcal{C}_n^{\epsilon,-}(\Sigma_t) \right) \,\subset\, \left( \mathcal{C}_{n}(\Sigma),\mathcal{C}_n^{\epsilon,-}(\Sigma) \right) ,
\end{equation}
which is a homotopy equivalence with homotopy inverse $\mathcal{C}_n(h_t)$, which is a map of pairs because $h_t$ is $1$-Lipschitz. This implies that we also have a homotopy equivalence of pairs of covering spaces, as follows. Let us write $\pi \colon \widetilde{\mathcal{C}}_{n}(\Sigma) \to \mathcal{C}_{n}(\Sigma)$ for the universal covering of $\mathcal{C}_{n}(\Sigma)$ and denote by $\widetilde{X} := \pi^{-1}(X)$ the corresponding lift of each subspace $X \subseteq \mathcal{C}_{n}(\Sigma)$.\footnote{Note that this is not necessarily the universal covering of $X$.} By taking (homotopy) pullbacks along the covering maps, the homotopy equivalence \eqref{eq:inc1} induces a homotopy equivalence
\begin{equation}
\label{eq:heq1}
\left( \widetilde{\mathcal{C}}_{n}(\Sigma_t),\widetilde{\mathcal{C}}_n^{\epsilon,-}(\Sigma_t) \right) \,\subset\, \left( \widetilde{\mathcal{C}}_{n}(\Sigma),\widetilde{\mathcal{C}}_n^{\epsilon,-}(\Sigma) \right) ,
\end{equation}
and therefore a chain homotopy equivalence of the corresponding relative chain complexes
\begin{equation}
\label{eq:chheq1}
\mathcal{S}_* \left( \widetilde{\mathcal{C}}_{n}(\Sigma_t),\widetilde{\mathcal{C}}_n^{\epsilon,-}(\Sigma_t) \right) \,\simeq\, \mathcal{S}_* \left( \widetilde{\mathcal{C}}_{n}(\Sigma),\widetilde{\mathcal{C}}_n^{\epsilon,-}(\Sigma) \right) .
\end{equation}

The compactness of $\Sigma$ ensures that $h_1$ is the uniform limit of $h_t$ as $t\rightarrow 1$, which implies that for each $\epsilon>0$ we may choose $t=t_\epsilon <1$ such that for all $p \in \Sigma$ we have \mbox{$ d(h_t(p) , h_1(p)) < \tfrac{\epsilon}{2}$}.
For such $t$, let $A_t\subset \mathcal{C}_n(\Sigma_t)$ be the subset of configurations $x=\{x_1,\dots,x_n\}\subset \Sigma_t$ such that $(h_1\circ h_t^{-1})(x_i)=(h_1\circ h_t^{-1})(x_j)$ for some $i\neq j$. We have that $A_t$ is closed and (by our definition of $t=t_\epsilon$) contained in the open subset $\mathcal{C}_n^\epsilon(\Sigma_t) \subset \mathcal{C}_n(\Sigma_t)$ and hence in the interior of the subset $\mathcal{C}_n^{\epsilon,-}(\Sigma_t) \subset \mathcal{C}_n(\Sigma_t)$. The excision theorem therefore implies that the inclusion of pairs
\begin{equation}
\label{eq:inc2}
\left( \mathcal{C}_{n}(\Sigma_t) \setminus A_t , \mathcal{C}_n^{\epsilon,-}(\Sigma_t) \setminus A_t \right) \,\subset\, \left( \mathcal{C}_{n}(\Sigma_t) , \mathcal{C}_n^{\epsilon,-}(\Sigma_t) \right)
\end{equation}
induces isomorphisms on homology with any twisted coefficients pulled back from $\mathcal{C}_{n}(\Sigma_t)$. (We recall from \cite[Theorem~5.13]{Davis} that the excision theorem for homology with twisted coefficients may be formulated in exactly the same way as for untwisted coefficients.) In fact, the proof of the excision theorem shows that this isomorphism is a consequence of a stronger property: the inclusion of pairs \eqref{eq:inc2} induces a chain homotopy equivalence of relative chain complexes of pairs of covering spaces pulled back from any covering space of $\mathcal{C}_{n}(\Sigma_t)$. (Cf.\ \cite[Proposition~2.21]{Hatcher}; see also \cite{Heistad1967} for a slightly different formulation of the excision theorem in terms of homotopy equivalences of chain complexes.) In particular, taking this covering space to be $\widetilde{\mathcal{C}}_{n}(\Sigma_t)$, we have a chain homotopy equivalence:
\begin{equation}
\label{eq:chheq2}
\mathcal{S}_* \left( (\mathcal{C}_{n}(\Sigma_t) \setminus A_t)^{\sim} , (\mathcal{C}_n^{\epsilon,-}(\Sigma_t) \setminus A_t)^{\sim} \right) \,\simeq\, \mathcal{S}_* \left( \widetilde{\mathcal{C}}_{n}(\Sigma_t) , \widetilde{\mathcal{C}}_n^{\epsilon,-}(\Sigma_t) \right) ,
\end{equation}
where we have written $(-)^{\sim} = \widetilde{(-)}$ on the left-hand side for typographical reasons.

The map $\mathcal{C}_n(h_1)\circ \mathcal{C}_n(h_t^{-1})$ gives a well-defined map of pairs
\[
\left( \mathcal{C}_{n}(\Sigma_t)\setminus A_t , \mathcal{C}_n^{\epsilon,-}(\Sigma_t) \setminus A_t \right) \, \longrightarrow \, \left( \mathcal{C}_{n}(\Sigma_1) , \mathcal{C}_n^{\epsilon,-}(\Sigma_1) \right) ,
\]
which is a homotopy inverse to the inclusion. Taking (homotopy) pullbacks along covering maps and passing to relative chain complexes, it follows that the inclusion induces a chain homotopy equivalence:
\begin{equation}
\label{eq:chheq3}
\mathcal{S}_* \left( \widetilde{\mathcal{C}}_{n}(\Sigma_1) , \widetilde{\mathcal{C}}_n^{\epsilon,-}(\Sigma_1) \right) \,\simeq\, \mathcal{S}_* \left( (\mathcal{C}_{n}(\Sigma_t) \setminus A_t)^{\sim} , (\mathcal{C}_n^{\epsilon,-}(\Sigma_t) \setminus A_t)^{\sim} \right) .
\end{equation}

Combining the chain homotopy equivalences \eqref{eq:chheq1}, \eqref{eq:chheq2} and \eqref{eq:chheq3}, we deduce that the inclusion of pairs $(\mathcal{C}_{n}(\Sigma_1) , \mathcal{C}_n^{\epsilon,-}(\Sigma_1)) \subset (\mathcal{C}_{n}(\Sigma) , \mathcal{C}_n^{\epsilon,-}(\Sigma))$ induces a chain homotopy equivalence
\begin{equation}
\label{eq:chheqcombined1}
\mathcal{S}_* \left( \widetilde{\mathcal{C}}_{n}(\Sigma_1) , \widetilde{\mathcal{C}}_n^{\epsilon,-}(\Sigma_1) \right) \,\simeq\, \mathcal{S}_* \left( \widetilde{\mathcal{C}}_{n}(\Sigma),\widetilde{\mathcal{C}}_n^{\epsilon,-}(\Sigma) \right) .
\end{equation}
At the level of relative chain complexes, we have therefore ``compressed'' configurations on the surface $\Sigma$ to configurations on the subspace $\Sigma_1=h_1(\Sigma)$; recall that this is equal to $\Gamma \cup \partial^-(\Sigma)$. The next step is to compress further to configurations on $\Gamma$. In the following, we will use the abbreviation $\mathcal{C}_n^{\epsilon,-}(\Gamma) := \mathcal{C}_{n}(\Gamma,W^-) \cup \mathcal{C}_n^{\epsilon}(\Gamma)$, where $W^-$ is defined by
\[
W^- := \{0\} \times \{ w_1,w_2,w'_1,w'_2,\dots,w_{2g-1},w_{2g},w'_{2g-1},w'_{2g} \} \subset \partial^-(\Sigma)
\]
in other words it is the finite set consisting of the $4g$ endpoints of the arcs $\gamma_1,\ldots,\gamma_{2g}$ in Figure~\ref{fig:model-surface} (recall that $\Gamma$ is the disjoint union of these arcs).

Let $U_\epsilon \subset \partial^-(\Sigma)$ be the open subset given by $x\in U_\epsilon \Leftrightarrow d(x,W^-)<\tfrac{\epsilon}{2}$ and define $B_\epsilon \subset \mathcal{C}_n(\Sigma_1)$ to be the subspace of configurations $x=\{x_1,\dots,x_n\} \subset \Sigma_1$ such that either $x_i \in \partial^-(\Sigma) \setminus U_\epsilon$ for some $i$ or there are indices $i\neq j$ such that $x_i$ and $x_j$ lie in the same component of $U_\epsilon$. It is straightforward to see that $B_\epsilon$ is closed in $\mathcal{C}_n(\Sigma_1)$. Moreover, $B_\epsilon$ is also contained in the interior of $\mathcal{C}_n^{\epsilon,-}(\Sigma_1) = \mathcal{C}_{n}(\Sigma_1,\partial^-(\Sigma)) \cup \mathcal{C}_n^{\epsilon}(\Sigma_1)$ because, for any configuration $x=\{x_1,\ldots,x_n\} \in B_\epsilon$ and any other configuration $y=\{y_1,\ldots,y_n\}$ in a sufficiently small neighbourhood of $x$ in $\mathcal{C}_n(\Sigma_1)$:
\begin{itemize}
\item if $x_i \in \partial^-(\Sigma) \setminus U_\epsilon$ for some $i$, then $y_i \in \partial^-(\Sigma)$ and so we have $y \in \mathcal{C}_{n}(\Sigma_1,\partial^-(\Sigma))$;
\item if $x_i$ and $x_j$ lie in the same component of $U_\epsilon$ for some $i\neq j$, then $d(y_i,y_j)<\epsilon$ and so we have $y \in \mathcal{C}_n^{\epsilon}(\Sigma_1)$.
\end{itemize}
Hence we may apply excision (in its formulation with relative chain complexes of pairs of covering spaces) to deduce that the inclusion induces a chain homotopy equivalence:
\begin{equation}
\label{eq:chheq4}
\mathcal{S}_* \left( (\mathcal{C}_{n}(\Sigma_1) \setminus B_\epsilon)^{\sim} , (\mathcal{C}_n^{\epsilon,-}(\Sigma_1) \setminus B_\epsilon)^{\sim} \right) \,\simeq\, \mathcal{S}_* \left( \widetilde{\mathcal{C}}_{n}(\Sigma_1) , \widetilde{\mathcal{C}}_n^{\epsilon,-}(\Sigma_1) \right) .
\end{equation}

Next, since configurations in $\mathcal{C}_{n}(\Sigma_1) \setminus B_\epsilon$ are contained in $\Gamma \cup U_\epsilon$ and no component of $U_\epsilon$ contains more than one configuration point, we may deformation retract $\mathcal{C}_{n}(\Sigma_1) \setminus B_\epsilon$ onto $\mathcal{C}_{n}(\Gamma)$ by contracting each (interval) component of $U_\epsilon$ to its midpoint. The deformation retraction $\Gamma \cup U_\epsilon \simeq \Gamma$ is through $1$-Lipschitz maps and sends $U_\epsilon$ into itself, so the induced deformation retraction $\mathcal{C}_{n}(\Sigma_1) \setminus B_\epsilon \simeq \mathcal{C}_{n}(\Gamma)$ preserves the subspace $\mathcal{C}_n^{\epsilon,-}(\Sigma_1)$. Thus it provides a homotopy inverse for the inclusion of pairs
\begin{equation}
\label{eq:inc5}
\left( \mathcal{C}_{n}(\Gamma) , \mathcal{C}_n^{\epsilon,-}(\Gamma) \right) \,\subset\, \left( \mathcal{C}_{n}(\Sigma_1) \setminus B_\epsilon , \mathcal{C}_n^{\epsilon,-}(\Sigma_1) \setminus B_\epsilon \right) ,
\end{equation}
where we note that $\mathcal{C}_n^{\epsilon,-}(\Gamma) = \mathcal{C}_{n}(\Gamma) \cap \mathcal{C}_n^{\epsilon,-}(\Sigma_1)$. Taking (homotopy) pullbacks along covering maps and passing to relative chain complexes, \eqref{eq:inc5} therefore induces a chain homotopy equivalence:
\begin{equation}
\label{eq:chheq5}
\mathcal{S}_* \left( \widetilde{\mathcal{C}}_{n}(\Gamma) , \widetilde{\mathcal{C}}_n^{\epsilon,-}(\Gamma) \right) \,\simeq\, \mathcal{S}_* \left( (\mathcal{C}_{n}(\Sigma_1) \setminus B_\epsilon)^{\sim} , (\mathcal{C}_n^{\epsilon,-}(\Sigma_1) \setminus B_\epsilon)^{\sim} \right) .
\end{equation}

Combining the chain homotopy equivalences \eqref{eq:chheqcombined1}, \eqref{eq:chheq4} and \eqref{eq:chheq5}, we have shown that the inclusion of pairs $(\mathcal{C}_{n}(\Gamma) , \mathcal{C}_n^{\epsilon,-}(\Gamma)) \subset (\mathcal{C}_{n}(\Sigma) , \mathcal{C}_n^{\epsilon,-}(\Sigma))$ induces a chain homotopy equivalence
\begin{equation}
\label{eq:chheqcombined2}
\mathcal{S}_* \left( \widetilde{\mathcal{C}}_{n}(\Gamma) , \widetilde{\mathcal{C}}_n^{\epsilon,-}(\Gamma) \right) \,\simeq\, \mathcal{S}_* \left( \widetilde{\mathcal{C}}_{n}(\Sigma),\widetilde{\mathcal{C}}_n^{\epsilon,-}(\Sigma) \right) ,
\end{equation}
in other words we have (at the level of relative chain complexes) ``compressed'' configurations on the surface $\Sigma$ to configurations on the disjoint union of arcs $\Gamma$.

The fundamental chain homotopy equivalence \eqref{eq:chheqcombined2} immediately implies isomorphisms both for twisted relative Borel--Moore homology and for twisted relative compactly-supported cohomology. First, we may tensor \eqref{eq:chheqcombined2} over $R[\pi_1(\mathcal{C}_{n}(\Sigma))]$ with $V$ and take homology to obtain an isomorphism of twisted relative homology groups for each $\epsilon > 0$; then taking the inverse limit as $0 \leftarrow \epsilon$, we obtain an isomorphism of twisted relative Borel--Moore homology:
\begin{equation}
\label{eq:combined-isomorphism-BM}
H_{*}^{BM}(\mathcal{C}_{n}(\Gamma),\mathcal{C}_{n}(\Gamma,W^-) ;V) \;\cong\; H_*^{BM}(\mathcal{C}_{n}(\Sigma),\mathcal{C}_{n}(\Sigma,\partial^-(\Sigma)) ;V).
\end{equation}
Here we are using the fact that, if $Y \subset \Sigma$ is closed, then $\mathcal{C}_n^\epsilon(Y)$ is a cofinal family of co-compact subsets of $\mathcal{C}_n(Y)$, which implies that for a pair $(Y,Z)$ of closed subspaces of $\Sigma$, we have
\begin{equation}
\label{eq:BM-as-limit}
H^{BM}_{*}(\mathcal{C}_{n}(Y),\mathcal{C}_n(Y,Z );V) \;\cong\; \lim_{0\leftarrow \epsilon}H_{*}(\mathcal{C}_{n}(Y),\mathcal{C}_n(Y,Z )\cup \mathcal{C}_n^\epsilon(Y);V).
\end{equation}
(As a notational point, we note that we simply write $V$ for the restriction to subspaces of $\mathcal{C}_n(\Sigma)$ of the local system $V$, which is defined a priori on $\mathcal{C}_n(\Sigma)$.)

Alternatively, if $V$ is a right (rather than left) representation of $\Heis$ in $(R,S)$-bimodules, we may apply the operation $\Hom_{S[\pi_1(\mathcal{C}_{n}(\Sigma))]}(-,V)$ to \eqref{eq:chheqcombined2} and take homology to obtain an isomorphism of twisted relative cohomology groups for each $\epsilon > 0$; then taking the direct limit as $\epsilon \to 0$, we obtain an isomorphism of twisted relative compactly-supported cohomology:
\begin{equation}
\label{eq:combined-isomorphism-cc}
H^{*}_{c}(\mathcal{C}_{n}(\Gamma),\mathcal{C}_{n}(\Gamma,W^-) ;V) \;\cong\; H^*_{c}(\mathcal{C}_{n}(\Sigma),\mathcal{C}_{n}(\Sigma,\partial^-(\Sigma)) ;V).
\end{equation}
In each case, to justify taking the limit, we need to know that \eqref{eq:chheqcombined2} is is a chain homotopy equivalence \emph{of inverse systems} as $\epsilon > 0$ varies. However, this is clear since it is induced by the inclusion of (pairs of) configuration spaces $(\mathcal{C}_{n}(\Gamma) , \mathcal{C}_n^{\epsilon,-}(\Gamma)) \subset (\mathcal{C}_{n}(\Sigma) , \mathcal{C}_n^{\epsilon,-}(\Sigma))$.

Finally, to complete the proof of the theorem, we need to calculate the left-hand sides of \eqref{eq:combined-isomorphism-BM} and \eqref{eq:combined-isomorphism-cc}. We will do this in the first case (for Borel--Moore homology); the calculation in the second case (for compactly-supported cohomology) is exactly dual.

We first observe that $\mathcal{C}_n(\Gamma)$ is a disjoint union indexed by $\mathcal{K}$, where each connected component $E_k$ ($k \in \mathcal{K}$) is a product of configuration spaces on the intervals $\gamma_i$. Configurations of at least two points in an interval form a simplex where the diagonal part of the boundary has been removed and the remaining boundary is the union of two faces.
Hence, for each $k\in \mathcal{K}$, the product $E_k$ is a topological ball where part of the boundary has been removed. The disjoint union of the topological boundaries $\partial E_k$ is precisely the subspace $\mathcal{C}_{n}(\Gamma,W^-) \subset \mathcal{C}_{n}(\Gamma)$.

For $\epsilon > 0$, let us consider the subspace $\mathcal{C}_n^\epsilon(\Gamma) = \amalg_{k\in \mathcal{K}} E^\epsilon_k$ of configurations where two points are $\epsilon$-close.
For sufficiently small $\epsilon > 0$, the pair $(E_k,E_k^\epsilon\cup \partial E_k)$ is homotopy equivalent to the pair $(D^n,\partial D^n)$. Using that the complements $\mathcal{C}_n(\Gamma) \setminus \mathcal{C}_n^\epsilon(\Gamma)$ form a family of compact subspaces cofinal to all compact subspaces in $\mathcal{C}_n(\Gamma)$, we deduce the computation of Borel--Moore homology:
For each $k$ we have $H_*^{BM}(E_k,\partial E_k;V)=H_n^{BM}(E_k,\partial E_k;V) \cong V$. Here we use the restriction of the local system $V$ to $E_k$, which is constant (i.e.~trivialisable).
We obtain that the Borel--Moore homology \eqref{eq:combined-isomorphism-BM} is trivial when $* \neq n$ and that each Borel--Moore homology class $E_k$ generates a direct summand isomorphic to the coefficients $V$ in degree $* = n$. In particular, when $V$ is the regular representation $\Z[\Heis]$, these classes form a basis over $\Z[\Heis]$ for the degree-$n$ Borel--Moore homology.
\end{proof}

\begin{remark}
\label{rmk:alternative-argument}
If one is just interested in the version of Theorem \ref{basis} for Borel--Moore homology (and not compactly-supported cohomology), then one could work directly with isomorphisms of twisted relative homology groups at each stage, rather than chain homotopy equivalences of relative chain complexes of pairs of covering spaces. The unified proof that we give above has the advantage that it simultaneously provides explicit bases both for twisted relative Borel--Moore homology and twisted relative compactly-supported cohomology, from which one may easily deduce a perfect pairing between the two; this is discussed further in \S\ref{ss:sesquilinear-form}.

We note that one could also deduce both results (for Borel--Moore homology and for compactly-supported cohomology) from the result for Borel--Moore homology in a specific case (i.e.~with a specific choice of $V$); this is explained in \hyperref[appendixB]{Appendix B}.
\end{remark}

\begin{proof}[Proof of Proposition \ref{prop:basis-naturality}]
The statement of Theorem \ref{basis} in the case $S = V = R[\Heis]$ implies that $H_*^{BM}(\mathcal{C}_{n}(\Sigma),\mathcal{C}_{n}(\Sigma,\partial^-(\Sigma)) ;R[\Heis])$ is free in each degree as a right $R[\Heis]$-module. The universal coefficient theorem\footnote{Compare the proof of Lemma \ref{lem:universal-coefficients} in \hyperref[appendixB]{Appendix B}.} provides $(R,S)$-module isomorphisms
\begin{equation}
\label{eq:UCT-identification}
H_n^{BM}(\mathcal{C}_{n}(\Sigma),\mathcal{C}_{n}(\Sigma,\partial^-(\Sigma)) ;V) \;\cong\; H_n^{BM}(\mathcal{C}_{n}(\Sigma),\mathcal{C}_{n}(\Sigma,\partial^-(\Sigma)) ;R[\Heis]) \otimes_{R[\Heis]} V
\end{equation}
for any $(R[\Heis],S)$-bimodule $V$. Moreover, both sides are functorial in $V$ and \eqref{eq:UCT-identification} is a natural isomorphism between these functors, i.e.\ the map on homology induced by $\xi \colon V \to V'$, under the identification \eqref{eq:UCT-identification}, is of the form $\id \otimes \xi$. The left-hand identity component of $\id \otimes \xi$ decomposes into a direct sum over $k \in \mathcal{K}$ of copies of $\id_{R[\Heis]}$ under the decomposition \eqref{eq:module-decomposition} of Theorem \ref{basis} in the case $S = V = R[\Heis]$. (We note that this is not circular, because here we are only using the tautological fact that the \emph{identity} decomposes as a direct sum of identities.) Since $\otimes$ distributes (naturally) over $\oplus$, we deduce that the map on homology induced by $\xi \colon V \to V'$, under the decomposition \eqref{eq:module-decomposition} of Theorem \ref{basis}, is the direct sum over $k \in \mathcal{K}$ of copies of $\xi$.
\end{proof}

\section{Action of mapping classes}
\label{s:action}

The \emph{mapping class group} of $\Sigma$, denoted by $\mathfrak{M}(\Sigma)$, is the group of orientation-preserving diffeomorphisms of $\Sigma$ fixing the boundary pointwise, modulo isotopies relative to the boundary. The isotopy class of a diffeomorphism $f$ is denoted by $[f]$. An oriented self-diffeomorphism fixing the boundary pointwise $f \colon \Sigma \rightarrow \Sigma$ gives us a homeomorphism $\mathcal{C}_{n}(f) \colon \mathcal{C}_{n}(\Sigma) \rightarrow \mathcal{C}_{n}(\Sigma)$, defined by $\{x_{1},x_{2},\ldots ,x_{n}\} \mapsto \{f(x_{1}),f(x_{2}),\ldots ,f(x_{n})\}$. If we ensure that the basepoint configuration of $\mathcal{C}_n(\Sigma)$ is contained in $\partial\Sigma$, then it is fixed by $\mathcal{C}_n(f)$ and this in turn induces a homomorphism $f_{\mathbb{B}_{n}(\Sigma)} = \pi_{1}(\mathcal{C}_{n}(f)) \colon \mathbb{B}_{n}(\Sigma) \rightarrow \mathbb{B}_{n}(\Sigma)$, which depends only on the isotopy class $[f]$ of $f$.

\subsection{Action on the Heisenberg group}

We first study the induced action on the Heisenberg group quotient.

\begin{proposition}
\label{f_Heisenberg}
There exists a unique homomorphism $f_{\Heis} \colon \Heis \rightarrow \Heis$ such that the following square commutes:
\begin{equation}
\label{eq:projection-equivariance}
  \begin{tikzcd}
     \mathbb{B}_{n}(\Sigma) \arrow[d,swap, "\phi"] \arrow[rr, "f_{\mathbb{B}_{n}(\Sigma)}"] && \mathbb{B}_{n}(\Sigma) \arrow[d,"\phi"]\\
     \Heis \arrow[rr, "f_{\Heis}"] && \Heis
  \end{tikzcd}
\end{equation}
Thus, there is an action of $\mathfrak{M}(\Sigma)$ on the Heisenberg group $\Heis$ given by
\begin{equation}\label{eq:action_on_Heis}
\Psi \colon f \mapsto f_\Heis \colon \mathfrak{M}(\Sigma) \longrightarrow \mathrm{Aut}(\Heis).
\end{equation}
\end{proposition}

\begin{proof}
Since $\phi$ is surjective, the homomorphism $f_\Heis$ will be uniquely determined by the formula $f_{\Heis}(\phi(\gamma))=\phi(f_{\mathbb{B}_{n}(\Sigma)}(\gamma))$ if it exists. To show that it exists, we need to show that the composition $\phi\circ f_{\mathbb{B}_{n}(\Sigma)}$ factors through $\phi$, which is equivalent to saying that $f_{\mathbb{B}_{n}(\Sigma)}$ sends $\mathrm{ker}(\phi)$ into itself.

Recall that the classical generator $\sigma_1$ is represented by a loop of configurations on a disc $D \subset \Sigma$ containing the base configuration. Let $T \subset \Sigma$ be a tubular neighbourhood of $\partial \Sigma$ containing $D$. Since $f$ fixes $\partial \Sigma$ pointwise, we may isotope $f$ so that it is the identity on $T$, in particular on $D$, which implies that $f_{\mathbb{B}_{n}(\Sigma)}$ fixes $\sigma_1$. We then deduce from part (a) of Proposition \ref{kernel} that $f_{\mathbb{B}_{n}(\Sigma)}$ sends $\ker(\phi)$ to itself, which completes the proof.
\end{proof}

\subsection{Structure of automorphisms of the Heisenberg group.}
\label{structure-of-Aut-Heis}

Recall that the centre of the Heisenberg group $\Heis$ is infinite cyclic, generated by the element $u$. Any automorphism of $\Heis$ must therefore send $u$ to $u^{\pm 1}$.

\begin{definition}\label{def:Heis-orientation}
We denote the index-$2$ subgroup of those automorphisms of $\Heis$ that fix $u$ by $\mathrm{Aut}^+(\Heis)$, and call these \emph{orientation-preserving}.
\end{definition}
From the proof of Proposition \ref{f_Heisenberg}, we observe that, for any $f \in \mathfrak{M}(\Sigma)$, the automorphism $f_\Heis$ is orientation-preserving in the sense of Definition \ref{def:Heis-orientation}. We may therefore refine the action $\Psi$  as follows:
\begin{equation}\label{eq:action_on_Heis_orientation}
\Psi \colon f \mapsto f_\Heis \colon \mathfrak{M}(\Sigma) \longrightarrow \mathrm{Aut}^+(\Heis).
\end{equation}

The quotient of $\Heis$ by its centre may be canonically identified with $H = H_1(\Sigma)$, so every automorphism of $\Heis$ induces an automorphism of $H$. Moreover, if it is orientation-preserving, then the induced automorphism of $H$ preserves the symplectic form $.$ on $H$: to see this, apply the automorphism to the equation $(0,x)(0,y)(0,-x)=(2x.y,y)$ in $\Heis$.
Thus we have a homomorphism $\cL \colon \mathrm{Aut}^+(\Heis) \to Sp(H)$ denoted by $\varphi\mapsto \overline \varphi$.

\begin{lemma}
\label{lem:split-ses}
There exists a split short exact sequence
\[
\begin{tikzcd}
1 \ar[r] & H^1(\Sigma;\Z) \ar[r,"j"] & \mathrm{Aut}^+(\Heis) \ar[r,"\cL"] & Sp(H) \ar[r] & 1
\end{tikzcd}
\]
where $j(c)=[(k,x)\mapsto (k+c(x),x)]$.
\end{lemma}
\begin{proof}
We observe that for an automorphism $\varphi \in \mathrm{Aut}^+(\Heis)$ we have $\varphi(k,x)=(k+c(x),\overline{\varphi}(x))$.
By applying $\varphi$ to $(k,x)(l,y)=(k+l+x.y,x+y)$), we deduce that $c$ is a homomorphism. We thus have $c\in \Hom(H_1(\Sigma;\Z),\Z)\cong H^1(\Sigma;\Z)$ .
We see that $j \colon H^1(\Sigma;\Z)\rightarrow \mathrm{Aut}^+(\Heis)$ is a group homomorphism whose image is in $\ker(\cL)$.
We next identify the kernel of $\cL$: an automorphism $\varphi\in \ker(\cL)$ takes the form $\varphi(k,x)=(k+c(x),x)$ where $c\in H^1(\Sigma;\Z)$ and $\varphi=j(c)$. This proves exactness in the middle of the sequence above. Injectivity of $j$ and surjectivity of $\cL$ may also be checked easily. Finally, a splitting of $\cL$ is given by the assignment $g\mapsto \varphi_g=[(k,x)\mapsto (k,g(x))]$.
\end{proof}

As corollary, we obtain that $\mathrm{Aut}^+(\Heis)$ is the \emph{affine symplectic group}. The splitting gives a decomposition as $Sp(H) \ltimes H^1(\Sigma;\Z)$, where the semi-direct product structure on the right-hand side is induced by the natural action of $Sp(H)$.
Corresponding to the splitting given in the proof, there is a function (which is not a group homomorphism) $\mathrm{Aut}^+(\Heis) \to H^1(\Sigma;\Z)\cong\Hom(H,\Z)$ defined by $\varphi\mapsto  \varphi^\diamond $, where $\varphi(0,x)=( \varphi^\diamond(x),\cL(x))$. We formulate the result below.
\begin{corollary}\label{lem:AutHeis}
The homomorphism $\cL \colon \mathrm{Aut}^+(\Heis) \to Sp(H)$ and function $(-)^\diamond \colon \mathrm{Aut}^+(\Heis) \to H^*$ induce an isomorphism
\begin{equation}
\label{eq:Heis_identification}
\mathrm{Aut}^+(\Heis) \cong Sp(H) \ltimes H^1(\Sigma;\Z) , \qquad \varphi \mapsto (\overline{\varphi} , \varphi^\diamond),
\end{equation}
where the semi-direct product structure on the right-hand side is induced by the natural action of $Sp(H)$ on $H^1(\Sigma;\Z)$.
\end{corollary}

\begin{remark}
\label{rmk:linearising}
Fixing a symplectic basis of $H$, the right-hand side of \eqref{eq:Heis_identification} is a subgroup of $GL_{2g}(\Z) \ltimes \Z^{2g}$, which may be embedded into $GL_{2g+1}(\Z)$. In this way, any orientation-preserving action of a group $G$ on $\Heis$ may be viewed as a linear representation of $G$ over $\Z$ of rank $2g+1$.
\end{remark}

The general form of an oriented automorphism $\varphi$ is therefore
\[
\varphi(k,x)=(k+\varphi^\diamond(x),\overline \varphi(x))\ ,
\]
where $\varphi^\diamond\in H^*$ and $\overline \varphi\in Sp(H)$ is the induced symplectic automorphism.
From the proof of Proposition \ref{f_Heisenberg} we observe that, for any $f \in \mathfrak{M}(\Sigma)$, the automorphism $f_\Heis$ is orientation-preserving in the sense of Definition \ref{def:Heis-orientation}. Hence for a mapping class $f\in \mathfrak{M}(\Sigma)$, the map $f_\Heis$ is represented as follows:
\begin{equation}\label{eq:action}
f_\Heis \colon (k,x)\mapsto (k+\delta_f(x),f_*(x)),
\end{equation}
where $\delta_f = (f_\Heis)^\diamond \in H^1(\Sigma;\Z)$.

\subsection{Recovering Morita's crossed homomorphism.}
\label{ss:Morita-crossed-hom}

In \cite{Morita1989}, Morita introduced a crossed homomorphism $\mathfrak{d} \colon \mathfrak{M}(\Sigma)\rightarrow H^1(\Sigma)$, $f\mapsto \mathfrak{d}_f$ representing a generator for $H^1(\mathfrak{M}(\Sigma);H^1(\Sigma))\cong \Z$. We will recover this crossed homomorphism from the action $f \mapsto f_\Heis$ on the Heisenberg group.

Recall that, for a given action of a group $G$ on an abelian group $K$, a \emph{crossed homomorphism} $\theta \colon G \to K$ is a function with the property that $\theta(g_2 g_1) = \theta(g_1) + g_1\theta(g_2)$ for all $g_1,g_2 \in G$.

\begin{remark}
\label{rmk:crossed-hom-lifts}
Crossed homomorphisms $G\rightarrow K$ are in one-to-one correspondence with lifts
\[
\begin{tikzcd}
G \arrow[rr,dashed] \arrow[drr] &&
\mathrm{Aut}(K) \ltimes K \arrow[d,twoheadrightarrow]
\\
&& \mathrm{Aut}(K),
\end{tikzcd}
\]
where the diagonal arrow is the given action of $G$ on $K$.
\end{remark}

\begin{proposition}
\label{p:Morita-crossed-hom}
The map $\delta \colon \mathfrak{M}(\Sigma)\rightarrow H^1(\Sigma)$, $f\mapsto \delta_f$, is a crossed homomorphism equal to Morita's crossed homomorphism $\mathfrak{d}$.
\end{proposition}
\begin{proof}
We first show that $\delta$ is a crossed homomorphism. Let $f,g$ be mapping classes; then we have, for $(k,x)\in \Heis$,
\[
(g\circ f)_\Heis(k,x)=g_\Heis(k+\delta_f(x),f_*(x))=(k+\delta_f(x)+\delta_g(f_*(x)),(g\circ f)_*(x))\ ,
\]
and so we obtain $\delta_{g\circ f}(x) = \delta_f(x)+f^*(\delta_g)(x)$, as required.

Recall that we use the same notation for the (free) generators $\alpha_i$, $\beta_i$, $1\leq i\leq g$, for $\pi_1(\Sigma)$ and the corresponding $\pi_1$ generators of the braid group $\mathbb{B}_n(\Sigma)$.
For $\gamma\in \pi_1(\Sigma)$, let us denote by $\gamma_i$ the element in the free group generated by $\alpha_i$, $\beta_i$ that is the image of $\gamma$ under the homomorphism that maps the other generators to $1$. Then we have a decomposition
\[
\gamma_i=\alpha_i^{\nu_1}\beta_i^{\mu_1}\dots \alpha_i^{\nu_m}\beta_i^{\mu_m}\ ,
\]
where $\nu_j$ and $\mu_j$ are $0$, $-1$ or $1$. The integer $d_i(\gamma)$ is then defined\footnote{There is a small misprint in \cite{Morita1989}.} by
\begin{equation}
\label{eq:definition-of-d}
\begin{aligned}
d_i(\gamma) &= \sum_{j=1}^m\nu_j\sum_{k=j}^m\mu_k-\sum_{j=1}^m\mu_j\sum_{k=j+1}^m\nu_k \\
&= \sum_{j=1}^m \sum_{k=1}^m \iota_{jk} \nu_j \mu_k,
\end{aligned}
\end{equation}
where $\iota_{jk} = +1$ when $j\leq k$ and $\iota_{jk} = -1$ when $j > k$.
The definition for the Morita crossed homomorphism is as follows:
\begin{equation}
\label{eq:Morita-crossed-hom}
\mathfrak{d}_f(\gamma)=\sum_{i=1}^g d_i(\pi_1(f)(\gamma))-d_i(\gamma)\ .
\end{equation}
For $\gamma\in \pi_1(\Sigma)$, consider the pure braid obtained by adding $n-1$ trivial strands to $\gamma$, which we also denote by $\gamma$. The above decomposition of $\gamma$ used for the definition of $d_i$ is also a decomposition in the generators of the braid group, and from the definition of the product in $\Heis$ we have that
\[
\phi(\gamma) = \left(\sum_{i=1}^gd_i(\gamma),[\gamma] \right) \in \Heis\ .
\]
This formula may be checked by recursion on the length of $\gamma$ as a word in the free generators of $\pi_1(\Sigma)$. It can also be deduced from \cite[Lemma 6.1]{Morita1989}.
The equality $\mathfrak{d}_f=\delta_f$ follows.
\end{proof}

\section{Constructing the representations}
\label{rep-MCG}

In this section we construct (\S\ref{ss:twisted-rep}) the twisted representation of Theorem \ref{athm:twisted}, as well as (\S\ref{ss:untwisted-rep-linearised}) the untwisted representation of Theorem \ref{athm:l-regular} associated to the linearised translation action of $\Heis$.

\subsection{A twisted representation of the mapping class group.}
\label{ss:twisted-rep}

The quotient homomorphism $\phi \colon \mathbb{B}_n(\Sigma) \twoheadrightarrow \Heis$ (Corollary \ref{hom_phi}) corresponds to a regular covering $\widetilde{\mathcal{C}}_n(\Sigma) \to \mathcal{C}_n(\Sigma)$. Let $f\in \mathfrak{M}(\Sigma)$ and write $f_\Heis$ for its action on the Heisenberg group $\Heis$ and $\mathcal{C}_n(f)$ for its action on the configuration space $\mathcal{C}_n(\Sigma)$. From Proposition \ref{f_Heisenberg} we know that $\pi_1(\mathcal{C}_n(f))=f_{\mathbb{B}_n(\Sigma)}$ preserves $\ker(\phi)$, which implies that there exists a unique lift of $\mathcal{C}_n(f)$ fixing the basepoint:
\begin{equation}
\label{eq:lifted-action}
\widetilde{\mathcal{C}}_n(f) \colon \widetilde{\mathcal{C}}_n(\Sigma) \longrightarrow \widetilde{\mathcal{C}}_n(\Sigma)\ .
\end{equation}
Following a classical construction in covering spaces \cite{Hatcher}, a model for $\widetilde{\mathcal{C}}_n(\Sigma)$ is given by equivalence classes $[\delta]$ of paths $\delta$ starting at the base configuration in $\mathcal{C}_n(\Sigma)$, with $[\gamma]=[\delta]$ if and only if $\phi(\overline{\gamma}\,\delta)=(0,0)$, where $\overline \gamma$ denotes the inverse path. In this model we have $\widetilde{\mathcal{C}}_n(f)([\delta])=[\mathcal{C}_n(f)\circ \delta]$ and the deck action of $h=\phi([\gamma])$ is $[\delta]\cdot h=[\delta\,\gamma]$. Then we get 
\begin{align*}
\widetilde{\mathcal{C}}_n(f)([\delta]\cdot h) = [\mathcal{C}_n(f)\circ (\delta\,\gamma)] &= [(\mathcal{C}_n(f)\circ \delta)\,(\mathcal{C}_n(f)\circ \gamma)] \\
&= [\mathcal{C}_n(f)\circ \delta] \cdot \phi([\mathcal{C}_n(f)\circ \gamma]) \\
&= \widetilde{\mathcal{C}}_n(f)([\delta])\cdot f_\Heis(h)\ .
\end{align*}
We have therefore proven the formula
\begin{equation}
\label{deck}
\widetilde{\mathcal{C}}_n(f)(x\cdot h)=\widetilde{\mathcal{C}}_n(f)(x)\cdot f_\Heis(h)
\end{equation}
for any $x \in \widetilde{\mathcal{C}}_n(\Sigma)$ and $h \in \Heis$. It follows that the induced action on the singular chain complex $\mathcal{S}_*(\widetilde{\mathcal{C}}_n(\Sigma))$ is twisted $R[\Heis]$-linear, which may be formulated as an $R[\Heis]$-linear isomorphism
\[
\mathcal{S}_*(\widetilde{\mathcal{C}}_n(f)) \colon \mathcal{S}_*(\widetilde{\mathcal{C}}_{n}(\Sigma))_{f_\Heis^{-1}} \longrightarrow  \mathcal{S}_*(\widetilde{\mathcal{C}}_{n}(\Sigma)) \ .
\]
Here the subscript on the domain means that the right action of $\Heis$ is twisted by $f_\Heis^{-1}$. The result for $R[\Heis]$-local homology is an $R[\Heis]$-linear isomorphism
\begin{equation}
\label{eq:twisted-linear-isomorphism}
\mathcal{C}_n(f)_* \colon H_*^{BM}(\mathcal{C}_{n}(\Sigma),\mathcal{C}_{n}(\Sigma,\partial^-(\Sigma)) ;R[\Heis])_{f_\Heis^{-1}} \longrightarrow H_*^{BM}(\mathcal{C}_{n}(\Sigma),\mathcal{C}_{n}(\Sigma,\partial^-(\Sigma)) ;R[\Heis]) \ .
\end{equation}
More generally, if $V$ is a left representation of the Heisenberg group in $(R,S)$-bimodules, then we obtain an $(R,S)$-linear isomorphism
\begin{equation}
\label{eq:twisted-linear-isomorphism-R}
\mathcal{C}_n(f)_* \colon H_*^{BM}\bigl(\mathcal{C}_{n}(\Sigma),\mathcal{C}_{n}(\Sigma,\partial^-(\Sigma)) ;{}_{f_\Heis}\! V\bigr) \longrightarrow H_*^{BM}\bigl(\mathcal{C}_{n}(\Sigma),\mathcal{C}_{n}(\Sigma,\partial^-(\Sigma)) ;V\bigr) \ ,
\end{equation}
where the left-hand homology group is obtained from the chain complex
\begin{equation}
\label{eq:twisting-on-left-or-right}
\Bigl( \mathcal{S}_* \bigl( \widetilde{\mathcal{C}}_n(\Sigma) \bigr)_{\! f_\Heis^{-1}} \Bigr) \, \underset{R[\Heis]}{\otimes} V \;\cong\;  \mathcal{S}_* \bigl( \widetilde{\mathcal{C}}_n(\Sigma) \bigr) \underset{R[\Heis]}{\otimes} \, \Bigl( {}_{f_\Heis}\! V \Bigr)\ .
\end{equation}
Here, ``obtained from'' means that we consider the quotients of this chain complex given by the relative singular complexes for all subspaces of $\widetilde{\mathcal{C}}_n(\Sigma)$ of the form $\pi^{-1}(\mathcal{C}_n(\Sigma , \partial^-(\Sigma)) \cup (\mathcal{C}_n(\Sigma) \smallsetminus T))$ for compact subsets $T \subset \mathcal{C}_n(\Sigma)$, where $\pi$ denotes the covering map $\widetilde{\mathcal{C}}_n(\Sigma) \to \mathcal{C}_n(\Sigma)$; we then take the homology of each of these quotients and take the inverse limit of this diagram.

Another way of describing this construction, and of keeping track of the twisting on each side, is to write the lifted action \eqref{eq:lifted-action} of $f$ as an $\Heis$-equivariant map
\begin{equation}
\label{eq:lifted-action-deck}
\widetilde{\mathcal{C}}_n(\Sigma)^{f_\Heis \circ \phi} \longrightarrow \widetilde{\mathcal{C}}_n(\Sigma)^{\phi},
\end{equation}
where the superscript indicates the quotient $\pi_1(\mathcal{C}_n(\Sigma)) = \mathbb{B}_n(\Sigma) \twoheadrightarrow \Heis$ determining the covering space as a space equipped with a right $\Heis$-action. Applying relative twisted Borel--Moore homology to \eqref{eq:lifted-action-deck}, considered as a map of regular covering spaces, we obtain \eqref{eq:twisted-linear-isomorphism} with $R[\Heis]$-local coefficients and \eqref{eq:twisted-linear-isomorphism-R} with $V$-local coefficients.

We may easily generalise this discussion by twisting both sides by an element $\tau \in \mathrm{Aut}(\Heis)$. The action $\mathcal{C}_n(f) \colon \mathcal{C}_n(\Sigma) \to \mathcal{C}_n(\Sigma)$ lifts to a map of regular covering spaces
\begin{equation}
\label{eq:lifted-action-deck-theta}
\widetilde{\mathcal{C}}_n(\Sigma)^{\tau \circ f_\Heis \circ \phi} \longrightarrow \widetilde{\mathcal{C}}_n(\Sigma)^{\tau \circ \phi}
\end{equation}
and, applying relative twisted Borel--Moore homology, we obtain an $R[\Heis]$-linear isomorphism
\begin{equation}
\label{eq:twisted-linear-isomorphism-theta}
H_*^{BM}(\mathcal{C}_{n}(\Sigma),\mathcal{C}_{n}(\Sigma,\partial^-(\Sigma)) ;R[\Heis])_{f_\Heis^{-1} \circ \tau^{-1}} \longrightarrow H_*^{BM}(\mathcal{C}_{n}(\Sigma),\mathcal{C}_{n}(\Sigma,\partial^-(\Sigma)) ;R[\Heis])_{\tau^{-1}}
\end{equation}
with $R[\Heis]$-local coefficients and an $(R,S)$-linear isomorphism
\begin{equation}
\label{eq:twisted-linear-isomorphism-R-theta}
H_*^{BM}\bigl(\mathcal{C}_{n}(\Sigma),\mathcal{C}_{n}(\Sigma,\partial^-(\Sigma)) ;{}_{\tau \circ f_\Heis}\!V\bigr) \longrightarrow H_*^{BM}\bigl(\mathcal{C}_{n}(\Sigma),\mathcal{C}_{n}(\Sigma,\partial^-(\Sigma)) ;{}_{\tau}\!V\bigr)
\end{equation}
with $V$-local coefficients.

These isomorphisms together form a twisted representation of the mapping class group $\mathfrak{M}(\Sigma)$. To formulate precisely the meaning of this statement, we consider mapping classes as morphisms in a groupoid whose objects are elements of $\mathrm{Aut}^+(\Heis)$. In standard terminology, this is called the \emph{action groupoid} for the left action $\mathfrak{M}(\Sigma) \to \mathrm{Aut}^+(\Heis)$, which we denote by $\mathfrak{M}(\Sigma) \actiongroupoidleft \mathrm{Aut}^+(\Heis)$. Morphisms $\sigma \to \tau$ are the mapping classes $f$ such that $\tau\circ f_\Heis=\sigma$.\footnote{\label{fn:action-groupoid}In general, for a group homomorphism $\theta \colon G \to H$, the groupoid $G \actiongroupoidleft H$ has object set $H$ and the morphisms $h \to h'$ are elements $g \in G$ such that $h'\theta(g)=h$, with composition given by the group operation of $G$. The notation comes from the fact that the connected components $\pi_0(G \actiongroupoidleft H)$ are given by the left cosets $\theta(G) \setminus H$. The terminology \emph{action groupoid} comes from the special case of a left action $\theta \colon G \to \mathrm{Aut}(X)$ on an object $X$. There is also a dual notion of a (right) action groupoid $H \actiongroupoidright G$ associated to an anti-homomorphism $\theta \colon G \nrightarrow H$, for example a right action $\theta \colon G \nrightarrow \mathrm{Aut}(X)$ on an object $X$.}

The above discussion proves the following, which is a functorial formulation of the twisted representation announced in Theorem \ref{athm:twisted}.

\begin{theorem}[{Theorem \ref{athm:twisted}(b)}]
\label{thm:MCG}
Associated to any left representation $V$ of $\Heis$ in $(R,S)$-bimodules, there is a functor 
\begin{equation}
\label{eq:representation-MCG-thm}
\mathfrak{M}(\Sigma) \actiongroupoidleft \mathrm{Aut}^+(\Heis) \longrightarrow \bimod{R}{S}
\end{equation}
where each object $\tau \colon \Heis \to \Heis$ is sent to the $(R,S)$-bimodule
\[
H_n^{BM}\bigl(\mathcal{C}_{n}(\Sigma),\mathcal{C}_{n}(\Sigma,\partial^-(\Sigma)) ;{}_{\tau}\!V\bigr)
\]
and the morphism $f \colon \tau \circ f_\Heis \to \tau$ is sent to the $(R,S)$-linear isomorphism \eqref{eq:twisted-linear-isomorphism-R-theta}.
\end{theorem}

As a corollary of this theorem, with $V$ equal to the tautological representation (see Remark~\ref{rmk:Heisenberg-matrices}), we obtain a twisted finite-dimensional representation of the mapping class group.

\subsection{The linearised translation action}
\label{ss:untwisted-rep-linearised}

The underlying set of the Heisenberg group $\Heis$ is $\Z\times H_1(\Sigma;\Z) \cong \Z^{2g+1}$, which we may endow with its usual affine structure (a simply transitive action of the abelian group $\Z^{2g+1}$). The first key observation is that left multiplication in $\Heis$ preserves this affine structure, in other words, for any $h_0=(k_0,x_0)\in \Heis$, the left translation action $l_{h_0} \colon \Heis \to \Heis$ is an affine automorphism. Indeed $l_{(k_0,x_0)}(k,x)=(k_0+k+x_0.x,x_0+x)$.
Left multiplication therefore gives us an affine action
\begin{equation}
\label{eq:affine-action}
\Heis \longrightarrow \mathrm{Aff}(\Z^{2g+1}).
\end{equation}

Recall that an \emph{affine space over a ring $R$} consists of an $R$-module $M$ and a set $A$ equipped with a simply transitive action of $(M,+)$, the underlying additive group of $M$. By a usual abuse of notation in affine geometry, we denote this (simply transitive) action of $(M,+)$ on $A$ also by `$+$'.
An \emph{affine automorphism} of $A$ is a bijection $f \colon A \to A$ such that $f(a+m)=f(a)+\varphi(m)$ for all $a\in A$ and $m\in M$ and some (necessarily unique) $R$-linear automorphism $\varphi \in \mathrm{Aut}_R(M)$.
After choosing an element $a_0\in A$ the affine space $A$ embeds as $M\times \{1\} \subset M\oplus R$ via $a_0 + m \mapsto (m,1)$, and any affine automorphism extends uniquely to an $R$-linear automorphism of $M \oplus R$, which is given by
\begin{align*}
\left( \begin{matrix}
\varphi & v_0 \\
0 & 1
\end{matrix} \right) ,
\end{align*}
where $v_0 \in \Hom_R(R,M) \cong M$ is the unique element such that $f(a_0) = a_0 + v_0$. This gives an injective group homomorphism, depending on $a_0 \in A$:
\[
\mathrm{Aff}(A) \lhook\joinrel\longrightarrow \mathrm{Aut}_R(M \oplus R),
\]
where $\mathrm{Aff}(A)$ denotes the group of affine automorphisms of $A$.
Applying this to the affine space $A = \Z^{2g+1}$ over $\Z$ with $a_0 = 0$, we obtain an injective group homomorphism
\begin{equation}
\label{eq:linearisation}
\mathrm{Aff}(\Z^{2g+1}) \lhook\joinrel\longrightarrow GL_{2g+2}(\Z)
\end{equation}
given by the above formula with $v_0 = f(0)$. The $\Z$-linear automorphism $\varphi$ underlying the affine automorphism $f = l_{(k_0,x_0)}$, given by the left translation action on $\Heis$, is $\varphi(k,x) = (k + x_0.x , x)$. We also have $v_0 = f(0) = l_{(k_0,x_0)}(0) = (k_0,x_0)$ in this case. The linearised action
\begin{equation}
\label{eq:linearised-action}
\rho_L = \eqref{eq:linearisation} \circ \eqref{eq:affine-action} \colon \Heis \longrightarrow GL_{2g+2}(\Z)
\end{equation}
on $L = \Heis \oplus \Z \cong \Z^{2g+2}$ is therefore given by the formula
\begin{equation}
\label{eq:rho-L-formula}
(k_0,x_0) \longmapsto \left( \begin{matrix}
1 & x_0.- & k_0 \\
0 & I & x_0 \\
0 & 0 & 1
\end{matrix} \right) ,
\end{equation}
in other words $\rho_L(k_0,x_0)$ acts by $(k,x,t) \mapsto (k',x',t')$, where
\[
\begin{cases}
k'=k+t\,k_0+x_0.x\\
x'=x+t\,x_0\\
t'=t.
\end{cases}
\]
The nice feature of this representation is that the twisted representation ${}_{\tau}\!L$  is canonically isomorphic to $L$, for any $\tau \in \mathrm{Aut}^+(\Heis)$.

\begin{lemma}
\label{lem:untwisting-tautological}
For any $\tau \in \mathrm{Aut}^+(\Heis)$, the linear map $\tau \oplus \id_\Z \colon L \to {}_{\tau}\!L$ gives an isomorphism of  
$\Z[\Heis]$-modules.
\end{lemma}
\begin{proof}
We first observe that any orientation-preserving automorphism of $\Heis$ preserves the structure of $\Heis = \Z^{2g+1}$ as a free $\Z$-module (see Corollary~\ref{lem:AutHeis} and Remark~\ref{rmk:linearising}). We therefore have a tautological homomorphism
\[
\mathrm{Aut}^+(\Heis) \longrightarrow GL_{2g+1}(\Z)
\]
given by sending $\tau$ to $\tau$ via the identification of the underlying set of $\Heis$ with $\Z^{2g+1}$. Composing this with the inclusion $GL_{2g+1}(\Z) \subset GL_{2g+2}(\Z)$ given by $- \oplus \id_{\Z}$, we obtain a $\Z$-linear automorphism $\tau \oplus \id_\Z \colon L \to L$. Notice that this inclusion is the linearisation homomorphism \eqref{eq:linearisation} restricted to $GL_{2g+1}(\Z) \subset \mathrm{Aff}(\Z^{2g+1})$.

We next check that $\tau$ intertwines the affine action $l_{h_0}$ and the twisted affine action $l_{\tau(h_0)}$, for any $h_0 \in \Heis$. For any other $h \in \Heis$, we have
\begin{align*}
l_{\tau(h_0)}(h) &= \tau(h_0)h \\
&= \tau(h_0\tau^{-1}(h)) = \tau\left(l_{h_0}(\tau^{-1}(h))\right) ,
\end{align*}
so we have the identity
\[
l_{\tau(h_0)}=\tau\circ l_{h_0}\circ \tau^{-1}
\]
in $\mathrm{Aff}(\Z^{2g+1})$. After linearisation, we obtain the formula
\begin{equation}
\label{eq:intertwining-formula}
\rho_L(\tau(h_0)) = (\tau \oplus \id_\Z) \circ \rho_L(h_0) \circ (\tau^{-1} \oplus \id_\Z) ,
\end{equation}
which is precisely the statement that $\tau \oplus \id_\Z$ intertwines the linear action $\rho_L(h_0)$ and its twist $\rho_L(\tau(h_0))$ by $\tau$.
\end{proof}

\begin{remark}
Alternatively, we may check formula \eqref{eq:intertwining-formula} in coordinates. By Corollary~\ref{lem:AutHeis} and Remark~\ref{rmk:linearising} we may identify $\mathrm{Aut}^+(\Heis)$ with the subgroup
\[
Sp(H) \ltimes H^\vee = \left( \begin{matrix}
1 & H^\vee \\
0 & Sp(H)
\end{matrix} \right) \subset GL(\Z \oplus H),
\]
where $H^\vee$ denotes $\Hom(H,\Z)$. Each element $\tau$ of $\mathrm{Aut}^+(\Heis)$ is then of the form $\left( \begin{matrix}
1 & v.- \\
0 & M
\end{matrix} \right)$ for $M \in Sp(H)$ and $v \in H$. Each $h_0 = (k_0,x_0) \in \Heis$ acts on $L = \Heis \oplus \Z = (\Z \oplus H) \oplus \Z$ by the block matrix \eqref{eq:rho-L-formula}. We have $\tau(k_0,x_0) = (k_0 + v.x_0 , Mx_0)$, which acts by the block matrix
\[
\left( \begin{matrix}
1 & Mx_0.- & k_0 + v.x_0 \\
0 & I & Mx_0 \\
0 & 0 & 1
\end{matrix} \right) .
\]
The intertwining formula \eqref{eq:intertwining-formula} then corresponds to the calculation:
\begin{align*}
\rho_L(\tau(h_0)) \circ (\tau \oplus \id_\Z) &= \left(\begin{matrix}
1 & Mx_0.- & k + v.x_0 \\
0 & I & Mx_0 \\
0 & 0 & 1
\end{matrix}\right) \left(\begin{matrix}
1 & v.- & 0 \\
0 & M & 0 \\
0 & 0 & 1
\end{matrix}\right) \\
&= \left(\begin{matrix}
1 & (v+x_0).- & k_0 + v.x_0 \\
0 & M & Mx_0 \\
0 & 0 & 1
\end{matrix}\right) \\
&= \left(\begin{matrix}
1 & v.- & 0 \\
0 & M & 0 \\
0 & 0 & 1
\end{matrix}\right) \left(\begin{matrix}
1 & x_0.- & k_0 \\
0 & I & x_0 \\
0 & 0 & 1
\end{matrix}\right) \\
&= (\tau \oplus \id_\Z) \circ \rho_L(h_0),
\end{align*}
where for the second equality we use the fact that $(Mx_0.-)\circ M = x_0.- \colon H \to \Z$ since $M \in Sp(H)$ preserves the symplectic form $-.-$.
\end{remark}

The following theorem is then immediate from Lemma~\ref{lem:untwisting-tautological}.

\begin{theorem}[Theorem \ref{athm:l-regular}]
\label{thm:l-regular}
There is a representation
\[
\mathfrak{M}(\Sigma) \longrightarrow \mathrm{Aut}_{\Z} \bigl(H_n^{BM}\bigl( \mathcal{C}_n(\Sigma),\mathcal{C}_{n}(\Sigma,\partial^-(\Sigma)) ; L\bigr) \bigr)
\]
associating to $f \in \mathfrak{M}(\Sigma)$ the composition of the isomorphism
\[
H_n^{BM}\bigl( \mathcal{C}_n(\Sigma),\mathcal{C}_{n}(\Sigma,\partial^-(\Sigma));L \bigr) \longrightarrow H_n^{BM}\bigl( \mathcal{C}_n(\Sigma),\mathcal{C}_{n}(\Sigma,\partial^-(\Sigma)) ;{}_{f_\Heis}\!L \bigr)
\]
induced by the coefficient isomorphism $f_\Heis \oplus \id_\Z$ with the functorial homology isomorphism
\[
\mathcal{C}_n(f)_* \colon H_n^{BM}\bigl( \mathcal{C}_n(\Sigma) , \mathcal{C}_{n}(\Sigma,\partial^-(\Sigma));{}_{f_\Heis}\!L\bigr) \longrightarrow H_n^{BM}\bigl( \mathcal{C}_n(\Sigma),\mathcal{C}_{n}(\Sigma,\partial^-(\Sigma));L \bigr)\ .
\]
\end{theorem}

\section{The Schr{\"o}dinger local system}
\label{Schroedinger}

A well-known representation of the Heisenberg group, which is infinite-dimensional and unitary, is the \emph{Schr{\"o}dinger representation}, which is parametrised by a non-zero real number $\hbar$. The left action on the Hilbert space $L^{2}(\R^{g})$ is given by the following formula:

\begin{equation}
\label{eq:Schroedinger-formula}
\left[\Pi_{\hbar }\left(k,x=\sum_{i=1}^g p_i a_i + q_i b_i \right) \psi \right](s)=e^{i\hbar \frac{k+p\cdot q}{2}}e^{i\hbar p\cdot s}\psi (s+q) .
\end{equation}

The Schr{\"o}dinger representation occupies a special place in the representation theory of the Heisenberg group, and in this section we explain how to leverage its properties to construct an untwisted representation of the full mapping class group $\mathfrak{M}(\Sigma)$, after passing to a central extension.

In \S\ref{ss:Stone-vN} we first discuss the Schr\"odinger representation in more detail, as well as the Stone--von Neumann theorem and its consequences. In \S\ref{ss:universal} we discuss the universal central extension of the mapping class group. We then prove Theorem \ref{athm:universal-central} in \S\ref{ss:inf-dim-unitary}, constructing untwisted representations of the universal central extension of the mapping class group. In \S\ref{ss:fd-unitary} we explain how to adapt our construction to the finite-dimensional analogues of the Schr\"odinger representation to prove Theorem \ref{athm:universal-central-fd}. Finally, in \S\ref{ss:sesquilinear-form} we show that, although these representations are not unitary in an obvious way, they do preserve a certain perfect sesquilinear pairing between two different homology groups (Proposition \ref{prop:perfect-sesquilinear-form}).

\subsection{The Schr\"odinger representation and the Stone--von Neumann theorem.}
\label{ss:Stone-vN}

The \emph{continuous Heisenberg group} is defined similarly to the discrete Heisenberg group. As a set it is $\R \times H_1(\Sigma;\R)$, with multiplication given by $(s,x)(t,y) = (s+t+x.y,x+y)$, where $.$ is the intersection form on $H_1(\Sigma;\R) = H_\R$. We denote it by $\Heisr$ and note that the discrete Heisenberg group $\Heis$ is naturally a subgroup of $\Heisr$. The proofs of Lemma \ref{lem:split-ses} and Corollary \ref{lem:AutHeis} work similarly for $\Heisr$, and the group $\mathrm{Aut}^+(\Heisr)$ of automorphisms of $\Heisr$ acting trivially on the centre decomposes as a semi-direct product $\mathrm{Aut}^+(\Heisr) \cong Sp(H_1(\Sigma;\R)) \ltimes H^1(\Sigma;\R)$.
There is a natural inclusion
\[
\mathrm{Aut}^+(\Heis) \lhook\joinrel\longrightarrow \mathrm{Aut}^+(\Heisr),
\]
denoted by $\varphi \mapsto \varphi_\R$, such that $\varphi_\R$ is an extension of $\varphi$. This inclusion is compatible with the decompositions into semi-direct products.

As an alternative to the explicit formula \eqref{eq:Schroedinger-formula}, the Schr{\"o}dinger representation may also be defined more abstractly as follows. First note that $\Heisr$ may be written as a semi-direct product
\[
\Heisr = \R \{ (0,b_1) ,\ldots, (0,b_g) \} \ltimes \R \{ (1,0) , (0,a_1) ,\ldots, (0,a_g) \} ,
\]
where $a_1,\ldots,a_g,b_1,\ldots,b_g$ form a symplectic basis for $H_1(\Sigma;\R)$. Fix a real number $\hbar > 0$. There is a one-dimensional complex unitary representation
\[
\R \{ (1,0) , (0,a_1) ,\ldots, (0,a_g) \} \longrightarrow \mathbb{S}^1 = U(1)
\]
defined by $(t,x) \mapsto e^{\hbar i t / 2}$. This may then be induced to a complex unitary representation of the whole group $\Heisr$ on the complex Hilbert space $L^2(\R \{ (0,b_1) ,\ldots, (0,b_g) \}) = L^2(\R^g)$. This is the Schr{\"o}dinger representation of $\Heisr$. From now on, let us denote this representation by
\begin{equation}
\label{eq:Schroedinger}
W = L^2(\R^g) \qquad\text{and}\qquad
\rho_W \colon \Heisr \longrightarrow U(W).
\end{equation}
We will usually not make the dependence on $\hbar$ explicit in the notation; in particular we write $\rho_W$ instead of $\rho_{W,\hbar}$. The key properties of $\rho_W$ that we shall need are the following.

\begin{theorem}[{The \emph{Stone--von Neumann theorem}; \cite[page 19]{LionVergne}}]
\label{th:LionVergne}
\hspace{0pt}

\begin{itemize}
\item[\textup{(a)}] The representation \eqref{eq:Schroedinger} is irreducible.
\item[\textup{(b)}] If $V$ is a complex Hilbert space and
\[
\rho \colon \Heisr \longrightarrow U(V)
\]
is an irreducible unitary representation such that $\rho(t,0) = e^{\hbar i t / 2} \id_{V}$ for all $t \in \R$, then there is an isomorphism $\kappa \colon V \to W$ such that, for any $(t,x) \in \Heisr$, the following diagram commutes:
\[
\begin{tikzcd}
V \ar[rr,"\kappa"] \ar[d,"{\rho(t,x)}",swap] && W \ar[d,"{\rho_W(t,x)}"] \\
V \ar[rr,"\kappa"] && W.
\end{tikzcd}
\]
\end{itemize}
\end{theorem}

\begin{corollary}
\label{uniqueness-of-Schroedinger}
If $\rho \colon \Heisr \to U(W)$ is an irreducible unitary representation such that $\rho(t,0) = e^{\hbar i t / 2}  \id_W$ for all $t \in \R$, then there is a commutative diagram
\[
\begin{tikzcd}
\Heisr \ar[rr,"{\rho_W}"] \ar[drr,"{\rho}",swap] && U(W) \ar[d,"{\mathrm{ad}_u}"] \\
&& U(W)
\end{tikzcd}
\]
for some element $u \in U(W)$, which is unique up to rescaling by an element of $\mathbb{S}^1$. Here, $\mathrm{ad}_u$ denotes the adjoint action of $u$ given by $\mathrm{ad}_u(v) = uvu^{-1}$.
\end{corollary}
\begin{proof}
Applying Theorem \ref{th:LionVergne} to the case $V=W$, the unitary isomorphism $\kappa$ provides an element $u$ as claimed. To see uniqueness up to a scalar in $\mathbb{S}^1$, note that any two such elements $u$ differ by an automorphism of the irreducible representation $\rho_W$, which must therefore be a scalar (in $\C^*$) multiple of the identity, by Schur's lemma. Moreover, since $\rho_W$ is unitary, this scalar must lie in $\mathbb{S}^1 \subset \C^*$.
\end{proof}

\begin{definition}
Denote by $PU(W) = U(W) / \mathbb{S}^1$ the \emph{projective unitary group} of the Hilbert space $W$. Since scalar multiples of the identity are central, this fits into a central extension
\begin{equation}
\label{eq:projective-unitary-central-extension}
\begin{tikzcd}
1 \ar[r] & \mathbb{S}^1 \ar[r] & U(W) \ar[r] & PU(W) \ar[r] & 1.
\end{tikzcd}
\end{equation}
We denote by $\omega_{PU} \colon PU(W) \times PU(W) \to \mathbb{S}^1$ a choice of $2$-cocycle corresponding to this central extension; in other words we write $U(W) \cong \mathbb{S}^1 \times PU(W)$ with multiplication given by $(s,g)(t,h) = (st\omega_{PU}(g,h),gh)$.
\end{definition}

\begin{definition}
\label{def:Shale-Weil}
For an automorphism $\varphi \in \mathrm{Aut}(\Heisr)$, Corollary \ref{uniqueness-of-Schroedinger} applied to the representation $\rho = \rho_W \circ \varphi$ tells us that there is a unique element $u = T(\varphi) \in PU(W)$ such that $\rho_W \circ \varphi = T(\varphi) \rho_W T(\varphi)^{-1}$. The assignment $\varphi \mapsto T(\varphi)$ defines a group homomorphism
\begin{equation}
\label{eq:Shale-Weil-projective-on-Aut}
T \colon \mathrm{Aut}(\Heisr) \longrightarrow PU(W).
\end{equation}
Restricting the homomorphism \eqref{eq:Shale-Weil-projective-on-Aut} to the subgroup $Sp_{2g}(\R) = Sp(H_\R) \subset \mathrm{Aut}^+(\Heisr) \subset \mathrm{Aut}(\Heisr)$, we obtain a projective representation
\begin{equation}
\label{eq:Shale-Weil-projective}
R = T|_{Sp_{2g}(\R)} \colon Sp_{2g}(\R) \longrightarrow PU(W).
\end{equation}
This is the \emph{Shale--Weil projective representation} of the symplectic group. (It is sometimes also called the \emph{Segal--Shale--Weil projective representation}, see for example \cite[page 53]{LionVergne}.) Pulling back the central extension \eqref{eq:projective-unitary-central-extension} along the homomorphism \eqref{eq:Shale-Weil-projective}, we then obtain a central extension
\begin{equation}
\label{eq:lifted-central-extension}
\begin{tikzcd}
1 \ar[r] & \mathbb{S}^1 \ar[r] & \overline{Sp}_{2g}(\R) \ar[r] & Sp_{2g}(\R) \ar[r] & 1
\end{tikzcd}
\end{equation}
and a lifted representation
\begin{equation}
\label{eq:Shale-Weil-lifted}
\overline{R} \colon \overline{Sp}_{2g}(\R) \longrightarrow U(W).
\end{equation}
The group $\overline{Sp}_{2g}(\R)$ is sometimes known as the \emph{Mackey obstruction group} of the projective representation \eqref{eq:Shale-Weil-projective}. Since \eqref{eq:lifted-central-extension} is pulled back from \eqref{eq:projective-unitary-central-extension} along $R$, we may write $\overline{Sp}_{2g}(\R) \cong \mathbb{S}^1 \times Sp_{2g}(\R)$ with multiplication given by $(s,g)(t,h) = (s.t.\omega_{Sp}(g,h),gh)$, where
\[
\omega_{Sp} = \omega_{PU} \circ (R \times R) \colon Sp_{2g}(\R) \times Sp_{2g}(\R) \longrightarrow PU(W) \times PU(W) \longrightarrow \mathbb{S}^1 .
\]
\end{definition}

\subsection{Universal central extensions.}
\label{ss:universal}

We recall the definition of the universal central extension of a group $G$ (see for example \cite[\S 6.9]{Weibel} for more details).

\begin{definition}
\label{def:stably-universal-extension}
If $G$ is a perfect group, i.e.~if we have $H_1(G;\Z)=0$, then there is an isomorphism $H^2(G;H_2(G;\Z)) \cong \Hom(H_2(G;\Z),H_2(G;\Z))$ by the universal coefficient theorem, and the $H_2(G;\Z)$-central extension of $G$ corresponding to the identity map is the \emph{universal central extension} of $G$. For $G = \mathfrak{M}(\Sigma)$ (recall that $\Sigma = \Sigma_{g,1}$), we have that $G$ is perfect when $g\geq 3$ and we have $H_2(G;\Z) \cong \Z$ when $g\geq 4$ (see \cite[Theorems 5.1 and 6.1]{Korkmaz}). In particular, for $g\geq 4$, let us denote by
\[
1 \longrightarrow \Z \longrightarrow \widetilde{\mathfrak{M}}(\Sigma) \xrightarrow{\;\mathfrak{p}\;} \mathfrak{M}(\Sigma) \longrightarrow 1
\]
the \emph{universal central extension of $\mathfrak{M}(\Sigma)$}.
\end{definition}

Consider the inclusion of surfaces $\Sigma_{g,1} \hookrightarrow \Sigma_{h,1}$ given by boundary connected sum with $\Sigma_{h-g,1}$. This induces an inclusion of mapping class groups
\begin{equation}
\label{eq:stabilisation-MCG}
\mathfrak{M}(\Sigma_{g,1}) \lhook\joinrel\longrightarrow \mathfrak{M}(\Sigma_{h,1})
\end{equation}
by extending diffeomorphisms by the identity on $\Sigma_{h-g,1}$. Recall from the introduction that the inclusion map \eqref{eq:stabilisation-MCG} induces isomorphisms on first and second (co)homology for all $h\geq g\geq 4$ (see \cite{Harer1985} or \cite{Wahl2013}), so the pullback of $\widetilde{\mathfrak{M}}(\Sigma_{h,1})$ along this inclusion is $\widetilde{\mathfrak{M}}(\Sigma_{g,1})$. The following definition is therefore consistent for any $g\geq 1$.

\begin{definition}
\label{def:stably-universal-extension-2}
We define the \emph{stably universal central extension} $\widetilde{\mathfrak{M}}(\Sigma_{g,1})$ of $\mathfrak{M}(\Sigma_{g,1})$ to be the pullback of $\widetilde{\mathfrak{M}}(\Sigma_{h,1})$ for any $h \geq \mathrm{max}(g,4)$.
\end{definition}

The following lemma explains how Morita's crossed homomorphism $\mathfrak{d}$ behaves with respect to increasing the genus via this inclusion. We first remark that the boundary connected sum decomposition $\Sigma_{h,1} \cong \Sigma_{g,1} \natural \Sigma_{h-g,1}$, which induces the inclusion \eqref{eq:stabilisation-MCG} above, also induces a free product decomposition $\pi_1(\Sigma_{h,1}) \cong \pi_1(\Sigma_{g,1}) * \pi_1(\Sigma_{h-g,1})$ of fundamental groups. This, in turn, induces a direct sum decomposition $H^1(\Sigma_{h,1}) \cong H^1(\Sigma_{g,1}) \oplus H^1(\Sigma_{h-g,1})$ on first cohomology, using the identification $H^1(-) \cong \Hom(\pi_1(-),\Z)$.

\begin{lemma}
\label{lem:Morita-crossed-hom-stabilisation}
The diagram
\begin{equation}
\label{eq:Morita-crossed-hom-stabilisation}
\begin{tikzcd}
\mathfrak{M}(\Sigma_{g,1}) \ar[rr,hook,"\eqref{eq:stabilisation-MCG}"] \ar[d,"\mathfrak{d}",swap] && \mathfrak{M}(\Sigma_{h,1}) \ar[d,"\mathfrak{d}"] \\
H^1(\Sigma_{g,1}) \ar[rr] && H^1(\Sigma_{h,1})
\end{tikzcd}
\end{equation}
commutes, where the bottom horizontal arrow is the inclusion of the left-hand summand of the decomposition $H^1(\Sigma_{h,1}) \cong H^1(\Sigma_{g,1}) \oplus H^1(\Sigma_{h-g,1})$.
\end{lemma}

\begin{proof}[Proof of Lemma \ref{lem:Morita-crossed-hom-stabilisation}]
As in the definition of the Morita crossed homomorphism (see equation \eqref{eq:Morita-crossed-hom} in \S\ref{ss:Morita-crossed-hom}), we use the the identification $H^1(-) \cong \Hom(\pi_1(-),\Z)$.
Under this identification, the bottom horizontal arrow in \eqref{eq:Morita-crossed-hom-stabilisation} is given by pre-composition with the projection $\mathrm{pr}_1 \colon \pi_1(\Sigma_{h,1}) \cong \pi_1(\Sigma_{g,1}) * \pi_1(\Sigma_{h-g,1}) \twoheadrightarrow \pi_1(\Sigma_{g,1})$ onto the first factor of the free product.

Let $f \in \mathfrak{M}(\Sigma_{g,1})$ and write $\hat{f} \in \mathfrak{M}(\Sigma_{h,1})$ for its image under \eqref{eq:stabilisation-MCG}. Let $\gamma \in \pi_1(\Sigma_{h,1})$ and write $\gamma_1 = \mathrm{pr}_1(\gamma) \in \pi_1(\Sigma_{g,1})$ and $\gamma_2 = \mathrm{pr}_2(\gamma) \in \pi_1(\Sigma_{h-g,1})$ for its images under the projections onto the two free factors.
Recall that the definition of $d_i(\gamma)$ (see equation \eqref{eq:definition-of-d}) depends only on the decomposition of $\gamma$ into the standard generators $\alpha_j,\beta_j$ of $\pi_1(\Sigma_{h,1})$ after forgetting those with $j>i$. This implies in particular that we have
\[
d_i(\gamma) = d_i(\gamma_1) \qquad\text{and}\qquad d_i(\pi_1(\hat{f})(\gamma)) = d_i(\pi_1(f)(\gamma_1))
\]
for $1 \leq i \leq g$. Moreover, since $\hat{f}$ acts by the identity on $\Sigma_{h-g,1}$, we also have
\[
d_i(\pi_1(\hat{f})(\gamma)) = d_i(\gamma)
\]
for $g+1 \leq i \leq h$.
From the defining formula \eqref{eq:Morita-crossed-hom} we deduce that
\[
\mathfrak{d}_{\hat{f}}([\gamma]) = \sum_{i=1}^h d_i(\pi_1(\hat{f})(\gamma)) - d_i(\gamma) = \sum_{i=1}^g d_i(\pi_1(f)(\gamma_1)) - d_i(\gamma_1) = \mathfrak{d}_f([\gamma_1]) = (\mathfrak{d}_f \circ \mathrm{pr}_1)([\gamma]),
\]
and so \eqref{eq:Morita-crossed-hom-stabilisation} commutes.
\end{proof}

\subsection{Constructing the representations.}
\label{ss:inf-dim-unitary}

We now prove Theorem \ref{athm:universal-central}.

From the previous two subsections, we have the following diagram:
\begin{equation}
\label{eq:big-diagram}
\begin{tikzcd}
\mathfrak{M}(\Sigma) \ar[r,"{\Phi = (\mathfrak{s},\mathfrak{d})}"] & \mathrm{Aut}^+(\Heis) \ar[r] & \mathrm{Aut}^+(\Heisr) \ar[r,"T"] & PU(W) \\
& Sp(H) \ltimes H \ar[u,"{\cong}"] \ar[r] & Sp(H_\R) \ltimes H_\R \ar[u,"{\cong}"] & 
\end{tikzcd}
\end{equation}
where unmarked arrows denote inclusions. For $g\geq 4$, by the universality of $\widetilde{\mathfrak{M}}(\Sigma)$, there is a morphism of central extensions
\begin{equation}
\label{eq:map-of-central-ext-1}
\begin{tikzcd}
\widetilde{\mathfrak{M}}(\Sigma) \ar[d,"\mathfrak{p}",swap] \ar[rr] && U(W) \ar[d] \\
\mathfrak{M}(\Sigma) \ar[rr] && PU(W)
\end{tikzcd}
\end{equation}
where the bottom horizontal arrow is the composition along the top of \eqref{eq:big-diagram}. Moreover, this extends to all $g\geq 1$ as follows. Consider the commutative diagram\footnote{We freely pass between the different notations $Sp_{2g}(\R) = Sp(H_\R)$ and $\R^{2g} = H_\R$, and similarly for the integral versions, depending on whether or not we wish to emphasise the genus $g$.}
\begin{equation}
\label{eq:Shale-Weil-representation-stabilisation-2}
\begin{tikzcd}
\mathfrak{M}(\Sigma_{g,1}) \ar[d,hook] \ar[rr,"{(\mathfrak{s},\mathfrak{d})}"] && Sp_{2g}(\R) \ltimes \R^{2g} \ar[rr,"T"] \ar[d] && PU(L^2(\R^g)) \ar[d] & U(L^2(\R^g)) \ar[l,two heads] \ar[d] \\
\mathfrak{M}(\Sigma_{h,1}) \ar[rr,"{(\mathfrak{s},\mathfrak{d})}"] && Sp_{2h}(\R) \ltimes \R^{2h} \ar[rr,"T"] && PU(L^2(\R^h)) & U(L^2(\R^h)) \ar[l,two heads]
\end{tikzcd}
\end{equation}
The right-hand side of this diagram arises as follows. We consider $L^2(\R^g)$ as the (closed) subspace of $L^2(\R^h)$ of those $L^2$-functions that factor through $\R^h = \R^g \times \R^{h-g} \twoheadrightarrow \R^g$. Any closed subspace of a Hilbert space has an orthogonal complement, so we may extend unitary automorphisms by the identity on this complement to obtain a homomorphism $U(L^2(\R^g)) \to U(L^2(\R^h))$, which descends to the projective unitary groups. The right-hand square of \eqref{eq:Shale-Weil-representation-stabilisation-2} is a pullback square (this is true for any closed subspace of a Hilbert space). Commutativity of the left-hand square follows from Lemma \ref{lem:Morita-crossed-hom-stabilisation} and commutativity of the middle square follows from the defining property of $T$ (Definition \ref{def:Shale-Weil}). Let us write $\overline{\mathfrak{M}}(\Sigma_{g,1})$ for the pullback of $U(W) \to PU(W)$ along $T \circ (\mathfrak{s},\mathfrak{d})$, and similarly for $\overline{\mathfrak{M}}(\Sigma_{h,1})$. Then $\overline{\mathfrak{M}}(\Sigma_{g,1})$ is the pullback of $\overline{\mathfrak{M}}(\Sigma_{h,1})$ along the inclusion of mapping class groups. From Definitions \ref{def:stably-universal-extension} and \ref{def:stably-universal-extension-2}, we also have that $\widetilde{\mathfrak{M}}(\Sigma_{g,1})$ is the pullback of $\widetilde{\mathfrak{M}}(\Sigma_{h,1})$ along the inclusion.

If we now take $h\geq 4$, then $\widetilde{\mathfrak{M}}(\Sigma_{h,1})$ is by definition the \emph{universal} central extension, so there is a unique morphism of central extensions $\widetilde{\mathfrak{M}}(\Sigma_{h,1}) \to \overline{\mathfrak{M}}(\Sigma_{h,1})$. Pulling back along the inclusion, we obtain a canonical morphism of central extensions $\widetilde{\mathfrak{M}}(\Sigma_{g,1}) \to \overline{\mathfrak{M}}(\Sigma_{g,1})$, even though $\widetilde{\mathfrak{M}}(\Sigma_{g,1})$ is not universal for $g\leq 3$. This gives us the desired morphism of central extensions \eqref{eq:map-of-central-ext-1}.

\begin{notation}
We denote by
\[
S \colon \widetilde{\mathfrak{M}}(\Sigma) \longrightarrow U(W)
\]
the top horizontal map of \eqref{eq:map-of-central-ext-1}.
\end{notation}

\begin{notation}
By abuse of notation, we write
\[
\rho_W \colon \Heis \longrightarrow U(W)
\]
for the restriction of the Schr{\"o}dinger representation \eqref{eq:Schroedinger} to the subgroup $\Heis \subset \Heisr$.
\end{notation}

A consequence of Definition \ref{def:Shale-Weil} is the following.

\begin{lemma}
\label{lem:SandW}
For $g \in \widetilde{\mathfrak{M}}(\Sigma)$ and $h \in \Heis$, we have the following equation in $U(W)$:
\begin{equation}
\label{eq:SandW}
S(g).\rho_W(h).S(g)^{-1} = \rho_W(\Phi(\mathfrak{p}(g))(h)).
\end{equation}
\end{lemma}

We now use this to construct \emph{untwisted} representations of the universal central extension $\widetilde{\mathfrak{M}}(\Sigma)$ of $\mathfrak{M}(\Sigma)$ on the homology of configuration spaces with coefficients in the Schr{\"o}dinger representation.

Let $\widetilde{\mathcal{C}}_n(\Sigma) \to \mathcal{C}_n(\Sigma)$ denote the connected covering of $\mathcal{C}_n(\Sigma)$ corresponding to the kernel of the surjective homomorphism $\pi_1(\mathcal{C}_n(\Sigma)) \twoheadrightarrow \Heis$. This is a principal $\Heis$-bundle. Taking free abelian groups fibrewise, we obtain
\begin{equation}
\label{eq:bundle-of-ZH-modules}
\Z[ \widetilde{\mathcal{C}}_n(\Sigma) ] \longrightarrow \mathcal{C}_n(\Sigma),
\end{equation}
which is a bundle of right $\Z[\Heis]$-modules. Via the Schr{\"o}dinger representation $\rho_W$, the Hilbert space $W$ becomes a left $\Z[\Heis]$-module, and we may take a fibrewise tensor product to obtain
\begin{equation}
\label{eq:bundle-of-Hilbert-spaces}
\Z[ \widetilde{\mathcal{C}}_n(\Sigma) ] \otimes_{\Z[\Heis]} W \longrightarrow \mathcal{C}_n(\Sigma),
\end{equation}
which is a bundle of Hilbert spaces. There is a natural action of the mapping class group $\mathfrak{M}(\Sigma)$ (up to homotopy) on the base space $\mathcal{C}_n(\Sigma)$, and the induced action on $\pi_1(\mathcal{C}_n(\Sigma))$ preserves the kernel of the surjection $\pi_1(\mathcal{C}_n(\Sigma)) \twoheadrightarrow \Heis$ (Proposition \ref{f_Heisenberg}), so that there is a well-defined twisted action of $\mathfrak{M}(\Sigma)$ on the bundle \eqref{eq:bundle-of-ZH-modules}, in the following sense. There are homomorphisms
\begin{align*}
\alpha \colon \mathfrak{M}(\Sigma) &\longrightarrow \mathrm{Aut}_{\Z}\bigl( \Z[ \widetilde{\mathcal{C}}_n(\Sigma)] \longrightarrow \mathcal{C}_n(\Sigma) \bigr) \\
\Phi \colon \mathfrak{M}(\Sigma) &\longrightarrow \mathrm{Aut}(\Heis)
\end{align*}
such that, for any $g \in \mathfrak{M}(\Sigma)$, $h \in \Heis$ and $m \in \Z[ \widetilde{\mathcal{C}}_n(\Sigma) ]$, we have
\begin{equation}
\label{eq:twisted-action}
\alpha(g)(m.h) = \alpha(g)(m).\Phi(g)(h).
\end{equation}
In other words, $\Phi$ measures the failure of $\alpha$ to be an action by fibrewise $\Z[\Heis]$-module automorphisms. In the above, the target of $\alpha$ is the group of $\Z$-module automorphisms of the bundle \eqref{eq:bundle-of-ZH-modules}, in other words the group of self-homeomorphisms of the total space $\Z[\widetilde{\mathcal{C}}_n(\Sigma)]$ that preserve the fibres of the projection and that are $\Z$-linear (but not necessarily $\Z[\Heis]$-linear) on each fibre.

\begin{theorem}
\label{thm:universal-ce-action}
The stably universal central extension $\widetilde{\mathfrak{M}}(\Sigma)$ of $\mathfrak{M}(\Sigma)$ acts on \eqref{eq:bundle-of-Hilbert-spaces} by Hilbert space bundle automorphisms
\[
\gamma \colon \widetilde{\mathfrak{M}}(\Sigma) \longrightarrow U \Bigl( \Z[ \widetilde{\mathcal{C}}_n(\Sigma) ] \otimes_{\Z[\Heis]} W \longrightarrow \mathcal{C}_n(\Sigma) \Bigr)
\]
via the formula
\begin{equation}
\label{eq:universal-ce-formula}
\gamma(g)(m \otimes v) = \alpha(\mathfrak{p}(g))(m) \otimes S(g)(v)
\end{equation}
for all $g \in \widetilde{\mathfrak{M}}(\Sigma)$, $m \in \Z[ \widetilde{\mathcal{C}}_n(\Sigma) ]$ and $v \in W$.
\end{theorem}
\begin{proof}
We must verify that the formula \eqref{eq:universal-ce-formula} is additive in $m$, unitary in $v$ and that it is $\Z[\Heis]$-balanced. The first two properties are evident by the definitions of $\alpha$ and $S$ respectively. The key property to be verified is therefore the third one, which in more detail says the following. Since we are taking the (fibrewise) tensor product over $\Z[\Heis]$, we have that $m.h \otimes v =m\otimes  \rho_W(h)(v)$ for any $h \in \Heis$, $m \in \Z[ \widetilde{\mathcal{C}}_n(\Sigma) ]$ and $v \in W$. (Note here that we denote the right $\Heis$-action on the fibres of $\Z[ \widetilde{\mathcal{C}}_n(\Sigma) ]$ simply by juxtaposition, whereas the left $\Heis$-action on $W$ is the Schr{\"o}dinger representation, denoted by $\rho_W$.) We therefore have to verify that, for each fixed $g \in \widetilde{\mathfrak{M}}(\Sigma)$, the formula \eqref{eq:universal-ce-formula} gives the same answer when applied to $m.h \otimes v$ or to $m \otimes \rho_W(h)(v)$. To see this, we calculate:
\begin{align*}
\gamma(g)(m.h \otimes v)
&= \alpha(\mathfrak{p}(g))(m.h) \otimes S(g)(v) && \text{by definition} \\
&= \alpha(\mathfrak{p}(g))(m) . \Phi(\mathfrak{p}(g))(h) \otimes S(g)(v) && \text{by eq.~\eqref{eq:twisted-action}} \\
&= \alpha(\mathfrak{p}(g))(m) \otimes  \rho_W(\Phi(\mathfrak{p}(g))(h)) (S(g)(v)) && \text{since $\otimes$ is over $\Z[\Heis]$} \\
&= \alpha(\mathfrak{p}(g))(m) \otimes  S(g) \circ \rho_W(h) \circ S(g)^{-1} (S(g)(v)) && \text{by eq.~\eqref{eq:SandW} [Lemma \ref{lem:SandW}]} \\
&= \alpha(\mathfrak{p}(g))(m) \otimes  S(g) (\rho_W(h)(v)) && \text{simplifying} \\
&= \gamma(g)(m \otimes \rho_W(h)(v)). && \text{by definition}
\end{align*}
This tells us that the formula \eqref{eq:universal-ce-formula} gives a well-defined fibrewise unitary bundle automorphism (i.e.~an automorphism of Hilbert space bundles) of \eqref{eq:bundle-of-Hilbert-spaces} for each fixed $g \in \widetilde{\mathfrak{M}}(\Sigma)$. It is then clear from the formula \eqref{eq:universal-ce-formula} that $\gamma$ is a group homomorphism.
\end{proof}

\begin{theorem}[{Theorem \ref{athm:universal-central}}]
\label{thm:Schroedinger}
The action of the mapping class group on the Borel--Moore homology of the configuration space $\mathcal{C}_n(\Sigma)$ with coefficients in the Schr{\"o}dinger representation induces a well-defined complex representation of the stably universal central extension $\widetilde{\mathfrak{M}}(\Sigma)$ of the mapping class group $\mathfrak{M}(\Sigma)$:
\begin{equation}
\label{eq:Schroedinger-rep-MCG}
\widetilde{\mathfrak{M}}(\Sigma) \longrightarrow GL^{\mathrm{bd}} \bigl( H_n^{BM}\bigl( \mathcal{C}_n(\Sigma) , \mathcal{C}_n(\Sigma,\partial^-(\Sigma)) ; W \bigr) \bigr)
\end{equation}
lifting a natural projective representation of $\mathfrak{M}(\Sigma)$ on the same space. Here, $GL^{\mathrm{bd}}$ denotes the group of bounded linear operators with respect to a certain Hilbert structure.
\end{theorem}

\begin{proof}
By Theorem \ref{thm:universal-ce-action}, we have a well-defined functor from the group $\widetilde{\mathfrak{M}}(\Sigma)$ to the category of spaces equipped with bundles of Hilbert spaces. Moreover, elements of the mapping class group fix the boundary of $\Sigma$ pointwise, so the action of the mapping class group on $\mathcal{C}_n(\Sigma)$ preserves the subspace $\mathcal{C}_n(\Sigma,\partial^-(\Sigma))$. Thus we have a functor from the group $\widetilde{\mathfrak{M}}(\Sigma)$ to the category of pairs of spaces equipped with bundles of Hilbert spaces.

On the other hand, relative twisted Borel--Moore homology $H_n^{BM}(-)$ is a functor from the category of pairs of spaces equipped with bundles of Hilbert spaces (and bundle maps whose underlying map of spaces is proper) to the category of complex vector spaces (cf.~\cite[\S V.4 and \S V.5]{Bredon}). Moreover, on the full subcategory of pairs of spaces admitting a finite relative Borel--Moore CW-complex structure, it may be augmented to take values in the category of Hilbert spaces and bounded operators. For objects, the Hilbert structure on Borel--Moore homology is induced by the Hilbert structure on Borel--Moore cellular chain complexes given by an orthogonal direct sum, over all cells, of copies of the Hilbert space coefficients. For morphisms, we may assume by cellular approximation that the underlying map is cellular, so it induces a bounded operator of Borel--Moore cellular chain complexes and hence a bounded operator on Borel--Moore homology. (We note that for boundedness it is essential that the relative Borel--Moore CW-complex structure has finitely many cells.)

The pair $(\mathcal{C}_n(\Sigma) , \mathcal{C}_n(\Sigma,\partial^-(\Sigma)))$ admits a finite relative Borel--Moore CW-complex structure, so we may compose these two functors to obtain the representation \eqref{eq:Schroedinger-rep-MCG} of $\widetilde{\mathfrak{M}}(\Sigma)$ by bounded linear operators on a Hilbert space, as desired.

This automatically descends to a projective representation of $\mathfrak{M}(\Sigma)$ since it sends the kernel of the central extension $\widetilde{\mathfrak{M}}(\Sigma) \to \mathfrak{M}(\Sigma)$ into the centre of the bounded linear automorphism group, which is contained in the kernel of the projection onto the projective bounded linear automorphism group.
\end{proof}

\subsection{Finite-dimensional Schr\"odinger representations.}
\label{ss:fd-unitary}

For an integer $N\geq 2$, the finite-dimensional Schr{\"o}dinger representation is a left action of the discrete Heisenberg group $\Heis$ on the Hilbert space $W_N = L^2((\Z/N)^g)$, which may be defined
as follows:
\begin{equation}
\label{eq:F_Schroedinger-formula}
\left[\varpi_N\left(k,x=\sum_{i=1}^g p_i a_i + q_i b_i \right) \psi \right](s)=e^{i\pi \frac{k+p\cdot q}{N}}e^{i\frac{2\pi}{N} p\cdot s}\psi (s+q) .
\end{equation}
Note that this matches the generic formula with $\hbar=\frac{2\pi}{N}$. It may also be constructed by composing the natural finite quotient
\[
\Heis = \Z^g \ltimes \Z^{g+1} \relbar\joinrel\twoheadrightarrow \Heis_N=(\Z/N)^g \ltimes \bigl( \Z/2N \times (\Z/N)^g \bigr)
\]
with the representation of $\Heis_N$ obtained by induction from the one-dimensional representation $\Z/2N \times (\Z/N)^g \twoheadrightarrow \Z/2N \hookrightarrow \mathbb{S}^1 = U(\C)$, where the second map is $t \mapsto \mathrm{exp}\left(\frac{\pi it}{N}\right)$.
Note that the kernel of the quotient map $\Heis \twoheadrightarrow \Heis_N$ may be written, in the original definition $\Heis=\Z\times H$, as the normal subgroup $I_N=\{(2Nk,Nx) \mid k\in \Z, \ x\in H\}$.

We may adapt the above construction using $W_N$ in place of $W$ and using the analogue of the Stone--von Neumann theorem for $W_N$, proven for $N$ even in \cite[Theorem~2.4]{Gelca_Uribe_TQFT} (see also \cite[Theorem~3.2]{Gelca_Uribe} and \cite[Theorem~2.6]{Gelca_Hamilton}). The proof for odd $N$ works similarly. The theorem states that $W_N$ is, up to projectively unique unitary isomorphism, the unique unitary representation of the finite group $\Heis_N$ where the action of $u=(1,0)$ is multiplication by $e^\frac{i\pi}{N}$.
The odd case is studied in \cite{DeRenziMartel} with explicit formulas for the untwisting process. We quote from \cite{DeRenziMartel} that for odd $N$ the mapping class group action on $\Heis$ fixes the kernel $I_N$ so that the Stone-von Neumann theorem constructs a projective representation of the mapping class group on the Borel--Moore homology of the configuration space $\mathcal{C}_n(\Sigma)$ with coefficients in $W_N$. 

If $N$ is even, then the action $f_\Heis$ of a mapping class $f$ on the Heisenberg group $\Heis$ may fail to preserve the normal subgroup $I_N$. From the formula \eqref{eq:action} we see that $f_\Heis(I_N)=I_N$ if and only if $\delta_f(x)$ is even for every $x$, equivalently $f$ is in the kernel of the reduced-modulo-$2$ crossed homomorphism $\overline{\delta} \colon \mathfrak{M}(\Sigma) \rightarrow H^1(\Sigma;\Z/2)$.
From equation \eqref{eq:Morita-crossed-hom} (and Proposition \ref{p:Morita-crossed-hom}), we see that $\overline{\delta}_f(x)\in \Z/2$ depends only on the modulo-$2$ symplectic action $f_*$ on $H_1(\Sigma;\Z/2)$. Denoting by $A_f=(\alpha_{i\, j})$ the matrix of this action in the symplectic basis $(a_1,\dots,a_g,b_1,\dots,b_g)$, we have
\begin{equation}\label{eq:Morita-crossed-2}
\overline \delta_f\left(\sum_{i=1}^g s_ia_i+s_{i+g}b_i\right)=\sum_{j=1}^{2g}\sum_{i=1}^{g} \alpha_{i\, j}\alpha_{i+g\, j}s_j
\end{equation}
This crossed homomorphism has a spin structure interpretation; indeed a spin structure can be defined as a quadratic form $q \colon H_1(\Sigma;\Z/2)\rightarrow \Z/2$, $q(0)=0$ and $q(x+y)=q(x)+q(y)+x.y$. It defines a modulo-$2$ crossed homomorphism $d_q$ via the formula $d_q(x)=q(f_*(x))-q(x)$. One can check that $\overline{\delta}=d_{q_0}$, where $q_0$ is the quadratic form that vanishes on the canonical basis.
It follows that the kernel of $\overline{\delta}$ is the spin mapping class group $\mathfrak{M}(\Sigma,q_0)$. The Arf invariant of $q_0$ is $0$, meaning that the index of this subgroup is $2^{g-1}(2^g+1)$. From the Stone--von Neumann theorem we obtain with the previous untwisting method:

\begin{theorem}[{Theorem \ref{athm:universal-central-fd}}]
\label{thm:Schroedinger-fd}
For $N$ even, there exists a projective action of the spin  mapping class group $\mathfrak{M}(\Sigma,q_0)$ on the Borel--Moore homology of the configuration space $\mathcal{C}_n(\Sigma)$ with coefficients in $W_N$ obtained by composing a coefficient isomorphism with the homological action. This gives a projective  representation of the of the spin mapping class group  on the
$\bigl(
\begin{smallmatrix}
2g+n-1 \\
n \\
\end{smallmatrix}
\bigr) N^{g}$-dimensional
complex Hilbert space
\begin{equation}
\label{eq:Schroedinger-rep-MCG-fd}
\mathcal{V}_{N,n} = H_n^{BM}\bigl( \mathcal{C}_n(\Sigma) , \mathcal{C}_n(\Sigma,\partial^-(\Sigma)) ; W_N \bigr) .
\end{equation}
\end{theorem}

\subsection{Preservation of a sesquilinear form}
\label{ss:sesquilinear-form}

As explained in the proof of Theorem \ref{thm:Schroedinger}, when using a Hilbert space as local coefficients, after choosing a CW-complex structure, which gives a Hilbert structure on the cellular chain groups as an orthogonal sum indexed by the cells, we get a Hilbert structure on homology.\footnote{In general we would have to quotient the closed subspace of cycles by the closure of the boundary subspace, but here we can use a finite cell structure.} However, it is not true that mapping classes will act as unitary operators on chains (or on homology). This is because, although we may use the cellular approximation theorem to represent the action of any mapping class by a cellular self-map of the configuration space, this self-map will in general fail to be an automorphism of the CW-complex structure, in particular it will fail to be a homeomorphism. Nevertheless, we will exhibit, in this section, a perfect sesquilinear form on homologies that is preserved by the action of mapping classes.

Suppose that $V$ is a representation of the discrete Heisenberg group $\Heis$ defined over a commutative ring with involution $R$, equipped with a perfect\footnote{Recall that a \emph{perfect} pairing $A \otimes B \to R$ is an $R$-linear map such that both dual maps $A \to \mathrm{Hom}_R(B,R)$ and $B \to \mathrm{Hom}_R(A,R)$ are isomorphisms.} Hermitian pairing $V\otimes V\rightarrow R$.
By Poincar{\'e} duality, and the fact that $\mathcal{C}_n(\Sigma)$ is a connected, oriented $2n$-manifold with boundary $\mathcal{C}_n(\Sigma,\partial\Sigma) = \{ c \in \mathcal{C}_n(\Sigma) \mid c \cap \partial\Sigma \neq \varnothing \}$, we obtain a sesquilinear pairing
\begin{equation}
\label{eq:non-self-pairing}
\langle -,- \rangle \colon H_n^{BM}(\mathcal{C}_n(\Sigma),\partial^-;V) \otimes H_n(\mathcal{C}_n(\Sigma),\partial^+;V) \longrightarrow R,
\end{equation}
where $\partial^\pm$ is an abbreviation of $\mathcal{C}_n(\Sigma,\partial^{\pm}(\Sigma))$, and we note that the boundary $\partial\mathcal{C}_n(\Sigma) = \mathcal{C}_n(\Sigma,\partial\Sigma)$ decomposes as $\partial^+ \cup \partial^-$ with $\partial\partial^+ = \partial\partial^- = \partial^+ \cap \partial^-$ (i.e.\ forming a manifold triad), corresponding to the decomposition of the boundary of the surface $\partial\Sigma = \partial^+(\Sigma) \cup \partial^-(\Sigma)$.
In more detail, the pairing \eqref{eq:non-self-pairing} is constructed by pre-composing the evaluation map $\mathbf{H} \otimes \Hom_R(\mathbf{H},R) \to R$, where $\mathbf{H} = H_n^{BM}(\mathcal{C}_n(\Sigma),\partial^-;V)$, with $\id \otimes \eta$, where $\eta$ is Poincar{\'e} duality composed with the canonical morphism from compactly-supported cohomology to the dual of Borel--Moore homology:
\begin{equation}
\label{eq:eta}
\eta = \varepsilon \circ \mathrm{PD} \colon H_n(\mathcal{C}_n(\Sigma),\partial^+;V) \xrightarrow{\;\cong\;} H^n_c(\mathcal{C}_n(\Sigma),\partial^-;V) \longrightarrow \Hom_R( \mathbf{H} , R ).
\end{equation}
By naturality of the homomorphism $\varepsilon$ and of Poincar{\'e} duality, the pairing \eqref{eq:non-self-pairing} is preserved by the actions of homeomorphisms of $\mathcal{C}_n(\Sigma)$ (that fix its boundary) and of coefficient isomorphisms. Since the mapping class group acts via homeomorphisms of the configuration space and coefficient isomorphisms, this means that \eqref{eq:non-self-pairing} is preserved by the mapping class group action.

Let us from now on assume that $V$ is a \emph{free} $R$-module and choose a basis $(v_j)_{j\in J}$ for $V$. Since the Hermitian pairing on $V$ is perfect, we may also choose another basis $(v'_j)_{j\in J}$ for $V$ that is dual to $(v_j)_{j\in J}$ with respect to this pairing.
It follows from Theorem \ref{basis} that the $R$-module $\mathbf{H} = H_n^{BM}(\mathcal{C}_n(\Sigma),\partial^-;V)$ has a basis of the form $\widetilde{E}_k\otimes v_j$, indexed by $k\in \mathcal{K}$ and $j \in J$, where $\widetilde{E}_k$ is a lift to the Heisenberg cover $\widetilde{\mathcal{C}}_n(\Sigma)$ of the product of simplices $\mathcal{C}_{k_i}(\gamma_i)$ for $1\leq i\leq 2g$ and $k_1 + \cdots + k_{2g} = n$. Moreover, in Theorem \ref{basis}, we also computed the compactly-supported cohomology $H^n_c(\mathcal{C}_n(\Sigma),\partial^-;V)$, which is Poincaré dual to $H_n(\mathcal{C}_n(\Sigma),\partial^+;V)$. As an $R$-module, it has a basis of the form $\widetilde{E}'_k \otimes v'_j$, indexed by $k\in \mathcal{K}$ and $j \in J$, where $\widetilde{E}'_k$ is a lift to the Heisenberg cover $\widetilde{\mathcal{C}}_n(\Sigma)$ of the $n$-cube $E'_k$ given by the product of $n$ pairwise disjoint arcs in $\Sigma$ with boundary on $\partial^+(\Sigma)$ where exactly $k_i$ of them intersect $\gamma_i$ transversely for each $1\leq i\leq 2g$. (This dual basis is described in more detail in \S\ref{s:computations}.) With respect to the pairing \eqref{eq:non-self-pairing}, we have that $\langle \widetilde{E}_k \otimes v_j , \widetilde{E}'_{k'} \otimes v'_{j'} \rangle$ evaluates to the Kronecker $\delta_{(k,j)}^{(k',j')}$ up to a sign. This proves that the pairing \eqref{eq:non-self-pairing} is perfect:

\begin{proposition}
\label{prop:perfect-sesquilinear-form}
Let $R$ be a commutative ring with involution and $V$ be a free $R$-module equipped with a perfect Hermitian pairing $V \otimes V \to R$. Suppose that we have a representation of the discrete Heisenberg group $\Heis$ on $V$ respecting this pairing. Then the sesquilinear pairing \eqref{eq:non-self-pairing} is perfect.
\end{proposition}

\begin{remark}
This applies in particular if $V$ is a complex Hilbert space with a unitary representation of $\Heis$, so it applies to the (projective) representations of $\mathfrak{M}(\Sigma)$ and of $\mathfrak{M}(\Sigma,q_0)$ constructed in Theorems \ref{thm:Schroedinger} and \ref{thm:Schroedinger-fd} (Theorems \ref{athm:universal-central} and \ref{athm:universal-central-fd}).
\end{remark}

\section{Relation to the Moriyama and Magnus representations}
\label{relation-to-Moriyama}

In this section we study the kernels of the twisted representations that we have constructed in Theorem \ref{athm:twisted} in the case when the coefficients are $V=\Z[\Heis]$, and prove Proposition \ref{aprop:kernel}. The proof will use:
\begin{itemize}
\item a theorem of Moriyama~\cite{Moriyama}, which identifies each $\mathfrak{J}(i)$ with the kernel of a certain homological representation of $\mathfrak{M}(\Sigma)$;
\item a theorem of Suzuki~\cite{Suzuki2005}, which identifies the Magnus kernel with the kernel of a certain twisted homological representation of $\mathfrak{M}(\Sigma)$ (a homological interpretation of the Magnus representation, which was originally defined via Fox calculus);
\end{itemize}
together with a study of the connections between our representations and those of Moriyama and Suzuki.

\subsection{The Moriyama representation.}
\label{ss:Moriyama}

Moriyama \cite{Moriyama} studied the action of the mapping class group $\mathfrak{M}(\Sigma)$ on the homology group $H_n^{BM}(\mathcal{F}_n(\Sigma');\Z)$ with trivial coefficients, where $\Sigma'$ denotes $\Sigma$ minus a point on its boundary and $\mathcal{F}_n(-)$ denotes the ordered configuration space, where elements are ordered $n$-tuples of distinct points. On the other hand, our construction \eqref{eq:representation-MCG-thm} (Theorem \ref{thm:MCG}) may be re-interpreted as a twisted representation
\begin{equation}
\label{eq:twisted-representation}
\mathfrak{M}(\Sigma) \longrightarrow \mathrm{Aut}_{\Z[\Heis]}^{\mathrm{tw}}\Bigl( H_n^{BM}(\mathcal{C}_n(\Sigma');\Z[\Heis]) \Bigr) .
\end{equation}
We pause to explain this re-interpretation. We must first of all explain the twisted automorphism group on the right-hand side of \eqref{eq:twisted-representation}. Let us write $\mathrm{Mod}_{\bullet}$ for the category whose objects are pairs $(R,M)$ consisting of a ring $R$ and a right $R$-module $M$, and whose morphisms are pairs $(\theta \colon R \to R' , \varphi \colon M \to M')$ such that $\varphi(mr) = \varphi(m)\theta(r)$. The automorphism group of $(R,M)$ in $\mathrm{Mod}_{\bullet}$ is written $\mathrm{Aut}_R^{\mathrm{tw}}(M)$; note that this is generally larger than the automorphism group $\mathrm{Aut}_R(M)$ of $M$ in $\mathrm{Mod}_R$.

If we set $R = V = \Z[\Heis]$ in Theorem \ref{thm:MCG}, then the functor \eqref{eq:representation-MCG-thm} that it supplies is of the form $\mathfrak{M}(\Sigma) \actiongroupoidleft \mathrm{Aut}^+(\Heis) \to \mathrm{Mod}_{\Z[\Heis]}$. In general, for any left group action $\theta \colon G \to \mathrm{Aut}(K)$, each functor $F \colon G \actiongroupoidleft \mathrm{Im}(\theta) \to \mathrm{Mod}_{\Z[K]}$ corresponds to a group homomorphism $G \to \mathrm{Aut}_{\Z[K]}^{\mathrm{tw}}(F(\id_K))$.\footnote{See footnote \ref{fn:action-groupoid} on page \pageref{fn:action-groupoid} for an explanation of the groupoid $G \actiongroupoidleft H$ in general.}
Thus \eqref{eq:representation-MCG-thm} corresponds to a homomorphism
\[
\mathfrak{M}(\Sigma) \longrightarrow \mathrm{Aut}_{\Z[\Heis]}^{\mathrm{tw}} \Bigl( H_n^{BM}\bigl(\mathcal{C}_n(\Sigma),\mathcal{C}_n(\Sigma,\partial^-(\Sigma)) ;\Z[\Heis]\bigr) \Bigr) .
\]
Finally, removing a point (equivalently, removing the closed interval $\partial^-(\Sigma)$) from the boundary of $\Sigma$ corresponds, on Borel--Moore homology of configuration spaces $\mathcal{C}_n(\Sigma)$, to taking homology relative to the subspace $\mathcal{C}_n(\Sigma,\partial^-(\Sigma))$ of configurations having at least one point in the interval. Thus $H_n^{BM}(\mathcal{C}_n(\Sigma),\mathcal{C}_n(\Sigma,\partial^-(\Sigma)) ;\Z[\Heis])$ and $H_n^{BM}(\mathcal{C}_n(\Sigma');\Z[\Heis])$ are isomorphic as $\Z[\Heis]$-modules, and we obtain \eqref{eq:twisted-representation}.

\begin{remark}
Forgetting the $\Z[\Heis]$-module structure gives an embedding of the right-hand side of \eqref{eq:twisted-representation} into the (untwisted) automorphism group $\mathrm{Aut}_\Z(H_n^{BM}(\mathcal{C}_n(\Sigma');\Z[\Heis]))$ over $\Z$.
\end{remark}

When $n=2$, Moriyama's representation is a quotient of ours. To see this, we consider the quotient of groups $\Heis \twoheadrightarrow \Z/2 = \mathfrak{S}_2$ given by sending $\sigma \mapsto \sigma$ and $a_i,b_i \mapsto 1$, which induces a map of twisted $\mathfrak{M}(\Sigma)$-representations
\begin{equation}
\label{eq:Heisenberg-to-Moriyama}
H_2^{BM}(\mathcal{C}_2(\Sigma');\Z[\Heis]) \relbar\joinrel\twoheadrightarrow H_2^{BM}(\mathcal{C}_2(\Sigma');\Z[\mathfrak{S}_2]) \cong H_2^{BM}(\mathcal{F}_2(\Sigma');\Z).
\end{equation}
The map \eqref{eq:Heisenberg-to-Moriyama} is surjective by Proposition \ref{prop:basis-naturality}, which tells us that it is isomorphic to a direct sum of copies of the surjective ring homomorphism $\Z[\Heis] \twoheadrightarrow \Z[\mathfrak{S}_2]$ induced by the quotient of groups $\Heis \twoheadrightarrow \mathfrak{S}_2$. The isomorphism on the right-hand side of \eqref{eq:Heisenberg-to-Moriyama} follows from Shapiro's lemma. (Shapiro's lemma holds for arbitrary coverings with ordinary homology, and for \emph{finite-sheeted} coverings with Borel--Moore homology. The proof for Borel--Moore homology, interpreted as the homology of the complex of locally-finite chains, is exactly the same as for ordinary homology, using the assumption that the covering is finite-sheeted to preserve the locally-finite property when projecting down the covering.) It therefore follows that the kernel of our representation is a subgroup of the kernel of $H_2^{BM}(\mathcal{F}_2(\Sigma');\Z)$, which was proven by Moriyama to be the Johnson kernel $\mathfrak{J}(2)$. In \S\ref{s:computations} we will compute the action of a genus-$1$ separating twist $T_\gamma \in \mathfrak{J}(2)$ on $H_2^{BM}(\mathcal{C}_2(\Sigma');\Z[\Heis])$, and in particular show that it is (very) non-trivial; see Theorem \ref{thm_calculation}. Thus the kernel of $H_2^{BM}(\mathcal{C}_2(\Sigma');\Z[\Heis])$ is \emph{strictly} smaller than $\mathfrak{J}(2)$.

For any $n \geq 2$, we have a map of twisted $\mathfrak{M}(\Sigma)$-representations
\[
H_n^{BM}(\mathcal{C}_n(\Sigma');\Z[\Heis]) \relbar\joinrel\twoheadrightarrow H_n^{BM}(\mathcal{C}_n(\Sigma');\Z),
\]
which is surjective by Proposition \ref{prop:basis-naturality} and the fact that the augmentation map $\Z[\Heis] \twoheadrightarrow \Z$ is surjective. By Shapiro's lemma and the calculations in \S\ref{s:Heisenberg-homology} of $H_n^{BM}(\mathcal{C}_n(\Sigma');V)$ for any local system $V$ on $\mathcal{C}_n(\Sigma')$, there are isomorphisms of abelian groups:
\begin{equation}
\label{eq:ordered-vs-unordered}
H_n^{BM}(\mathcal{F}_n(\Sigma');\Z) \cong H_n^{BM}(\mathcal{C}_n(\Sigma');\Z[\mathfrak{S}_n]) \cong H_n^{BM}(\mathcal{C}_n(\Sigma');\Z) \otimes \Z[\mathfrak{S}_n] .
\end{equation}
Moreover, these are in fact both isomorphisms of $\mathfrak{M}(\Sigma)$-representations over $\Z$: for the left-hand side this is due to the naturality of Shapiro's lemma; for the right-hand side, it is because this isomorphism is induced by the natural map $H_n^{BM}(\mathcal{C}_n(\Sigma');\Z) \otimes \Z[\mathfrak{S}_n] \to H_n^{BM}(\mathcal{C}_n(\Sigma');\Z[\mathfrak{S}_n])$ (one of the maps involved in the universal coefficient theorem) and the $\mathfrak{M}(\Sigma)$-action is induced from an action (up to homotopy) at the level of spaces. Since $\mathfrak{M}(\Sigma)$ acts trivially on $\mathfrak{S}_n$, the right-hand side of \eqref{eq:ordered-vs-unordered} is a direct sum of $n!$ copies of $H_n^{BM}(\mathcal{C}_n(\Sigma');\Z)$. We therefore deduce that the kernel of the $\mathfrak{M}(\Sigma)$-representation $H_n^{BM}(\mathcal{C}_n(\Sigma');\Z)$ is the same as the kernel of the $\mathfrak{M}(\Sigma)$-representation $H_n^{BM}(\mathcal{F}_n(\Sigma');\Z)$. (This is also shown in \cite{PS}.) The latter kernel was proven by Moriyama to be the $n$th term $\mathfrak{J}(n)$ of the Johnson filtration.

Summarising this discussion, we have:

\begin{proposition}
\label{prop:kernel-Johnson-filtration}
The kernel of the twisted $\mathfrak{M}(\Sigma)$-representation \eqref{eq:twisted-representation} is contained in the $n$th term $\mathfrak{J}(n)$ of the Johnson filtration. When $n=2$ it is moreover a \emph{proper} subgroup of the Johnson kernel $\mathfrak{J}(2)$.
\end{proposition}

\subsection{The Magnus representation.}

The kernel of our representation \eqref{eq:twisted-representation} is also contained in the kernel of the Magnus representation. This may be seen as follows. The $\mathfrak{M}(\Sigma)$-equivariant surjection $\Heis \twoheadrightarrow H$ induces a map of twisted $\mathfrak{M}(\Sigma)$-representations
\begin{equation}
\label{eq:Heisenberg-to-higher-Magnus}
H_n^{BM}(\mathcal{C}_n(\Sigma');\Z[\Heis]) \relbar\joinrel\twoheadrightarrow H_n^{BM}(\mathcal{C}_n(\Sigma');\Z[H]),
\end{equation}
which is surjective by Proposition \ref{prop:basis-naturality}.
By a similar argument as in \S\ref{ss:Moriyama} above, the kernel of the twisted $\mathfrak{M}(\Sigma)$-representation $H_n^{BM}(\mathcal{C}_n(\Sigma');\Z[H])$ is the same as the kernel of the twisted $\mathfrak{M}(\Sigma)$-representation $H_n^{BM}(\mathcal{F}_n(\Sigma');\Z[H])$. Moreover, it is shown in \cite{PS} that there is an inclusion of twisted $\mathfrak{M}(\Sigma)$-representations
\begin{equation}
\label{eq:Magnus-to-higher-Magnus}
\bigl[ H_1^{BM}(\mathcal{F}_1(\Sigma');\Z[H]) \bigr]^{\otimes n} \lhook\joinrel\longrightarrow H_n^{BM}(\mathcal{F}_n(\Sigma');\Z[H]).
\end{equation}
By a result of Suzuki~\cite{Suzuki2005}, $H_1^{BM}(\mathcal{F}_1(\Sigma');\Z[H])$ is the \emph{Magnus representation} of $\mathfrak{M}(\Sigma)$ (this is a homological interpretation of the Magnus representation, which was originally defined via Fox calculus). The maps of representations \eqref{eq:Heisenberg-to-higher-Magnus} and \eqref{eq:Magnus-to-higher-Magnus} imply that
\begin{align*}
\mathrm{ker}\bigl[ H_n^{BM}(\mathcal{C}_n(\Sigma');\Z[\Heis]) \bigr] \; &\subseteq\; \mathrm{ker}\bigl[ H_n^{BM}(\mathcal{C}_n(\Sigma');\Z[H]) \bigr] \\
&=\; \mathrm{ker}\bigl[ H_n^{BM}(\mathcal{F}_n(\Sigma');\Z[H]) \bigr] \;\subseteq\; \mathrm{ker}(\mathrm{Magnus}^{\otimes n}).
\end{align*}
In general, if $V$ is a representation of a group $G$ over an integral domain $R$, the kernel of the tensor power $V^{\otimes n}$ consists of those $g \in G$ that act on $V$ by an element of $\{\lambda \in R \mid \lambda^n = 1\}$. For the Magnus representation, the ground ring is $\Z[H]$, whose only roots of $1$ are $\{1\}$ when $n$ is odd and $\{\pm 1\}$ when $n$ is even. Thus when $n$ is odd we have $\mathrm{ker}(\mathrm{Magnus}^{\otimes n}) = \mathrm{ker}(\mathrm{Magnus})$ and when $n$ is even we either have the same equality or $\mathrm{ker}(\mathrm{Magnus}^{\otimes n})$ contains $\mathrm{ker}(\mathrm{Magnus})$ as an index-$2$ subgroup.

Combining this discussion with the statement of Proposition \ref{prop:kernel-Johnson-filtration} and writing $\mathrm{Mag}(\Sigma)$ for the kernel of the Magnus representation, we may complete the proof of Proposition \ref{aprop:kernel}.

\begin{proposition}[{Proposition \ref{aprop:kernel}}]
\label{prop:kernel}
The kernel of \eqref{eq:twisted-representation} is contained in $\mathfrak{J}(n) \cap \mathrm{Mag}(\Sigma)$.
\end{proposition}
\begin{proof}
Let $f$ be an element of the kernel of \eqref{eq:twisted-representation}. By Proposition \ref{prop:kernel-Johnson-filtration}, we know that $f \in \mathfrak{J}(n)$. By the discussion above, we know that the action of $f$ under the Magnus representation is either $\id$ or $-\id$. It remains to rule out the possibility that it is $-\id$, so let us suppose this and derive a contradiction. Consider the morphism of representations
\[
H_1^{BM}(\mathcal{F}_1(\Sigma');\Z[H]) \longrightarrow H_1^{BM}(\mathcal{F}_1(\Sigma');\Z)
\]
induced by the augmentation map $\Z[H] \to \Z$ of the coefficients. Assuming that $f$ acts by $-\id$ under the Magnus representation (the left-hand side), it follows that it also acts by $-\id$ on the representation on the right-hand side. But the right-hand side may be identified with the symplectic action of the mapping class group on $H = H_1(\Sigma;\Z)$, so in particular it follows that $f$ does not lie in the Torelli group, i.e.~$f \notin \mathfrak{J}(1)$. But we know from above that $f \in \mathfrak{J}(n) \subseteq \mathfrak{J}(1)$, a contradiction.
\end{proof}

\begin{remark}
It is known \cite[\S 6]{Suzuki2003} that the kernel of the Magnus representation does not contain $\mathfrak{J}(n)$ for any $n\geq 1$, so Proposition \ref{prop:kernel} implies that the kernel of \eqref{eq:twisted-representation} is \emph{strictly} contained in $\mathfrak{J}(n)$.
\end{remark}

\subsection{Other related representations.}

Recently, the representations of $\mathfrak{M}(\Sigma)$ on the ordinary (rather than Borel--Moore) homology of the configuration space $\mathcal{F}_n(\Sigma)$ has been studied\footnote{This is equivalent to studying the homology of $\mathcal{F}_n(\Sigma')$ since the inclusion $\mathcal{F}_n(\Sigma') \hookrightarrow \mathcal{F}_n(\Sigma)$ is a homotopy equivalence. On the other hand, for \emph{Borel--Moore} homology, this would not be equivalent, since the inclusion is not a \emph{proper} homotopy equivalence.} by Bianchi, Miller and Wilson~\cite{BianchiMillerWilson}: they prove that, for each $n$ and $i$, the kernel of the $\mathfrak{M}(\Sigma)$-representation $H_i(\mathcal{F}_n(\Sigma);\Z)$ contains $\mathfrak{J}(i)$, and is in general strictly \emph{larger} than $\mathfrak{J}(i)$. They conjecture that the kernel of the $\mathfrak{M}(\Sigma)$-representation on the total homology $H_*(\mathcal{F}_n(\Sigma);\Z)$ is equal to the subgroup generated by $\mathfrak{J}(n)$ and the Dehn twist around the boundary. Even more recently, Bianchi and Stavrou \cite{BianchiStavrou2022} have shown that, for $g\geq 2$, the kernel of the $\mathfrak{M}(\Sigma)$-representation $H_n(\mathcal{F}_n(\Sigma);\Z)$ does \emph{not} contain $\mathfrak{J}(n-1)$.

The $\mathfrak{M}(\Sigma)$-representation $H_i(\mathcal{C}_n(\Sigma);\mathbb{F})$, for certain field coefficients $\mathbb{F}$, has been completely computed. For $\mathbb{F} = \mathbb{F}_2$ it has been computed in \cite[Theorem 3.2]{Bianchi2020} and is \emph{symplectic}, i.e.~it restricts to the trivial action on the Torelli group $\mathfrak{T}(\Sigma) = \mathfrak{J}(1)$. For $\mathbb{F} = \Q$ it has been computed in \cite[Theorem 1.4]{Stavrou2023} and is not symplectic, i.e.~its kernel does not contain $\mathfrak{J}(1)$, but it restricts to the trivial action on the Johnson kernel $\mathfrak{J}(2)$.

\section{Computations for \texorpdfstring{$n=2$}{n=2}}
\label{s:computations}

In this section we will do some computations in the case $n=2$, when $V$ is the regular representation $\Z[\Heis]$ of the Heisenberg group $\Heis$. The main goal is to obtain in this case an explicit formula for the action of a Dehn twist along a genus-$1$ separating curve on the generic Heisenberg homology $H_2^{BM}(\mathcal{C}_2(\Sigma),\mathcal{C}_2(\Sigma,\partial^-(\Sigma));\Z[\Heis])$. When the surface has genus $1$ this is displayed in Figure \ref{fig:bigmatrix}; in general, the formula is given by Theorem \ref{thm_calculation}. One may compare these calculations to the calculations of An and Ko~\cite[page 274]{An}, although they consider representations of surface braid groups whereas we consider representations of mapping class groups.

We will start with the case where the surface itself has genus $1$, where we first compute the action of the Dehn twists $T_a$, $T_b$, along the standard essential curves $a$, $b$. Since $T_a$ and $T_b$ act non-trivially on the local system $\Z[\Heis]$, they do not act by automorphisms, but give isomorphisms in the category of spaces with local systems, which, after taking homology with local coefficients, give isomorphisms in the category of $\Z[\Heis]$-modules. We refer to \cite[Chapter 5]{Davis} for functoriality results concerning homology with local coefficients. The upshot is a twisted representation of the full mapping class group $\mathfrak{M}(\Sigma)$. Recall that in Theorem \ref{thm:MCG} we obtained a groupoid formulation of the twisted mapping class group representation as a functor on the \emph{action groupoid} $\mathfrak{M}(\Sigma) \actiongroupoidleft \mathrm{Aut}^+(\Heis)$, which gives here a functor

\begin{equation}
\label{eq:twisted-representation-sec8}
\mathfrak{M}(\Sigma) \actiongroupoidleft \mathrm{Aut}^+(\Heis) \longrightarrow \mathrm{Mod}_{\Z[\Heis]}.
\end{equation}

We briefly recall from \S\ref{rep-MCG} some of the relevant details of the construction of this twisted representation. Let $f\in \mathfrak{M}(\Sigma)$ and let $f_\Heis$ be its action on the Heisenberg group. Then the Heisenberg homology $H_*^{BM}(\mathcal{C}_{n}(\Sigma),\mathcal{C}_{n}(\Sigma,\partial^-(\Sigma)) ;\Z[\Heis])$ is defined from the regular covering space $\widetilde{\mathcal{C}}_{n}(\Sigma)$ associated with the quotient $\phi \colon \mathbb{B}_n(\Sigma)\twoheadrightarrow \Heis$. As explained in \S\ref{rep-MCG}, at the level of homology there is a twisted functoriality and, in particular, associated with $f$, we get a right $\Z[\Heis]$-linear isomorphism 
\[
\mathcal{C}_n(f)_* \colon H_*^{BM}(\mathcal{C}_{n}(\Sigma),\mathcal{C}_{n}(\Sigma,\partial^-(\Sigma)) ;\Z[\Heis])_{f_\Heis^{-1}} \longrightarrow H_*^{BM}(\mathcal{C}_{n}(\Sigma),\mathcal{C}_{n}(\Sigma,\partial^-(\Sigma)) ;\Z[\Heis])\ .
\]
Our choice for twisting on the source with $f_\Heis^{-1}$ rather than on the target with $f_\Heis$ will slightly simplify the writing of the matrix. Note also that when working with coefficients in a left $\Z[\Heis]$-representation $V$ the twisting on the right by $f_\Heis^{-1}$ will correspond to twisting the action on $V$ by $f_\Heis$ (see \eqref{eq:twisting-on-left-or-right}).
More generally, for any $\tau\in \mathrm{Aut}(\Heis)$, we have a \emph{shifted} isomorphism
\[
(\mathcal{C}_n(f)_*)_\tau \colon H_*^{BM}(\mathcal{C}_{n}(\Sigma),\mathcal{C}_{n}(\Sigma,\partial^-(\Sigma)) ;\Z[\Heis])_{f_\Heis^{-1}\circ \tau} \longrightarrow H_*^{BM}(\mathcal{C}_{n}(\Sigma),\mathcal{C}_{n}(\Sigma,\partial^-(\Sigma)) ;\Z[\Heis])_\tau\ .
\]
In terms of the functor \eqref{eq:twisted-representation-sec8} on the action groupoid, the above map $(\mathcal{C}_n(f)_*)_{\tau}$ is the image of the morphism $f \colon  \tau^{-1} \circ f_\Heis \to \tau^{-1}$ of $\mathfrak{M}(\Sigma) \actiongroupoidleft \mathrm{Aut}^+(\Heis)$. If $f$, $g$ are two mapping classes, the composition formula (functoriality of \eqref{eq:twisted-representation-sec8}) states the following:
\[
\mathcal{C}_n(g\circ f)_*=\mathcal{C}_n( g)_*\circ (\mathcal{C}_n( f)_*)_{g_\Heis^{-1}}\ .
\]
We will need to compute compositions in specific bases. Note that a basis $B$ for a right $\Z[\Heis]$-module $M$ is also a basis for the twisted module $M_\tau$, $\tau\in \mathrm{Aut}(\Heis)$.
\begin{lemma}\label{lem:matrix}
Let $M,M'$ be free right $\Z[\Heis]$-modules with fixed bases $B$, $B'$ and let $\tau\in \mathrm{Aut}(\Heis)$.
If a $\Z[\Heis]$-linear map $F \colon M\rightarrow M'$ has matrix $\mathrm{Mat}(F)$  in the bases $B$, $B'$, then the matrix of the shifted
$\Z[\Heis]$-linear map $F_\tau \colon M_\tau\rightarrow M'_\tau$ is $\tau^{-1}(\mathrm{Mat}(F))$.
\end{lemma}

The action of $\tau^{-1}$ on the matrix is given by its action on each individual coefficient.

\begin{proof}
We note that the maps $F$ and $F_\tau$ are equal as maps of $\Z$-modules.
Let $B=(e_j)_{j\in J}$, $B'=(f_i)_{i\in I}$, $\mathrm{Mat}(F)=(m_{i,j})_{i\in I,j\in J}$. Then for coefficients
$h_j\in \Heis$, $j\in J$, we have
\begin{align*}
F_\tau \biggl( \sum_j e_j\cdot_\tau h_j \biggr) &= F \biggl( \sum_j e_j\tau(h_j) \biggr) \\
&= \sum_{i,j}f_im_{ij}\tau(h_j)\\
&= \sum_{i,j}f_i\cdot_\tau \tau^{-1}(m_{ij}) h_j ,
\end{align*}
which gives the stated result.
\end{proof}

\begin{figure}[t]
\centering
\includegraphics[scale=0.4]{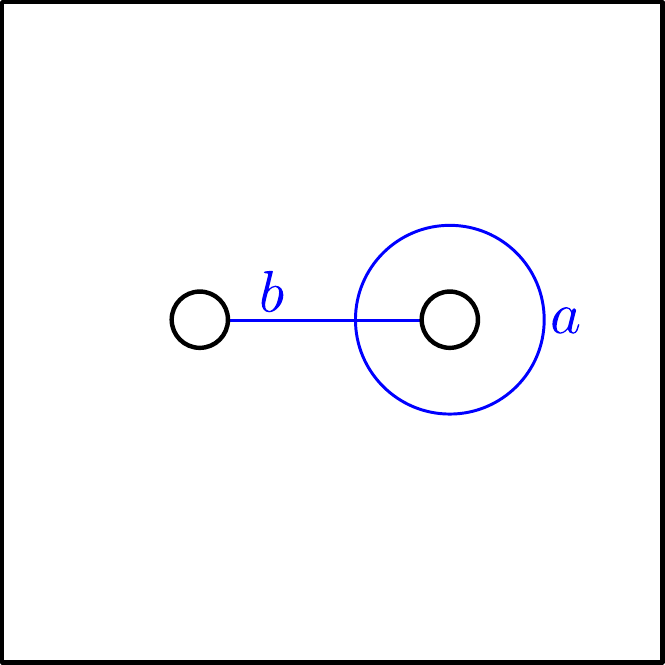}
\includegraphics[scale=0.4]{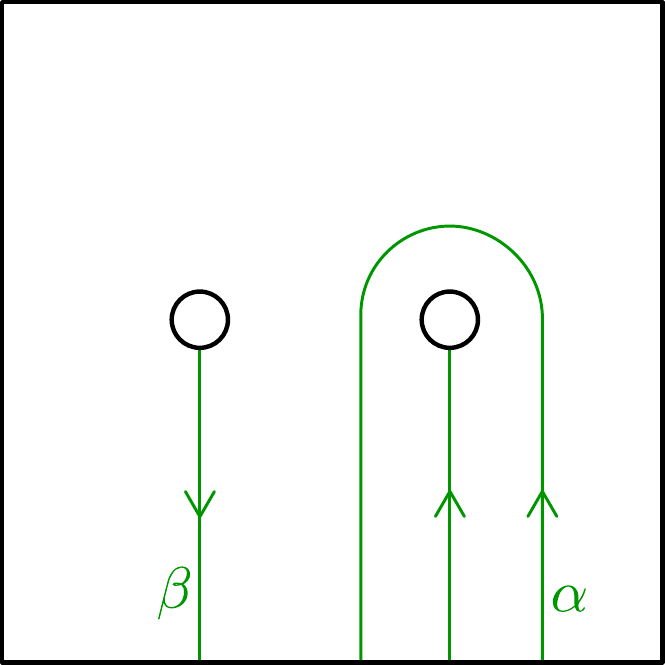}
\includegraphics[scale=0.4]{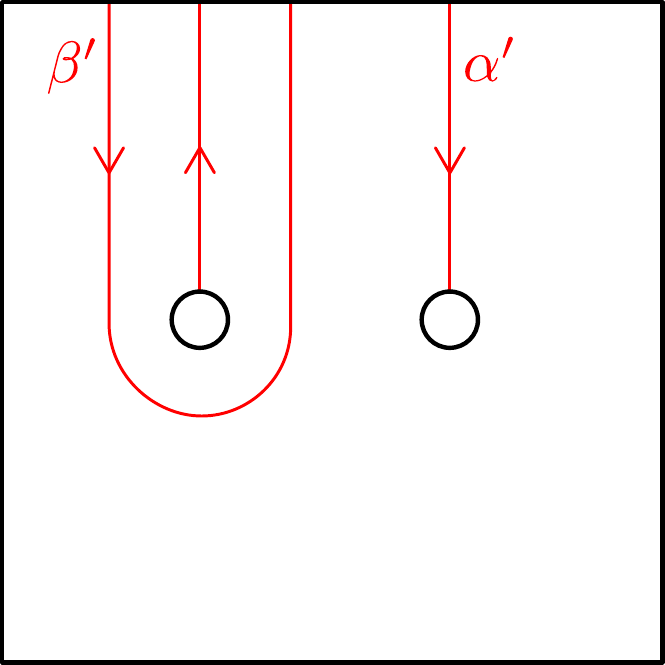}
\caption{The closed curves $a$, $b$ and the arcs $\alpha$, $\beta$, $\alpha'$, $\beta'$. \\ \emph{In all of our figures, we depict surfaces $\Sigma_{g,1}$ as rectangles with $2g$ holes, which are identified in pairs by reflections to form $g$ handles.}}
\label{fig:ab}
\end{figure}

\subsection{Genus one}

Here we consider the genus-$1$ case with $n=2$ configuration points. Let $a$, $b$ be the simple closed curves representing the symplectic basis of $H_1(\Sigma)$ previously denoted by $a_1,b_1$; see Figure \ref{fig:ab}. We will use the same notation $a$, $b$ for the curves, their homology classes and their lifts in $\Heis$ which were previously denoted by $\tilde{a}$, $\tilde{b}$. The corresponding Dehn twists are denoted by $T_a$, $T_b$.
The surface braid group $\mathbb{B}_2(\Sigma)$ is generated by the three elements $\alpha$, $\beta$, $\sigma$, where we again drop the subscript `$1$'. We will depict $\alpha$ and $\beta$ as the arcs in the middle of Figure \ref{fig:ab}. Although these arcs are not based loops, they may be completed to based loops, uniquely up to homotopy, by a path of configurations contained in the bottom edge of the square, since the space of configurations of two points in the bottom edge of the square is contractible and contains the base configuration. In this notation, the quotient $\phi \colon \mathbb{B}_2(\Sigma) \twoheadrightarrow \Heis$ (see Corollary \ref{hom_phi}) sends $\alpha \mapsto a$, $\beta \mapsto b$ and $\sigma \mapsto u$.

The homology module $H_2^{BM}(\mathcal{C}_{2}(\Sigma),\mathcal{C}_{2}(\Sigma,\partial^-(\Sigma)) ;\Z[\Heis])$ was computed using the compression trick in Theorem \ref{basis}. It is free of rank $3$ over $\Z[\Heis]$ with a basis indexed by the ordered partitions of $2$ into two parts.
Here we replace the arcs $\gamma_1$, $\gamma_2$ from Figure \ref{fig:model-surface} with the arcs $\alpha$, $\beta$ depicted in Figure \ref{fig:ab}, and the basis is denoted by $w(\alpha)=E_{(2,0)}$, $w(\beta)=E_{(0,2)}$, $v(\alpha,\beta)=E_{(1,1)}$.
The first two are represented by properly embedded $2$-simplices, while the third one is represented by a properly embedded square.
In more detail, $w(\alpha)$ is represented by the cycle in the $2$-point configuration space given by the subspace where both points lie on the arc $\alpha$. Similarly, $w(\beta)$ is given by the subspace where both points lie on $\beta$ and $v(\alpha,\beta)$ is given by the subspace where exactly one point lies on each of these arcs.

In fact, we have to be even more careful to specify these elements precisely, since the preceding description only determines them \emph{up to the action of the deck transformation group} $\Heis$, because we have just described cycles in the configuration space $\mathcal{C}_2(\Sigma)$, whereas cycles for the Heisenberg-twisted homology are cycles in the covering space $\widetilde{\mathcal{C}}_2(\Sigma)$. Let us fix, once and for all, a lift $\widetilde{c}_0$ of the base configuration $c_0 \subset \partial\Sigma$. Then any contractible subspace $X \subset \mathcal{C}_2(\Sigma)$ has a canonical lift $\widetilde{X} \subset \widetilde{\mathcal{C}}_2(\Sigma)$ to the covering space, after choosing a path in $\mathcal{C}_2(\Sigma)$ from the base configuration $c_0$ to a point in $X$, which is uniquely determined by requiring that the chosen path lifts to a path in $\widetilde{\mathcal{C}}_2(\Sigma)$ from $\widetilde{c}_0$ to $\widetilde{X}$. Once we have a simplex or square in $\mathcal{C}_2(\Sigma)$ representing a relative cycle, a lift to $\widetilde{\mathcal{C}}_2(\Sigma)$ is therefore determined by a choice of a path (called a ``tether'') in $\mathcal{C}_2(\Sigma)$ from $c_0$ to a point in the cycle. For $w(\alpha)$, $w(\beta)$ and $v(\alpha,\beta)$, we choose these tethers as illustrated in the top row of Figure \ref{fig:tethers}.

\begin{figure}[t]
\centering
\includegraphics[scale=0.4]{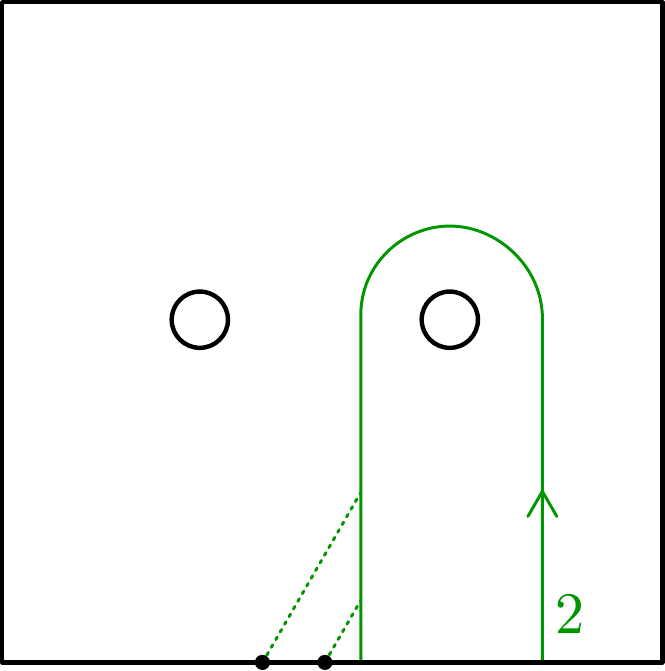}
\includegraphics[scale=0.4]{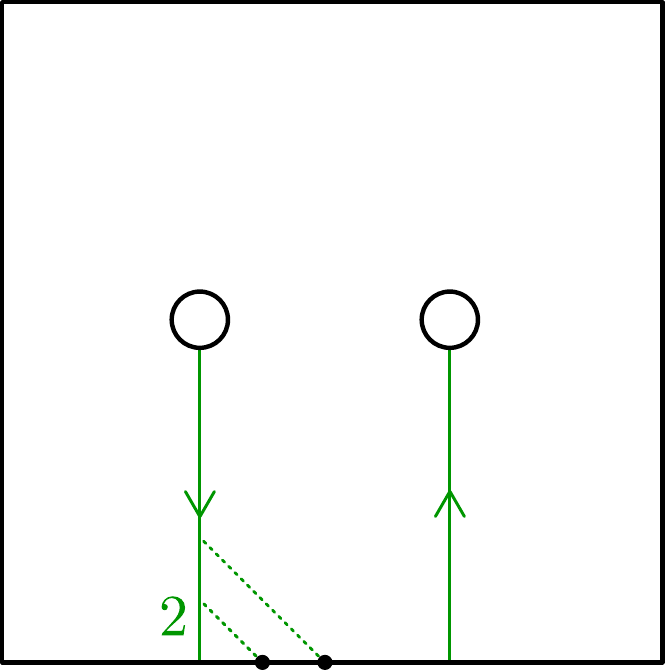}
\includegraphics[scale=0.4]{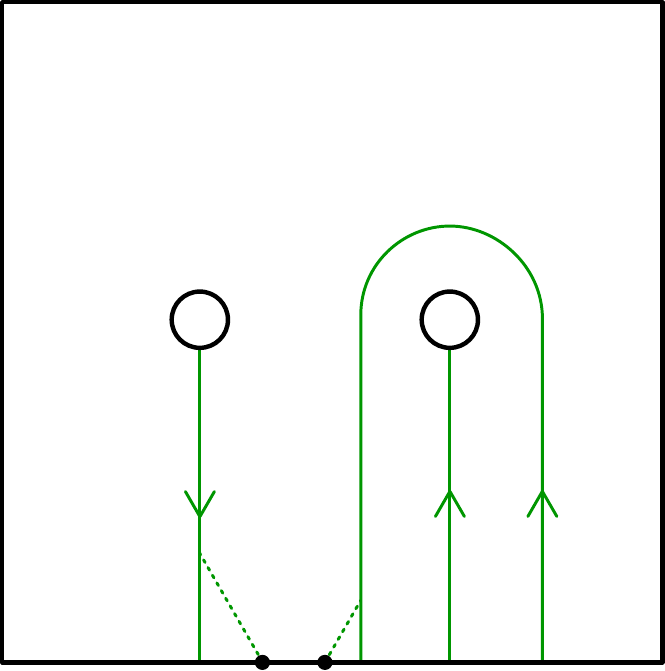}
\\
\includegraphics[scale=0.4]{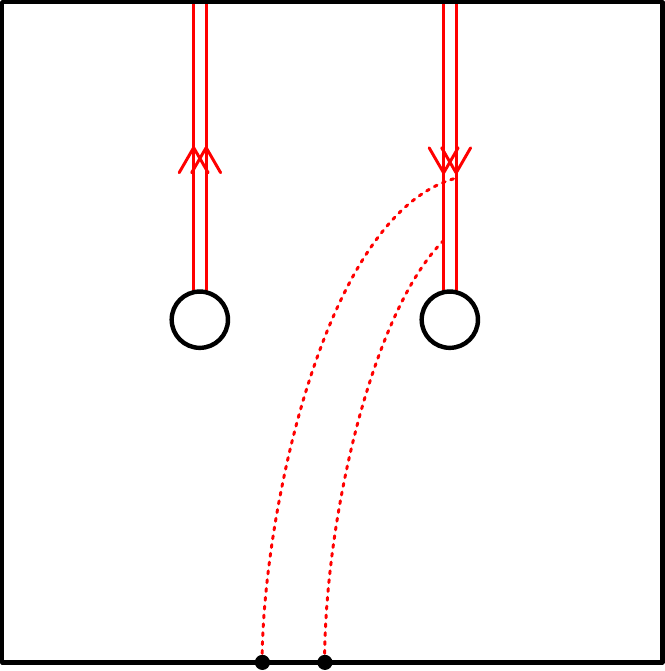}
\includegraphics[scale=0.4]{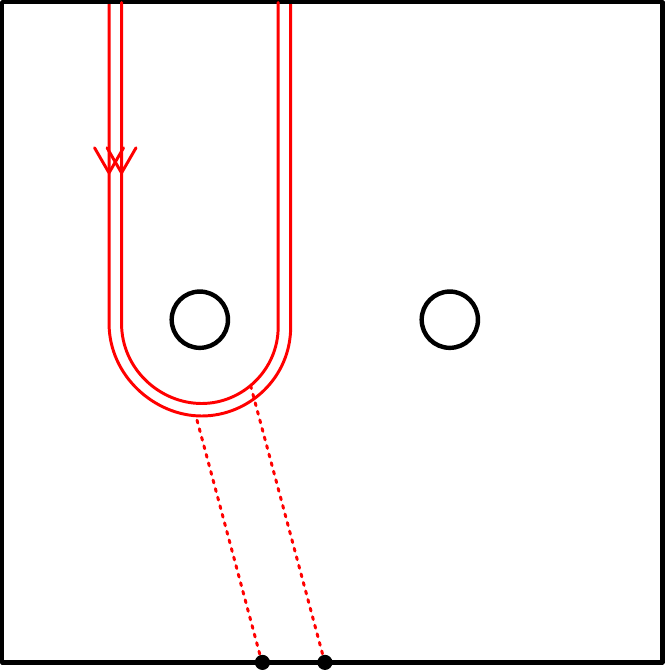}
\includegraphics[scale=0.4]{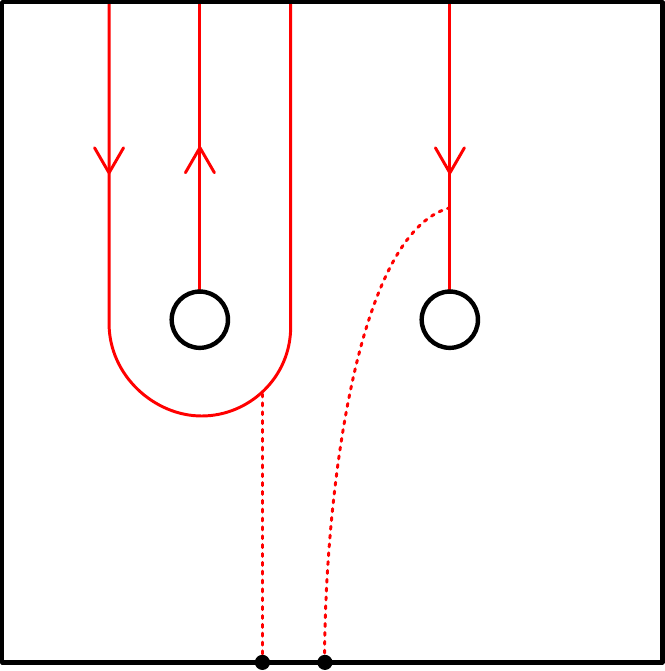}
\caption{Tethers.}
\label{fig:tethers}
\end{figure}

By Poincar{\'e} duality, and the fact that $\mathcal{C}_2(\Sigma)$ is a connected, oriented $4$-manifold with boundary $\mathcal{C}_2(\Sigma,\partial\Sigma) = \{ c \in \mathcal{C}_2(\Sigma) \mid c \cap \partial\Sigma \neq \varnothing \}$, we have a non-degenerate pairing
\begin{equation}
\label{eq:pairing}
\langle -,- \rangle \colon H_2^{BM}(\mathcal{C}_2(\Sigma),\partial^-;\Z[\Heis]) \otimes H_2(\mathcal{C}_2(\Sigma),\partial^+;\Z[\Heis]) \longrightarrow \Z[\Heis],
\end{equation}
where $\partial^\pm$ is an abbreviation of $\mathcal{C}_2(\Sigma,\partial^{\pm}(\Sigma))$, and we note that the boundary $\partial\mathcal{C}_2(\Sigma) = \mathcal{C}_2(\Sigma,\partial\Sigma)$ decomposes as $\partial^+ \cup \partial^-$, corresponding to the decomposition of the boundary of the surface $\partial\Sigma = \partial^+(\Sigma) \cup \partial^-(\Sigma)$. (Formally, it is a \emph{manifold triad}.) The above pairing is linear in the first variable and antilinear in the second one, where we use the anti-involution on the Heisenberg group ring that extends the inverse map. Similarly to the standard case, under transversality hypotheses, the pairing is given by an intersection formula that counts, with signs, the geometric intersections in the Heisenberg cover of a smooth cycle $S$ with all of the $\Heis$-translated copies of a smooth cycle $T$:
\begin{equation}
\label{eq:intersection-formula}
\langle [S],[T] \rangle = \sum_{h \in \Heis}(S\,.\,Th)\,h .
\end{equation}
There are natural elements of $H_2(\mathcal{C}_2(\Sigma),\partial^+;\Z[\Heis])$ that are dual to $w(\alpha)$, $w(\beta)$ and $v(\alpha,\beta)$ with respect to this pairing, which we denote by $\overline{w}(\alpha')$, $\overline{w}(\beta')$ and $v(\alpha',\beta')$ respectively. The element $v(\alpha',\beta')$ is defined exactly as above: it is given by the subspace of $2$-point configurations where one point lies on each of the arcs $\alpha'$ and $\beta'$ of Figure \ref{fig:ab}. The element $\overline{w}(\alpha')$ is defined as follows: first replace the arc $\alpha'$ with two parallel copies $\alpha'_1$ and $\alpha'_2$ (as in the bottom-left of Figure \ref{fig:tethers}), and then $\overline{w}(\alpha')$ is given by the subspace of $2$-point configurations where one point lies on each of $\alpha'_1$ and $\alpha'_2$. The element $\overline{w}(\beta')$ is defined exactly analogously. Again, in order to specify these elements precisely, we have to choose tethers; the choices that we make are illustrated in the bottom row of Figure \ref{fig:tethers}.

A practical description of the pairing \eqref{eq:pairing} is as follows. Let $x = w(\gamma)$ or $v(\gamma,\delta)$ for disjoint arcs $\gamma$, $\delta$ with endpoints on $\partial^-(\Sigma)$, and choose a tether for $x$, namely a path $t_x$ from $c_0$ to a point in $x$. Similarly, let $y = \overline{w}(\epsilon)$ or $v(\epsilon,\zeta)$ for disjoint arcs $\epsilon$, $\zeta$ with endpoints on $\partial^+(\Sigma)$, and choose a tether $t_y$ for $y$. Suppose that the arcs $\gamma \sqcup \delta$ intersect the arcs $\epsilon \sqcup \zeta$ transversely. Then the pairing \eqref{eq:pairing} is given by the formula
\begin{equation}
\label{eq:pairing-formula}
\bigl\langle [x,t_x],[y,t_y] \bigr\rangle = \sum_{p=\{p_1,p_2\} \in x \cap y} \mathrm{sgn}(p_1) . \mathrm{sgn}(p_2) . \mathrm{sgn}(\ell_p) . \phi(\ell_p),
\end{equation}
where $\ell_p \in \mathbb{B}_2(\Sigma)$ is the loop in $\mathcal{C}_2(\Sigma)$ given by concatenating:
\begin{itemize}
\item the tether $t_x$ from $c_0$ to a point in $x$,
\item a path in $x$ to the intersection point $p$,
\item a path in $y$ from $p$ to the endpoint of the tether $t_y$,
\item the reverse of the tether $t_y$ back to $c_0$,
\end{itemize}
the Heisenberg evaluation $\phi(\ell_p)$ of this loop (see Corollary \ref{hom_phi}) detects the contributing translation in the formula \eqref{eq:intersection-formula}, $\mathrm{sgn}(\ell_p) \in \{+1,-1\}$ denotes the sign of the induced permutation in $\mathfrak{S}_2$ and $\mathrm{sgn}(p_i) \in \{+1,-1\}$ is given by the sign convention in Figure \ref{fig:intersection-sign}. 

(In fact, there should be an extra global $-1$ sign on the right-hand side of \eqref{eq:pairing-formula}, which we have suppressed for simplicity. Thus \eqref{eq:pairing-formula} is really a formula for $-\eqref{eq:pairing}$. This global sign ambiguity does not affect our calculations, since all we need is \emph{a} non-degenerate pairing of the form \eqref{eq:pairing}, and any non-degenerate pairing multiplied by a unit is again a non-degenerate pairing. This extra global sign also appears in Bigelow's formula \cite[page 475, ten lines above Lemma 2.1]{Bigelow2001}. See \hyperref[appendixA]{Appendix A} for further explanations of these signs.)

With this description of \eqref{eq:pairing}, it is easy to verify that the matrix
\[
\left(
\begin{matrix}
\langle [w(\alpha)] , [\overline{w}(\alpha')] \rangle &
\langle [w(\alpha)] , [\overline{w}(\beta')] \rangle &
\langle [w(\alpha)] , [v(\alpha',\beta')] \rangle \\
\langle [w(\beta)] , [\overline{w}(\alpha')] \rangle &
\langle [w(\beta)] , [\overline{w}(\beta')] \rangle &
\langle [w(\beta)] , [v(\alpha',\beta')] \rangle \\
\langle [v(\alpha,\beta)] , [\overline{w}(\alpha')] \rangle &
\langle [v(\alpha,\beta)] , [\overline{w}(\beta')] \rangle &
\langle [v(\alpha,\beta)] , [v(\alpha',\beta')] \rangle
\end{matrix}
\right) \;\in\; \mathrm{Mat}_{3,3}(\Z[\Heis])
\]
is the identity; this is the precise sense in which these two $3$-tuples of elements are ``dual'' to each other.\footnote{Since we know that $w(\alpha)$, $w(\beta)$, $v(\alpha,\beta)$ form a basis for the $\Z[\Heis]$-module $H_2^{BM}(\mathcal{C}_2(\Sigma),\partial^-;\Z[\Heis])$, it follows that the elements $\overline{w}(\alpha')$, $\overline{w}(\beta')$, $v(\alpha',\beta')$ are $\Z[\Heis]$-linearly independent in the $\Z[\Heis]$-module $H_2(\mathcal{C}_2(\Sigma),\partial^+;\Z[\Heis])$. In fact, they form a basis for this $\Z[\Heis]$-module (which is therefore free), by Poincar{\'e} duality and the compactly-supported cohomology version of Theorem \ref{basis}.}

\begin{figure}[t]
\centering
\includegraphics[scale=0.4]{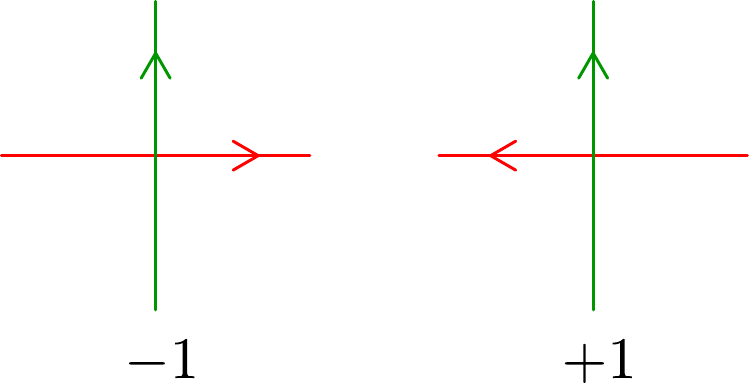}
\caption{Sign convention for intersections between cycles representing elements of the homology groups \textcolor{cycleBM}{$H_n^{BM}(\mathcal{C}_n(\Sigma),\partial^-;\Z[\Heis])$ (vertical arcs)} and \textcolor{cycleD}{$H_n(\mathcal{C}_n(\Sigma),\partial^+;\Z[\Heis])$ (horizontal arcs)}.}
\label{fig:intersection-sign}
\end{figure}

\begin{theorem}
\label{thm:MaMb}
With respect to the ordered basis $(w(\alpha) , w(\beta) , v(\alpha,\beta))$:

\textup{(a)} The matrix for the isomorphism
\[
\mathcal{T}_a=\mathcal{C}_2(T_a)_* \colon H_2^{BM}(\mathcal{C}_{2}(\Sigma),\partial^-;\Z[\Heis])_{(T_a^{-1})_\Heis} \longrightarrow H_2^{BM}(\mathcal{C}_{2}(\Sigma),\partial^-;\Z[\Heis])
\]
is
\[
M_a=\left(\begin{array}{rrr}
1 &  u^{2} a^{ 2 }b^{ -2 } & (u^{-1}-1)ab^{ -1 }\\
0 & 1 & 0 \\
0 & -ab^{ -1 }& 1
\end{array}\right)\ .
\]

\textup{(b)} The matrix for the isomorphism
\[
\mathcal{T}_b=\mathcal{C}_2(T_b)_* \colon H_2^{BM}(\mathcal{C}_{2}(\Sigma),\partial^-;\Z[\Heis])_{(T_b^{-1})_\Heis} \longrightarrow H_2^{BM}(\mathcal{C}_{2}(\Sigma),\partial^-;\Z[\Heis])
\]
is
\[
M_b=\left(\begin{array}{rrr}
u^{-2} b^{ -2 } & 0 & 0 \\
-u^{-1}  & 1& 1 - u^{-1}  \\
-u^{-1} b^{-1}& 0 & b^{-1}
\end{array}\right) \ .
\]
\end{theorem}
\begin{proof}
Let us simplify the notation for the basis and the corresponding dual homology classes by
\[
(e_1,e_2,e_3) = (w(\alpha) , w(\beta) , v(\alpha,\beta)) \qquad (e'_1,e'_2,e'_3) = (\overline{w}(\alpha'), \overline{w}(\beta'), v(\alpha',\beta')).
\]
Using the non-degenerate pairing \eqref{eq:pairing} and elementary linear algebra, we have that
\[
\mathcal{C}_2(f)_*(e_i) = \sum_{j=1}^3 e_j . \langle \mathcal{C}_2(f)_*(e_i) , e'_j \rangle
\]
for any $f \in \mathfrak{M}(\Sigma)$. Computing the matrices $M_a$ and $M_b$ therefore consists in computing $\langle \mathcal{T}_a(e_i) , e'_j \rangle$ and $\langle \mathcal{T}_b(e_i) , e'_j \rangle$ for $i,j \in \{1,2,3\}$. We will explain how to compute two of these 18 elements of $\Z[\Heis]$, the remaining 16 being left as exercises for the reader. In each case the idea is the same: apply the Dehn twist to the explicit cycle (described above) representing the homology class $e_i$, and then use the formula \eqref{eq:pairing-formula} to compute the pairing.

We begin by computing $\langle \mathcal{T}_a(e_2) , e'_1 \rangle = \langle \mathcal{T}_a(w(\beta)) , \overline{w}(\alpha') \rangle$, the top-middle entry of $M_a$.
\begin{align*}
\langle \mathcal{T}_a(w(\beta)) , \overline{w}(\alpha') \rangle
&= \langle w(T_a(\beta)) , \overline{w}(\alpha') \rangle \\
&= \raisebox{-0.5\height}{\includegraphics[scale=0.4]{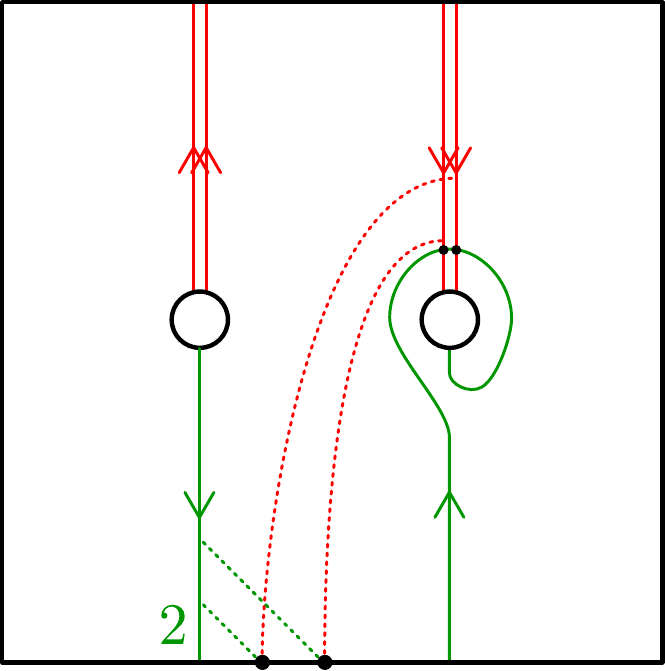}} \\
&= (-1).(-1).(+1).\phi \left(\; \raisebox{-0.48\height}{\includegraphics[scale=0.3]{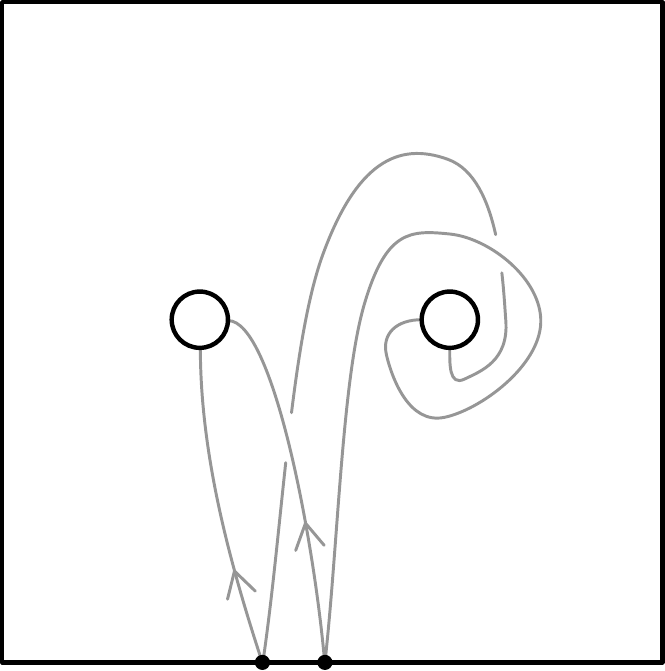}} \;\right) \\
&= \phi(\alpha \beta^{-1} \sigma^{-1} \alpha \beta^{-1} \sigma) \\
&= ab^{-1}ab^{-1} \\
&= u^2 a^{2} b^{-2}.
\end{align*}
Here the grey figure represents the graph of a braid up to vertical isotopy (specifically, the loop $\ell_p$ from \eqref{eq:pairing-formula}), viewed from above, where the braid is moving downwards as we go forwards along the loop. Recall that we concatenate loops from right to left.

We next calculate $\langle \mathcal{T}_a(e_3) , e'_1 \rangle = \langle \mathcal{T}_a(v(\alpha,\beta)) , \overline{w}(\alpha') \rangle$, the top-right entry of $M_a$. This is slightly more complicated, since in this case there are two intersection points in the configuration space $\mathcal{C}_2(\Sigma)$, so we obtain a Heisenberg polynomial (i.e.~element of $\Z[\Heis]$) with two terms.
\begin{align*}
\langle \mathcal{T}_a(v(\alpha,\beta)) , \overline{w}(\alpha') \rangle
&= \langle v(\alpha , T_a(\beta)) , \overline{w}(\alpha') \rangle \\
&= \raisebox{-0.5\height}{\includegraphics[scale=0.4]{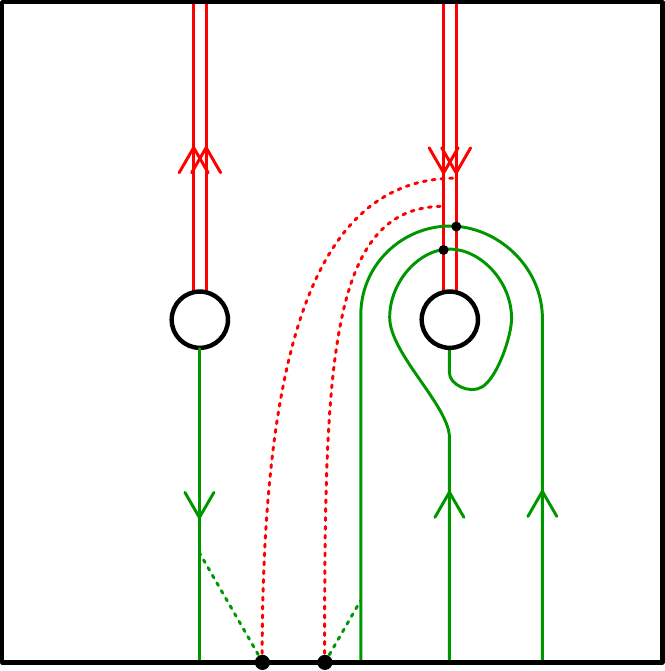}} + \raisebox{-0.5\height}{\includegraphics[scale=0.4]{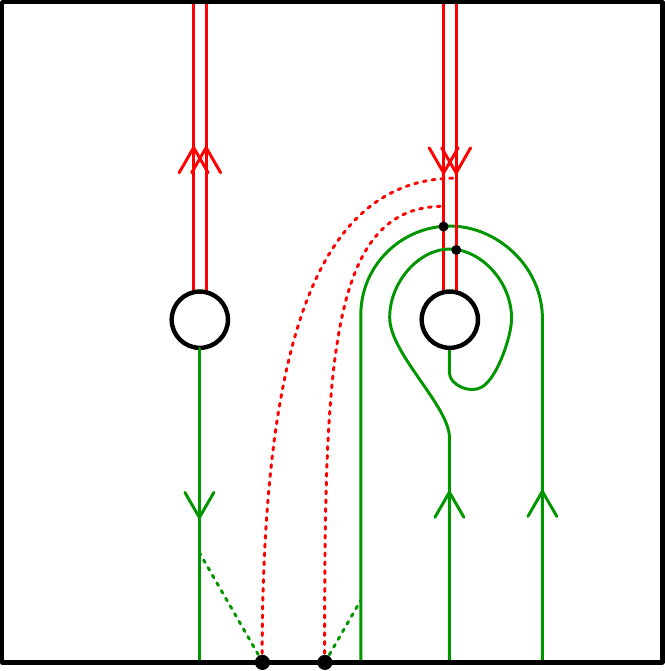}} \\
&= (-1).(+1).(-1).\phi \left(\; \raisebox{-0.48\height}{\includegraphics[scale=0.3]{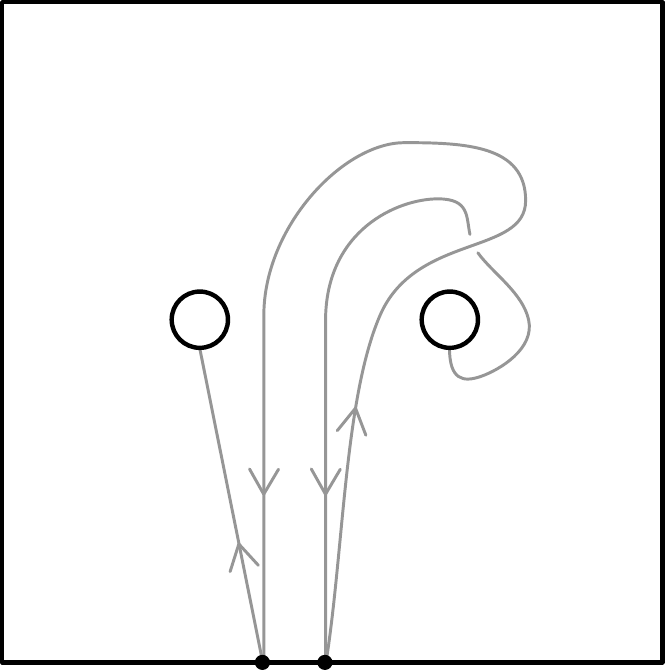}} \;\right) \\
& \qquad\qquad\qquad +(+1).(-1).(+1).\phi \left(\; \raisebox{-0.48\height}{\includegraphics[scale=0.3]{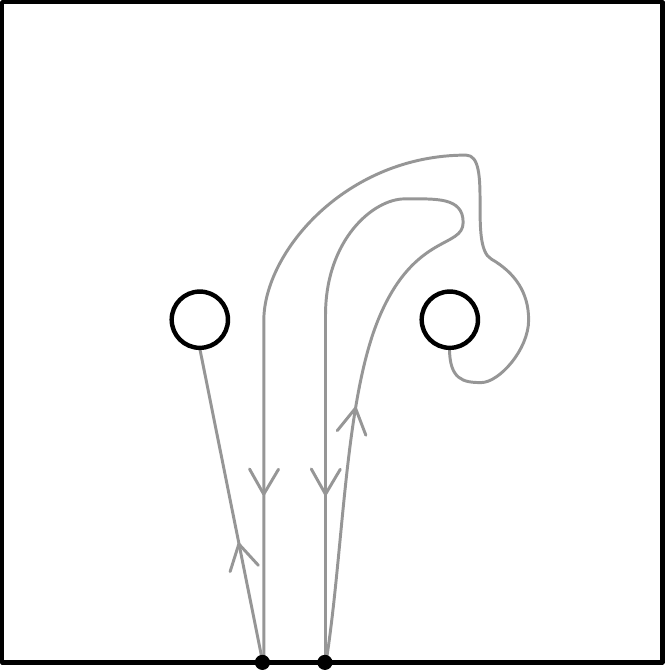}} \;\right) \\
&= \phi(\sigma^{-1} \alpha \beta^{-1}) - \phi(\alpha \beta^{-1}) \\
&= u^{-1}ab^{-1} - ab^{-1} \\
&= (u^{-1}-1)ab^{-1}.
\end{align*}

The other 16 entries of the matrices $M_a$ and $M_b$ may be computed analogously.
\end{proof}

\begin{notation}
To shorten the notation in the following, we will use the abbreviation
\[
A \coloneqq H_2^{BM}(\mathcal{C}_{2}(\Sigma),\mathcal{C}_{2}(\Sigma,\partial^-(\Sigma)) ;\Z[\Heis]) = H_2^{BM}(\mathcal{C}_{2}(\Sigma),\partial^-;\Z[\Heis]).
\]
\end{notation}

\begin{remark}[\emph{Verifying the braid relation.}]
Recall that $\mathfrak{M}(\Sigma_{1,1})$ is generated by $T_a$ and $T_b$ subject to the single relation $T_a T_b T_a = T_b T_a T_b$. It must therefore be the case that the isomorphism
\[\begin{tikzcd}
A_{(T_a T_b T_a)^{-1}_\Heis} \ar[rr," (\mathcal{T}_a)_{(T_a T_b)_\Heis^{-1}}"] && A _{(T_a T_b)_\Heis^{-1} }\ar[rr," (\mathcal{T}_b)_{(T_a)_\Heis^{-1}}"] && A_{(T_a)_\Heis^{-1}} \ar[r,"\mathcal{T}_a"] & A
\end{tikzcd}\]
is equal to the isomorphism
\[\begin{tikzcd}
A_{(T_b T_a T_b)^{-1}_\Heis} \ar[rr," (\mathcal{T}_b)_{(T_b T_a)_\Heis^{-1}}"] && A _{(T_b T_a)_\Heis^{-1} }\ar[rr," (\mathcal{T}_a)_{(T_b)_\Heis^{-1}}"] && A_{(T_b)_\Heis^{-1}} \ar[r,"\mathcal{T}_b"] & A
\end{tikzcd}\]
in other words, using Lemma \ref{lem:matrix}, we must have the following equality of matrices:
\begin{equation}
\label{eq:braid-relation}
M_a . (T_a)_\Heis (M_b) . (T_a T_b)_\Heis (M_a) = M_b . (T_b)_\Heis (M_a) . (T_b T_a)_\Heis (M_b),
\end{equation}
where $M_a$ and $M_b$ are as in Theorem \ref{thm:MaMb} and the automorphisms $(T_a)_\Heis , (T_b)_\Heis \in \mathrm{Aut}(\Heis)$ are extended linearly to automorphisms of $\Z[\Heis]$ and thus to automorphisms of matrices over $\Z[\Heis]$. Indeed, one may calculate that both sides of \eqref{eq:braid-relation} are equal to
\begin{equation}
\label{eq:action-of-aba}
\left(\begin{matrix}
0 & u^2 a^{2} b^{-2} & 0 \\
-u^{-1} & 1 + (u^{-3}-u^{-2})a - u^{-5}a^{2} & (1-u^{-1})(1+u^{-3}a) \\
0 & -ab^{-1} - u^{-1}a^{2}b^{-1} & u^{-1}ab^{-1}
\end{matrix}\right) .
\end{equation}
\end{remark}

\begin{remark}[\emph{The Dehn twist around the boundary.}]
\label{rmk:Dehn-twist-boundary}
In a similar way, we may compute the matrix $M_\partial$ for the action $\mathcal{T}_\partial$ of the Dehn twist $T_\partial$ around the boundary of $\Sigma_{1,1}$. We note that $T_\partial$ lies in the Chillingworth subgroup of $\mathfrak{M}(\Sigma_{1,1})$, so its action on $\Heis$ is trivial and the action $\mathcal{T}_\partial$ is an \emph{automorphism}
\[
\mathcal{T}_\partial \colon A \longrightarrow A.
\]
However, to compute its matrix $M_\partial$, it is convenient to decompose $\mathcal{T}_\partial$ into isomorphisms as follows. By the $2$-chain relation \cite[Proposition~4.12]{FarbMargalit}, the Dehn twist $T_\partial$ decomposes as $T_\partial = (T_a T_b)^6$. If we write $s = T_a T_b T_a = T_b T_a T_b$, this becomes $T_\partial = s^4$. Then $\mathcal{T}_\partial$ decomposes as
\[\begin{tikzcd}
A=A_{s^{-4}_\Heis} \ar[rr,"(\mathcal{T}_{s})_{s^{-3}_\Heis}"] && A_{s^{-3}_\Heis} \ar[rr,"(\mathcal{T}_{s})_{s^{-2}_\Heis}"] && A_{s^{-2}_\Heis} \ar[rr," (\mathcal{T}_s)_{s_\Heis^{-1}}"] && A_{s_\Heis^{-1}  } \ar[r,"\mathcal{T}_s"] & A
\end{tikzcd}\]
where $\mathcal{T}_s$ denotes the action of $s$, given by the matrix \eqref{eq:action-of-aba} above. The matrix $M_\partial$ may therefore be obtained by multiplying together four copies of \eqref{eq:action-of-aba}, shifted by the actions of $\id$, $s_\Heis$, $s^2_\Heis$ and $s^3_\Heis$ respectively. This may be implemented in Sage to show that $M_\partial$ is equal to the matrix displayed in Figure \ref{fig:bigmatrix}. More details of these Sage calculations are given in \hyperref[appendixC]{Appendix C}.

One may verify explicitly by hand that, if we set $a=b=u^2=1$ in the matrix $M_\partial = (\text{Figure \ref{fig:bigmatrix}})$, it simplifies to the identity matrix. This is expected, since applying this specialisation to our representation recovers the second Moriyama representation (as discussed in \S\ref{relation-to-Moriyama}; see in particular the quotient \eqref{eq:Heisenberg-to-Moriyama} of $\mathfrak{M}(\Sigma)$-representations), whose kernel is the Johnson kernel $\mathfrak{J}(2)$ by \cite{Moriyama}, which contains $T_\partial$.
\end{remark}

\begin{figure}[t]
\centering
\begin{adjustwidth}{-2.2cm}{-2.2cm}
{\small
\[
\left(
\begin{matrix}
  \begin{smallmatrix}
   u^{-8} b^{ -2 } + u^{-4} a^{ 2 } -u a^{ 2 }b^{ -2 } +\\
   ( u^{-1} - u^{-2} )a^{ 2 }b^{ -1 }+
   ( u^{-3} - u^{-4} )ab^{ -2 } +\\
   ( u^{-4} - u^{-5} )ab^{ -1 }
  \end{smallmatrix}
&
  \begin{smallmatrix}
  (u^2 + 1 - 2u^{-1} + u^{-2} + u^{-4})a^2b^{-2}-ua^2b^{-4}+\\
  (-u^2 + u + u^{-1} - u^{-2})a^2b^{-3}-u^{-3}a^2+\\
  (-1 + u^{-1} + u^{-3} - u^{-4})a^2b^{-1}
  \end{smallmatrix}
&
  \begin{smallmatrix}
  (-1 + 2u^{-1} - u^{-2} - u^{-4} + u^{-5})a^2b^{-1}+\\
  (u - 1)a^2b^{-3}+(u^2 - u - u^{-1} + 2u^{-2} - u^{-3})a^2b^{-2}+\\
  (-u^{-3} + u^{-4})ab^{-1}+(u^{-4} - u^{-5})ab^{-3}+\\
  (-u^{-2} + u^{-3} + u^{-5} - u^{-6})ab^{-2}+\\
  (-u^{-3} + u^{-4})a^2
  \end{smallmatrix}
\\[10ex]
  \begin{smallmatrix}
  -u^{-1} - u^{-3} + 2 u^{-4} - u^{-5} - u^{-7} + u^{-2} a^{ -2 }+ \\
  ( u^{-1} - u^{-2} - u^{-4} + u^{-5} )a^{ -1 } + u^{-6} a^{ 2 }+\\
   ( u^{-3} - u^{-4} - u^{-6} + u^{-7} )a
  \end{smallmatrix}
&
  \begin{smallmatrix}
  (1 + u^{-2} - u^{-3} + u^{-6})+u^{-6}a^2b^{-2}-u^{-1}b^{-2}+\\
  (u^{-3} - u^{-4})ab^{-2}+(-1 + u^{-1} + u^{-3} - u^{-4})b^{-1}+\\
  (u^{-2} - 2u^{-3} + u^{-4} + u^{-6} - u^{-7})ab^{-1}-u^{-5}a^2+\\
  (-u^{-2} + u^{-3} + u^{-5} - u^{-6})a+(u^{-5} - u^{-6})a^2b^{-1}
  \end{smallmatrix}
&
  \begin{smallmatrix}
  (-u^{-6} + u^{-7})a^2b^{-1}+\\
  (u^{-1} - u^{-2} - u^{-4} + 2u^{-5} - u^{-6})b^{-1}+\\
  (-u^{-3} + 2u^{-4} - u^{-5} - u^{-7} + u^{-8})ab^{-1}+\\
  1 - u^{-1} + u^{-2} - 3u^{-3} + 2u^{-4} + u^{-6} - u^{-7}+\\
  (-u^{-2} + 2u^{-3} - u^{-4} + u^{-5} - 2u^{-6} + u^{-7})a+\\
  (u^{-2} - u^{-3})a^{-1}b^{-1}+(-1 + u^{-1} + u^{-3} - u^{-4})a^{-1}+\\
  (-u^{-5} + u^{-6})a^2
  \end{smallmatrix}
\\[10ex]
  \begin{smallmatrix}
  -u^{-6} a^{ -1 }b^{ -1 } +( -u^{-3} +  u^{-4}   - u^{-7} )b^{ -1 } -u^{-4} +\\
   ( u^{-1} - u^{-4} + u^{-5} )ab^{ -1 } + u^{-2} a^{ 2 }b^{ -1 }+ \\
   ( -u^{-3} + u^{-6} )a + u^{-5} a^{ 2 }
  \end{smallmatrix}
&
  \begin{smallmatrix}
  (-1 - u^{-2} + 2u^{-3} - u^{-6})ab^{-1}+u^{-1}ab^{-3}+\\
  u^{-2}a^2b^{-3} +(1 - u^{-1} - u^{-3} + u^{-4})ab^{-2}+\\
  (u^{-1 }- u^{-2} + u^{-5})a^2b^{-2}+\\
  (-u^{-1} + u^{-4} - u^{-5})a^2b^{-1}+(u^{-2} - u^{-5})a-u^{-4}a^2
  \end{smallmatrix}
&
  \begin{smallmatrix}
  u^{-3}+(u^{-2}- u^{-3} - u^{-5} + u^{-6})a +\\
  (-u^{-1} + u^{-2} - u^{-5} + u^{-6})ab^{-2} +
  (-u^{-2} + u^{-3})a^2b^{-2}+\\
  (-1 + u^{-1} + 2u^{-3} - 3u^{-4} + u^{-7})ab^{-1} +\\
  (-u^{-1} + u^{-2} - u^{-5} + u^{-6})a^2b^{-1}+(-u^{-4 }+ u^{-5})b^{-2}+\\
  (u^{-2} - u^{-3} - u^{-5} + u^{-6})b^{-1}+(-u^{-4} + u^{-5})a^2
  \end{smallmatrix}
\end{matrix}
\right)
\]
}
\end{adjustwidth}
\caption{The action of the Dehn twist around the boundary of $\Sigma_{1,1}$.}
\label{fig:bigmatrix}
\end{figure}

\subsection{Higher genus}

\begin{figure}[t]
\centering
\includegraphics[scale=0.4]{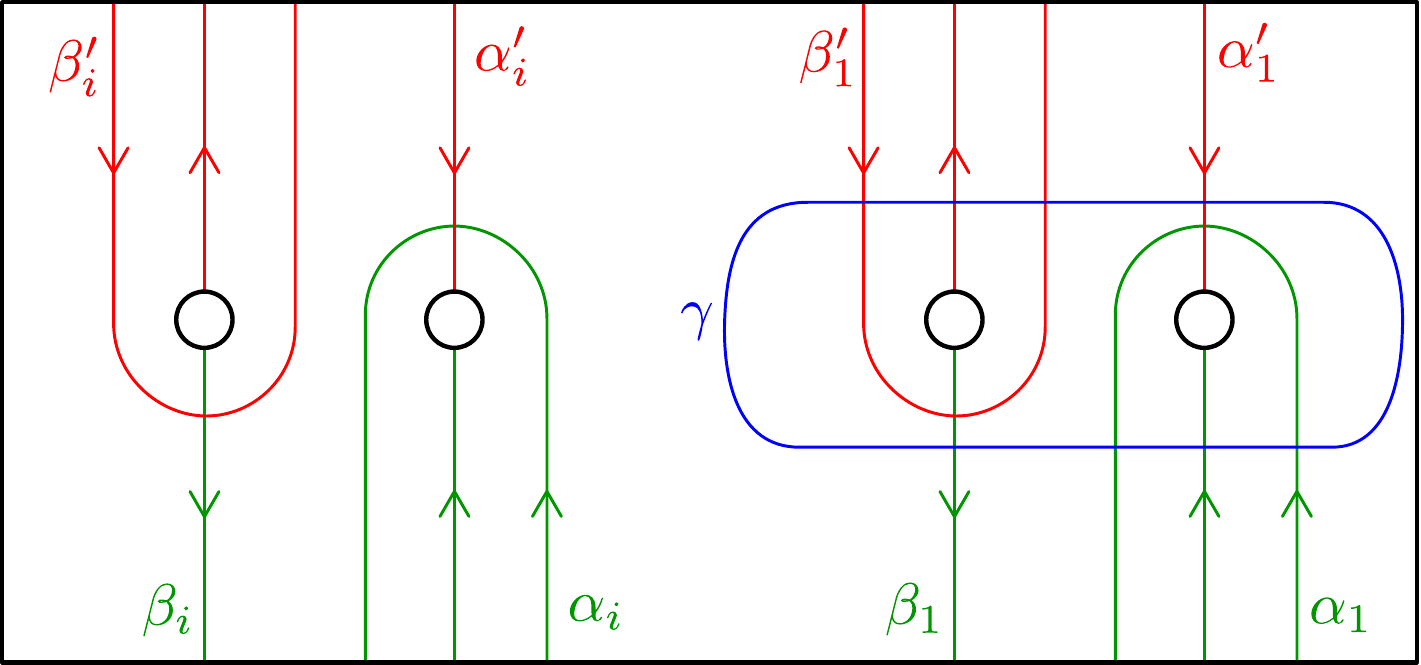}
\caption{The arcs $\alpha_i,\beta_i,\alpha'_i,\beta'_i$ and the closed genus-one-separating curve $\gamma$.}
\label{fig:generators-higher-genus}
\end{figure}

For arbitrary genus $g \geq 1$, we view the surface $\Sigma = \Sigma_{g,1}$ as the quotient of the punctured rectangle depicted in Figure \ref{fig:generators-higher-genus}, where the $2g$ holes are identified in pairs by reflection. The arcs $\alpha_i,\beta_i$ for $i \in \{1,\ldots,g\}$ form a symplectic basis for the first homology of $\Sigma$ relative to the lower edge of the rectangle. By Theorem \ref{basis}, a basis for the free $\Z[\Heis]$-module $H_2^{BM}(\mathcal{C}_2(\Sigma),\mathcal{C}_2(\Sigma,\partial^-(\Sigma));\Z[\Heis])$ is given by the homology classes represented by the $2$-cycles
\begin{itemize}
\item $w(\epsilon)$ for $\epsilon \in \{\alpha_1,\beta_1,\alpha_2,\beta_2,\ldots,\alpha_g,\beta_g\}$,
\item $v(\delta,\epsilon)$ for $\delta,\epsilon \in \{\alpha_1,\beta_1,\alpha_2,\beta_2,\ldots,\alpha_g,\beta_g\}$ with $\delta < \epsilon$
\end{itemize}
where we use the ordering $\alpha_1 < \beta_1 < \alpha_2 < \cdots < \alpha_g < \beta_g$. Here $w(\epsilon)$ denotes the subspace of configurations where both points lie on $\epsilon$ and $v(\delta,\epsilon)$ denotes the subspace of configurations where one point lies on each of $\delta$ and $\epsilon$. As in the genus-$1$ setting, we have to be more careful to specify these elements precisely; this is done by choosing, for each of the $2$-cycles listed above, a path (called a ``tether'') in $\mathcal{C}_2(\Sigma)$ from a point in the cycle to $c_0$, the base configuration, which is contained in the bottom edge of the rectangle. Since the space of configurations of two points in the bottom edge of the rectangle is contractible, it is equivalent to choose a path in $\mathcal{C}_2(\Sigma)$ from a point in the cycle to \emph{any} configuration contained in the bottom edge of the rectangle.

For cycles of the form $w(\epsilon)$, we may choose tethers exactly as in the genus-$1$ setting: see the top-left and top-middle of Figure \ref{fig:tethers}. For cycles of the form $v(\alpha_i,\beta_i)$, we may also choose tethers exactly as in the genus-$1$ setting: see the top-right of Figure \ref{fig:tethers}. For other cycles of the form $v(\delta,\epsilon)$, we choose tethers as illustrated in Figure \ref{fig:tethers-higher-genus}.

\begin{figure}[t]
\centering
\includegraphics[scale=0.3]{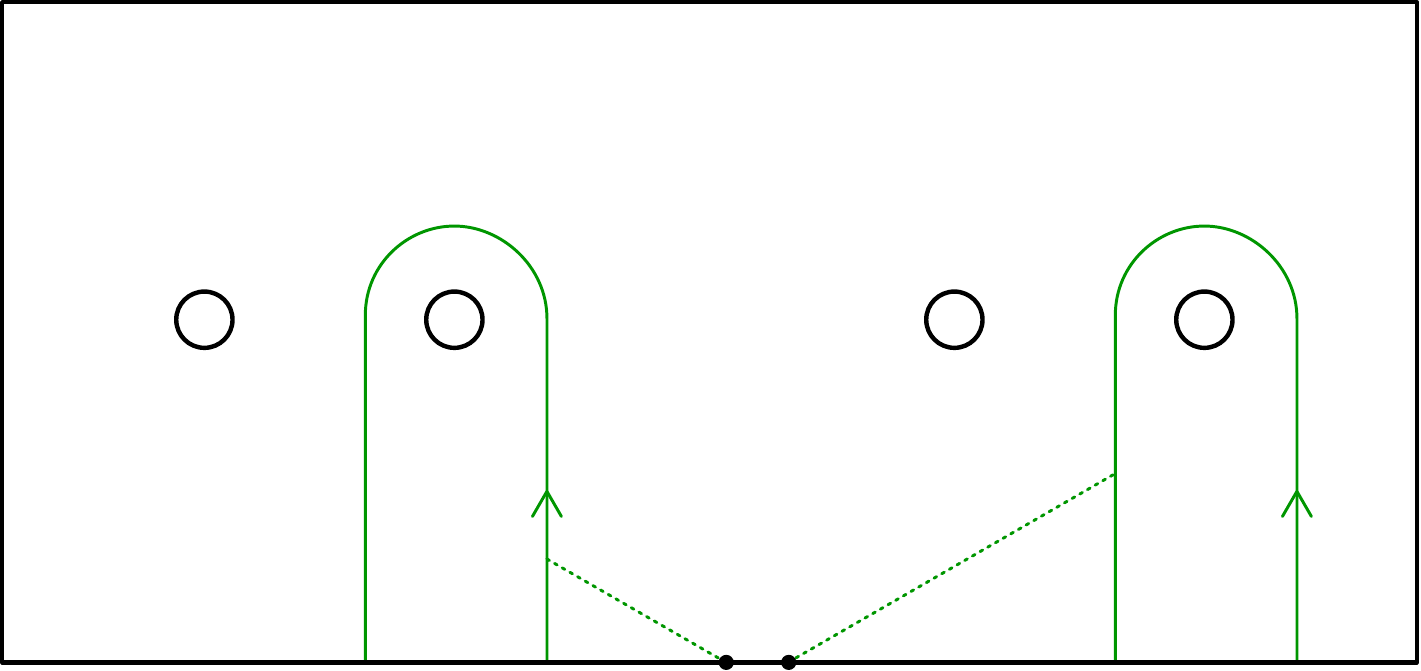}
\includegraphics[scale=0.3]{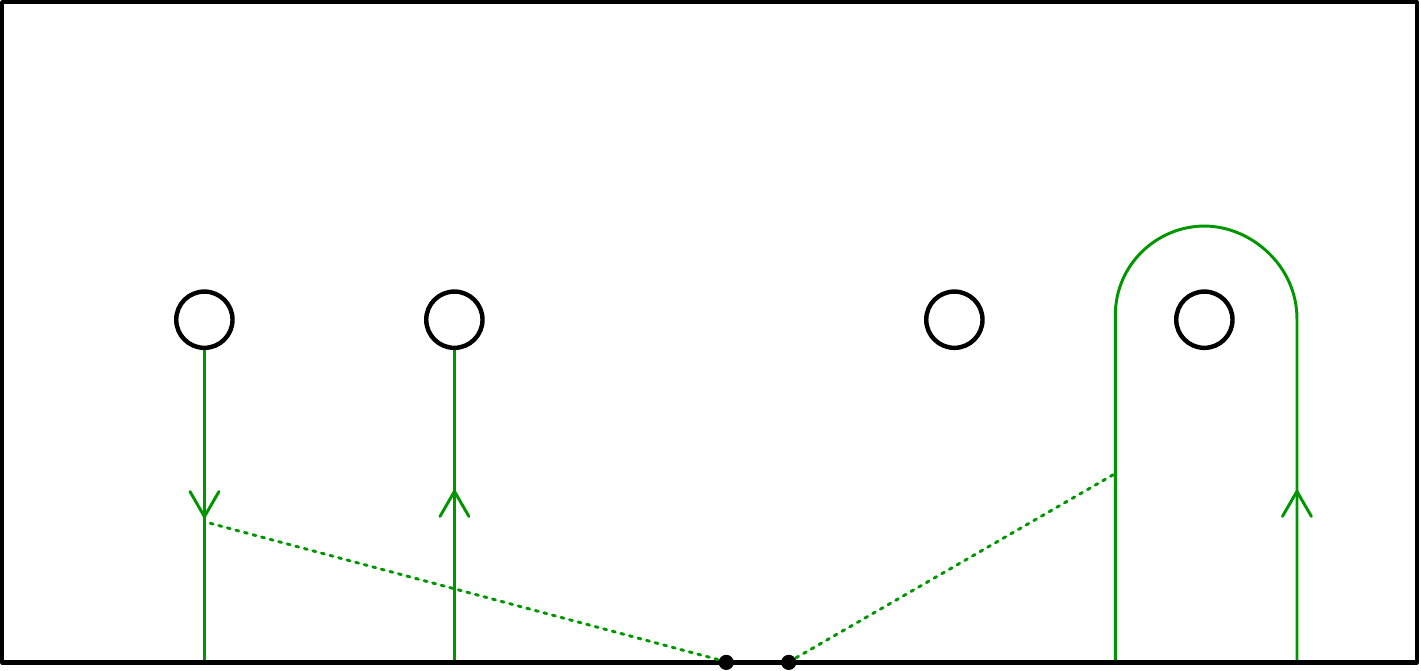}
\\
\includegraphics[scale=0.3]{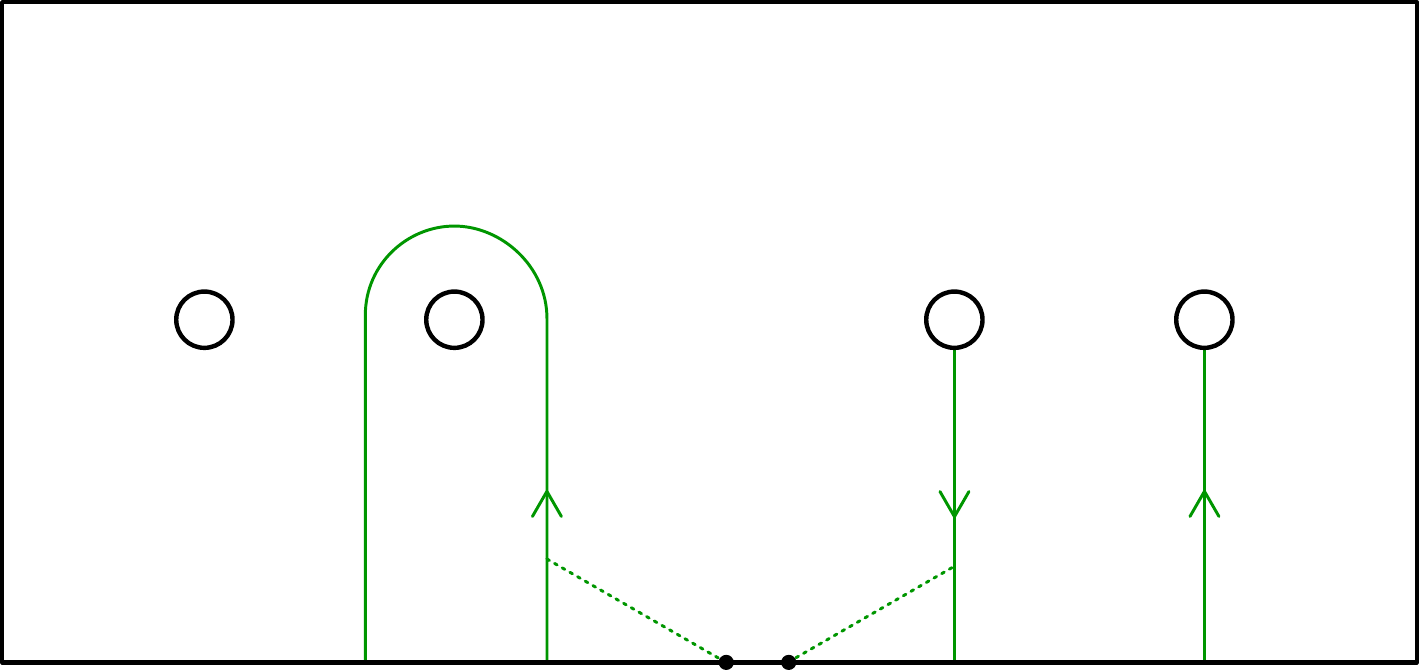}
\includegraphics[scale=0.3]{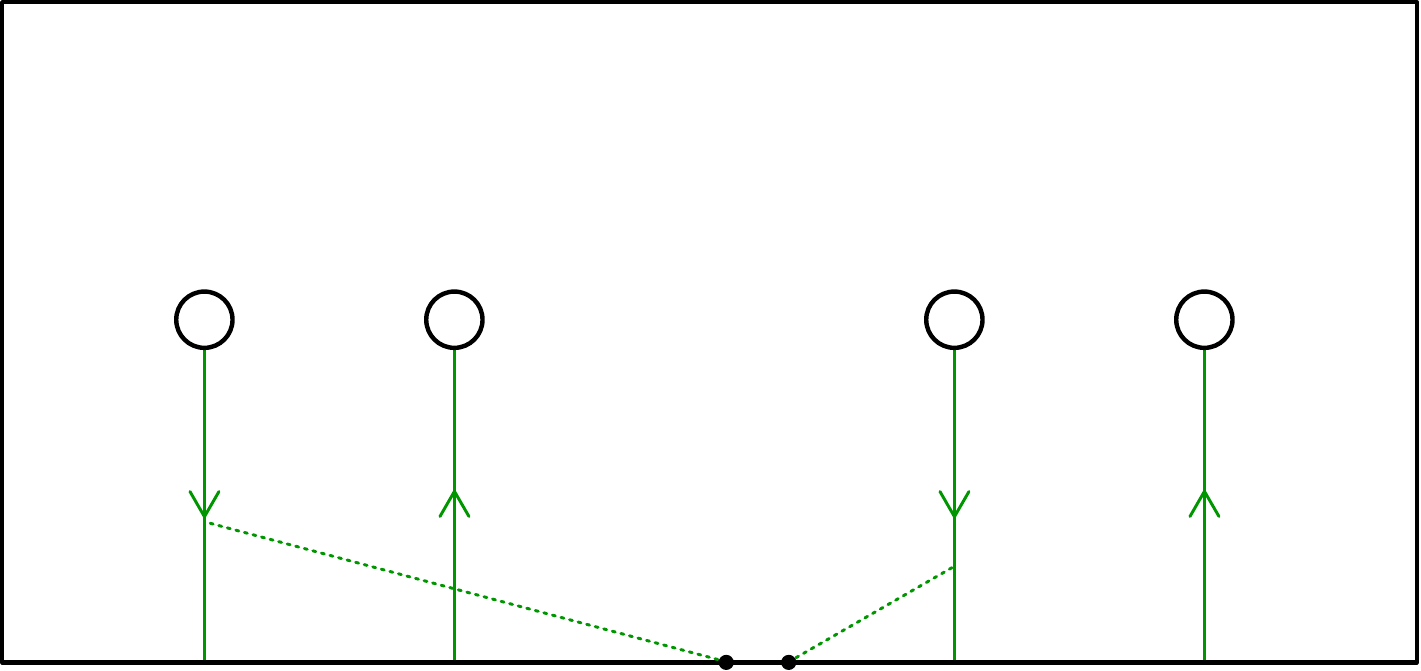}
\caption{More tethers.}
\label{fig:tethers-higher-genus}
\end{figure}

Exactly as in the genus-$1$ setting, there is a non-degenerate pairing \eqref{eq:pairing} defined via Poincar{\'e} duality for the $4$-manifold-with-boundary $\mathcal{C}_2(\Sigma)$. Associated to the collection of arcs $\alpha'_i,\beta'_i$ illustrated in Figure \ref{fig:generators-higher-genus} there are elements of $H_2(\mathcal{C}_2(\Sigma),\partial^+;\Z[\Heis])$:
\begin{itemize}
\item $\overline{w}(\epsilon)$ for $\epsilon \in \{\alpha'_1,\beta'_1,\alpha'_2,\beta'_2,\ldots,\alpha'_g,\beta'_g\}$,
\item $v(\delta,\epsilon)$ for $\delta,\epsilon \in \{\alpha'_1,\beta'_1,\alpha'_2,\beta'_2,\ldots,\alpha'_g,\beta'_g\}$ with $\delta < \epsilon$
\end{itemize}
where we use the ordering $\alpha'_1 < \beta'_1 < \alpha'_2 < \cdots < \alpha'_g < \beta'_g$. Here, $\overline{w}(\epsilon)$ is the subspace of configurations where one point lies on each of $\epsilon^+$ and $\epsilon^-$, where $\epsilon^+,\epsilon^-$ are two parallel, disjoint copies of $\epsilon$. As above, we specify these elements precisely by choosing tethers (paths in $\mathcal{C}_2(\Sigma)$ from a point on the cycle to a configurations contained in the bottom edge of the rectangle). For elements of the form $\overline{w}(\epsilon)$ or $v(\alpha'_i,\beta'_i)$, we choose these exactly as in the genus-$1$ setting; see the bottom row of Figure \ref{fig:tethers}. For other elements of the form $v(\delta,\epsilon)$, we choose them as illustrated in Figure \ref{fig:tethers-higher-genus-prime}.

\begin{figure}[t]
\centering
\includegraphics[scale=0.3]{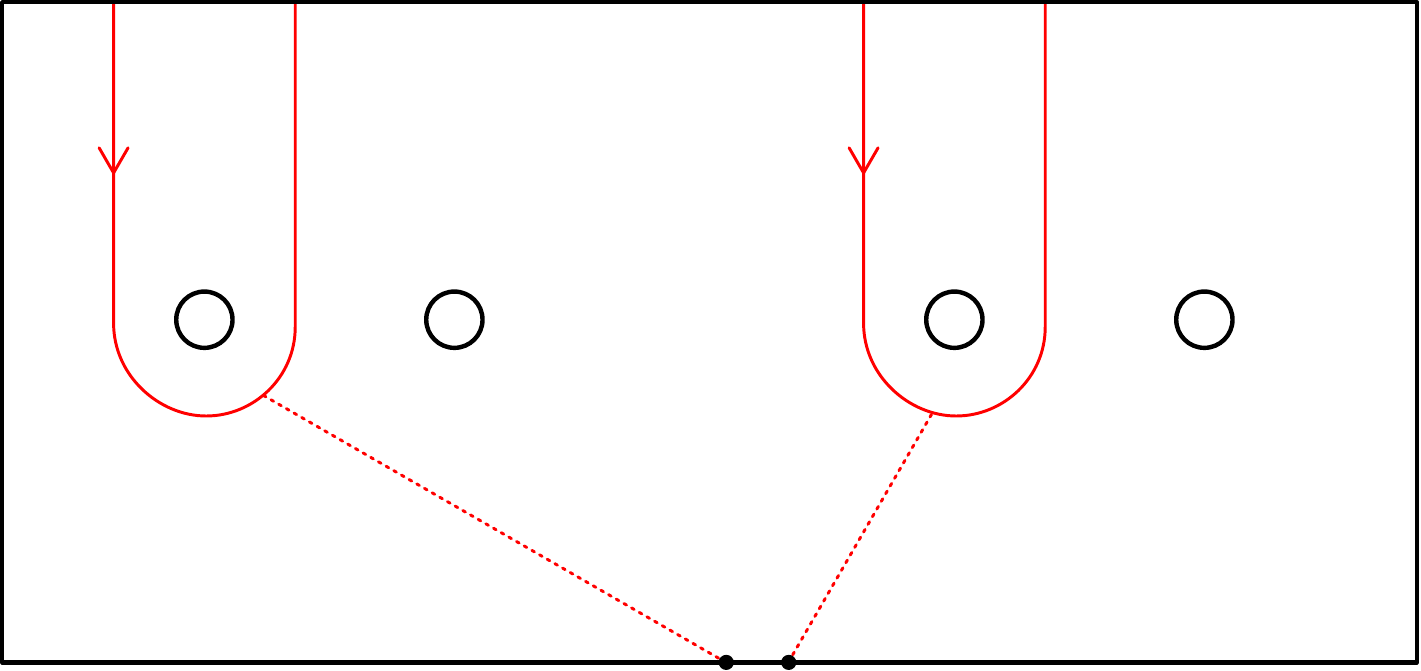}
\includegraphics[scale=0.3]{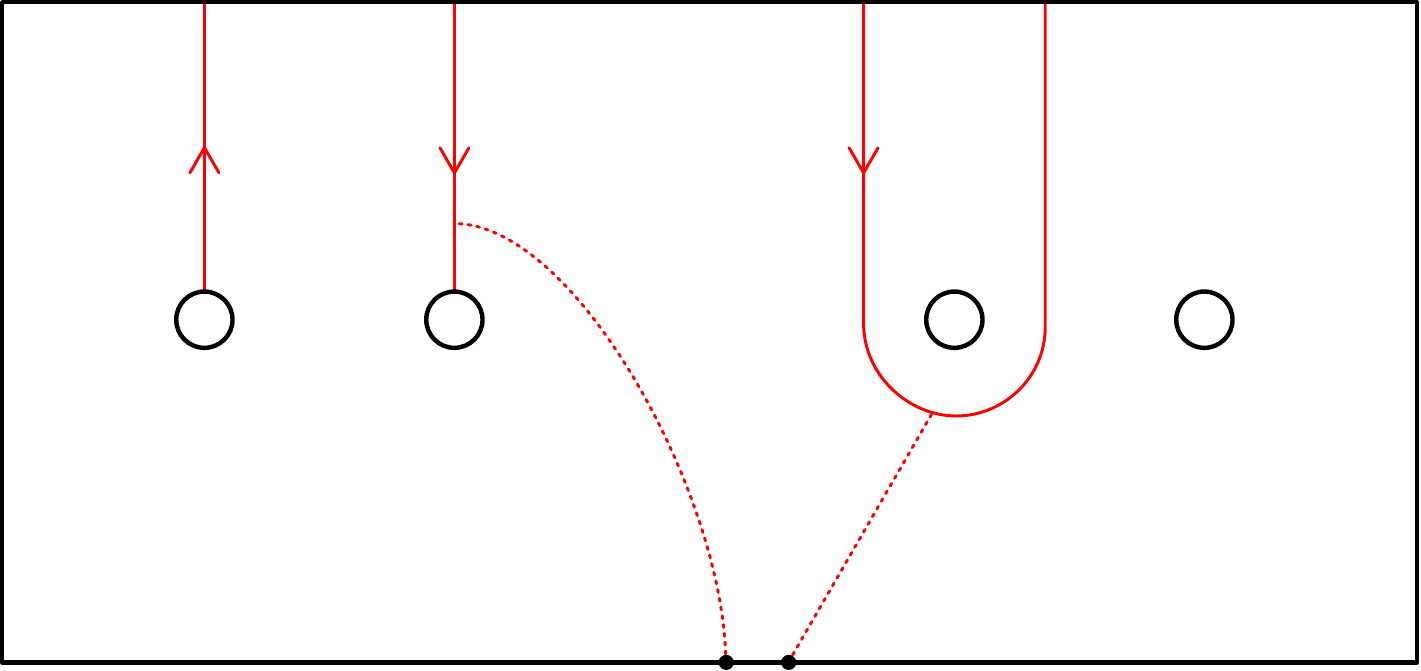}
\\
\includegraphics[scale=0.3]{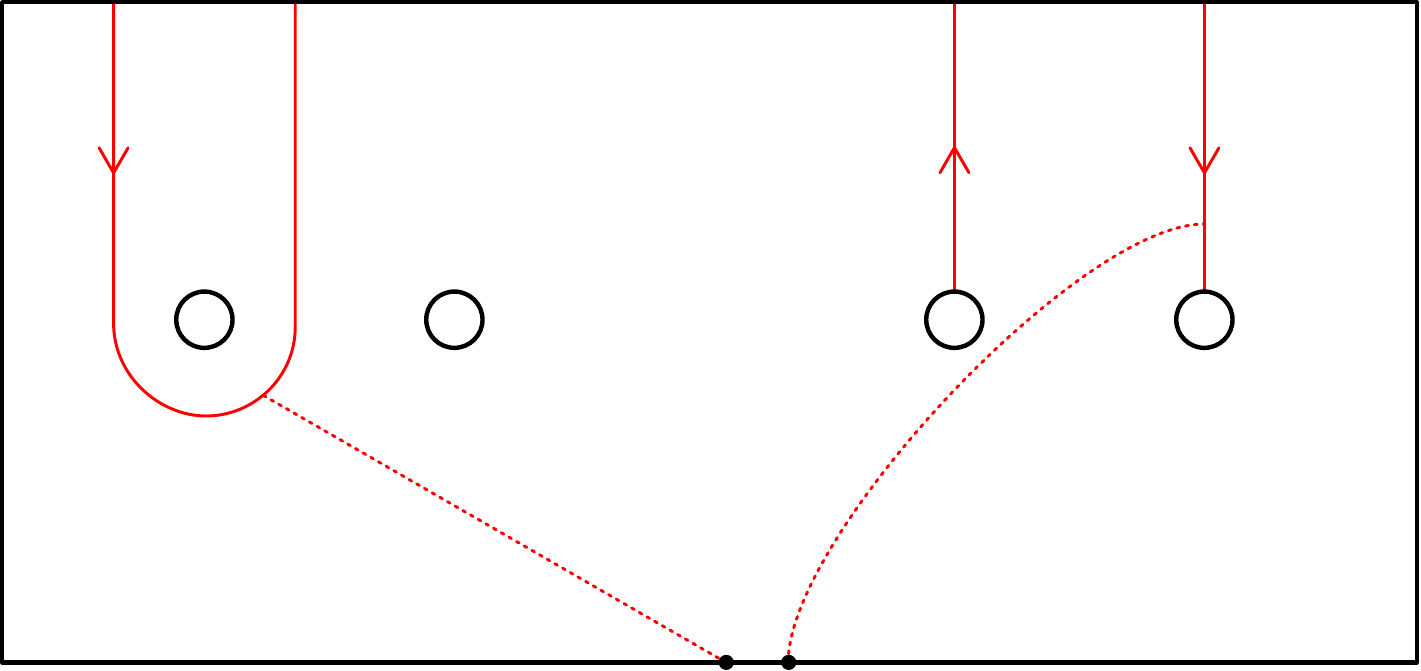}
\includegraphics[scale=0.3]{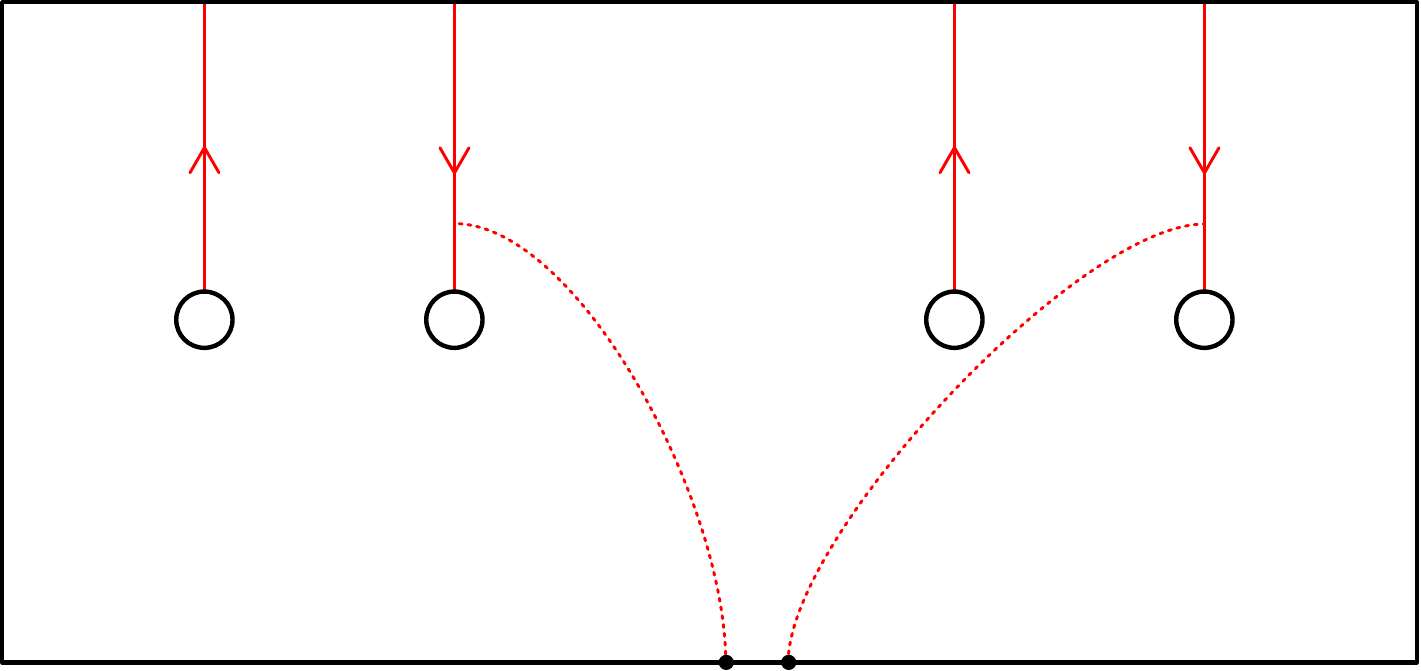}
\caption{Even more tethers.}
\label{fig:tethers-higher-genus-prime}
\end{figure}

\begin{remark}
These choices of tethers may seem a little arbitrary, and indeed they are; however, any different choice would have the effect simply of changing the chosen basis for the Heisenberg homology $H_2^{BM}(\mathcal{C}_2(\Sigma),\partial^-;\Z[\Heis])$ by rescaling each basis vector by a unit of $\Z[\Heis]$. This would have the effect of conjugating the matrices that we calculate by an invertible diagonal matrix.
\end{remark}

The geometric formula \eqref{eq:pairing-formula} for the non-degenerate pairing $\langle - \mathbin{,} - \rangle$ holds exactly as in the genus-$1$ setting, and one may easily verify using this formula that the bases
\begin{align}
\begin{split}
\label{eq:ordered-bases}
\mathcal{B} &= \{ w(\epsilon) , v(\delta,\epsilon) \mid \delta < \epsilon \in \{ \alpha_1,\ldots,\beta_g \} \} \\
\mathcal{B}' &= \{ \overline{w}(\epsilon) , v(\delta,\epsilon) \mid \delta < \epsilon \in \{ \alpha'_1,\ldots,\beta'_g \} \}
\end{split}
\end{align}
for $H_2^{BM}(\mathcal{C}_2(\Sigma),\partial^-;\Z[\Heis])$ and for $H_2(\mathcal{C}_2(\Sigma),\partial^+;\Z[\Heis])$ respectively are dual with respect to this pairing. Choose a total ordering of $\mathcal{B}$ as follows:
\begin{itemize}
\item $w(\alpha_1)$, $w(\beta_1)$, $v(\alpha_1,\beta_1)$,
\item $v(\alpha_1,\epsilon)$ for $\epsilon = \alpha_2,\beta_2,\ldots,\alpha_g,\beta_g$,
\item $v(\beta_1,\epsilon)$ for $\epsilon = \alpha_2,\beta_2,\ldots,\alpha_g,\beta_g$,
\item followed by all other basis elements in any order,
\end{itemize}
and similarly for $\mathcal{B}'$. Denote by $\gamma$ the genus-$1$ separating curve in $\Sigma$ pictured in Figure \ref{fig:generators-higher-genus}.

\begin{theorem}
\label{thm_calculation}
With respect to the ordered bases \eqref{eq:ordered-bases}, the matrix for the automorphism $\mathcal{T}_\gamma = \mathcal{C}_2(T_\gamma)_*$ of $H_2^{BM}(\mathcal{C}_{2}(\Sigma),\partial^-;\Z[\Heis])$ is given in block form as
\begin{equation}
\label{eq:action-of-gamma}
M_\gamma = 
\left(\begin{matrix}
\Lambda & 0 & 0 & 0 \\
0 & p.I & r.I & 0 \\
0 & q.I & s.I & 0 \\
0 & 0 & 0 & I
\end{matrix}\right) ,
\end{equation}
where $\Lambda$ is the $3 \times 3$ matrix depicted in Figure \ref{fig:bigmatrix}, the middle two columns and rows each have width/height $2g-2$ and the Heisenberg polynomials $p,q,r,s \in \Z[\Heis]$ are:
\begin{itemize}
\item $p = -ab^{-1}+u^{-2}b^{-1}+u^{-2}a$,
\item $q = 1-a^{-1}+u^{-2}-u^{-2}a$,
\item $r = a(-b^{-1}+b^{-2}+u^{-2}-u^{-2}b^{-1})$,
\item $s = 1-b^{-1}+u^{-2}+u^{-2}ab^{-1}-u^{-2}a$,
\end{itemize}
where we are abbreviating the elements $a_1,b_1 \in \Heis$ as $a,b$ respectively.
\end{theorem}

\begin{proofnobox}
As in the proof of Theorem \ref{thm:MaMb}, this reduces to computing $\langle \mathcal{T}_\gamma(e_i) , e'_j \rangle$ as $e_i$ and $e'_j$ run through the ordered bases \eqref{eq:ordered-bases}.

First note that the basis elements come in three types: those entirely supported in the genus-$1$ subsurface containing $\gamma$ (the first three elements), those supported partially in this subsurface and partially in the complementary genus-$(g-1)$ subsurface (the next $4g-4$ elements) and those supported entirely outside of the genus-$1$ subsurface (the remaining elements). The Dehn twist $T_\gamma$ does not mix these two complementary subsurfaces, so $M_\gamma$ is a block matrix with respect to this partition.

The top-left $3 \times 3$ matrix involves only the basis elements $w(\alpha_1)$, $w(\beta_1)$, $v(\alpha_1,\beta_1)$ and their duals, and so the calculation of this submatrix is identical to the calculation in genus $1$, which is given by the matrix in Figure \ref{fig:bigmatrix}.

The bottom-right submatrix involves only basis elements supported outside of the genus-$1$ subsurface containing $\gamma$, so the effect of $\mathcal{T}_\gamma$ is the identity on these elements.

It remains to calculate the middle $(4g-4) \times (4g-4)$ submatrix, which records the effect of $\mathcal{T}_\gamma$ on $v(\alpha_1,\epsilon)$ and $v(\beta_1,\epsilon)$ for $\epsilon \in \{\alpha_2,\ldots,\beta_g\}$. Since $\epsilon \cap \gamma = \varnothing$, we must have
\begin{align*}
\mathcal{T}_\gamma(v(\alpha_1,\epsilon)) &= p_\epsilon . v(\alpha_1,\epsilon) + q_\epsilon . v(\beta_1,\epsilon) \\
\mathcal{T}_\gamma(v(\beta_1,\epsilon)) &= r_\epsilon . v(\alpha_1,\epsilon) + s_\epsilon . v(\beta_1,\epsilon)
\end{align*}
for some $p_\epsilon,q_\epsilon,r_\epsilon,s_\epsilon \in \Z[\Heis]$. Precisely, we have
\begin{align*}
p_\epsilon &= \big\langle v(T_\gamma(\alpha_1),\epsilon) , v(\alpha'_1,\epsilon') \big\rangle & q_\epsilon &= \big\langle v(T_\gamma(\alpha_1),\epsilon) , v(\beta'_1,\epsilon') \big\rangle \\
r_\epsilon &= \big\langle v(T_\gamma(\beta_1),\epsilon) , v(\alpha'_1,\epsilon') \big\rangle & s_\epsilon &= \big\langle v(T_\gamma(\beta_1),\epsilon) , v(\beta'_1,\epsilon') \big\rangle ,
\end{align*}
where $\epsilon'$ denotes the dual of $\epsilon$, and we have again used the fact that $\epsilon \cap \gamma = \varnothing$ to rewrite $\mathcal{T}_\gamma(v(\alpha_1,\epsilon)) = v(T_\gamma(\alpha_1) , T_\gamma(\epsilon)) = v(T_\gamma(\alpha_1) , \epsilon)$ and similarly for $\mathcal{T}_\gamma(v(\alpha_1,\epsilon))$. From these formulas and \eqref{eq:pairing-formula} it is clear that $p_\epsilon,q_\epsilon,r_\epsilon,s_\epsilon$ do not in fact depend on $\epsilon$. Indeed, when computing these values of the non-degenerate pairing, we may ignore one of the two configuration points (the one that starts on the left in the base configuration and which travels via the arcs $\epsilon$ and $\epsilon'$), since it contributes neither to the signs nor to the loops $\ell_p$ in the formula \eqref{eq:pairing-formula}. We will compute $s_\epsilon = s$, leaving the computation of the other three polynomials as exercises for the reader. In the following computations, as mentioned above, we ignore one of the two configuration points, since it does not contribute anything non-trivial to the formula \eqref{eq:pairing-formula}.
\begin{align*}
s &= \langle v(T_\gamma(\beta_1),\epsilon) , v(\beta'_1,\epsilon') \rangle \\
&= \phantom{+} \raisebox{-0.5\height}{\includegraphics[scale=0.4]{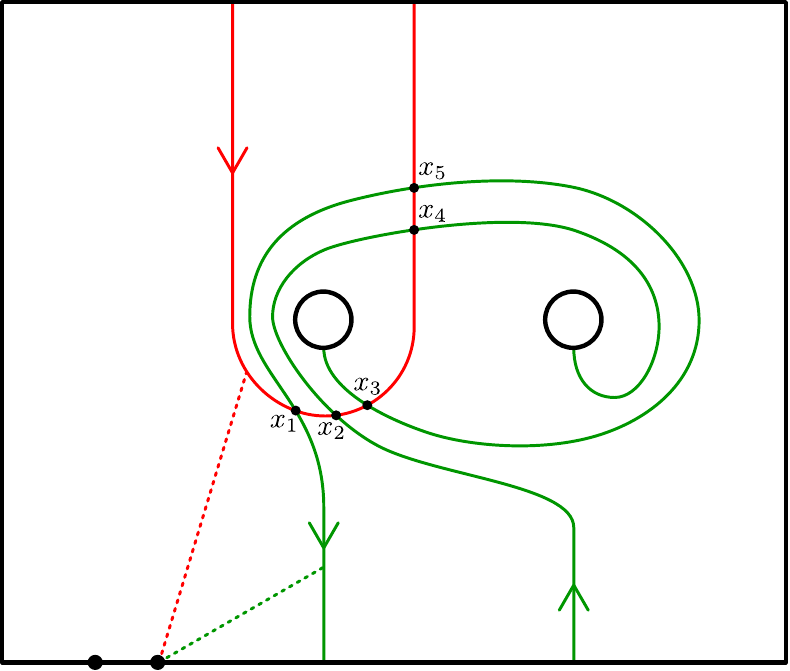}} \;\; (5 \text{ intersection points: } x_1,\ldots,x_5) \\
&= \phantom{+} \phi \left(\; \raisebox{-0.48\height}{\includegraphics[scale=0.3]{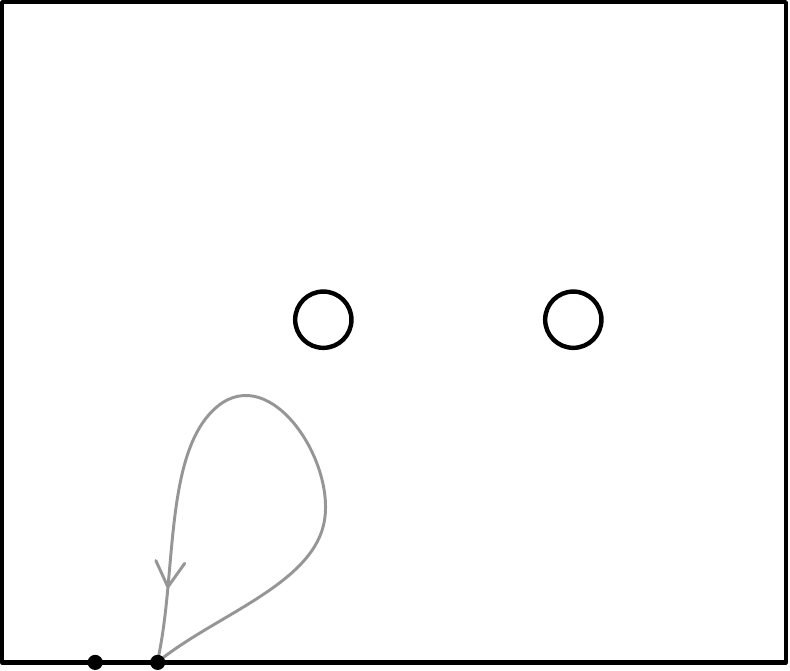}} \;\right) - \phi \left(\; \raisebox{-0.48\height}{\includegraphics[scale=0.3]{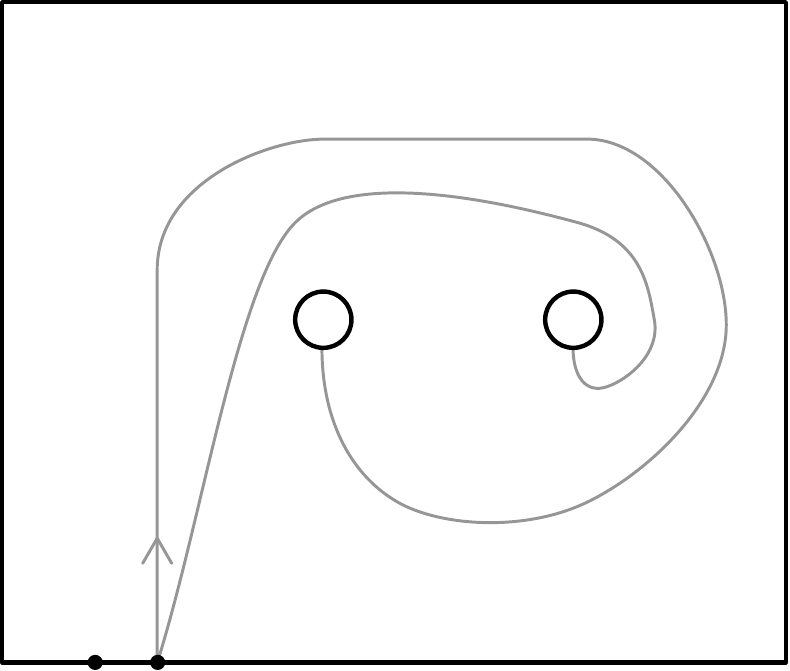}} \;\right) \\
&\phantom{=} + \phi \left(\; \raisebox{-0.48\height}{\includegraphics[scale=0.3]{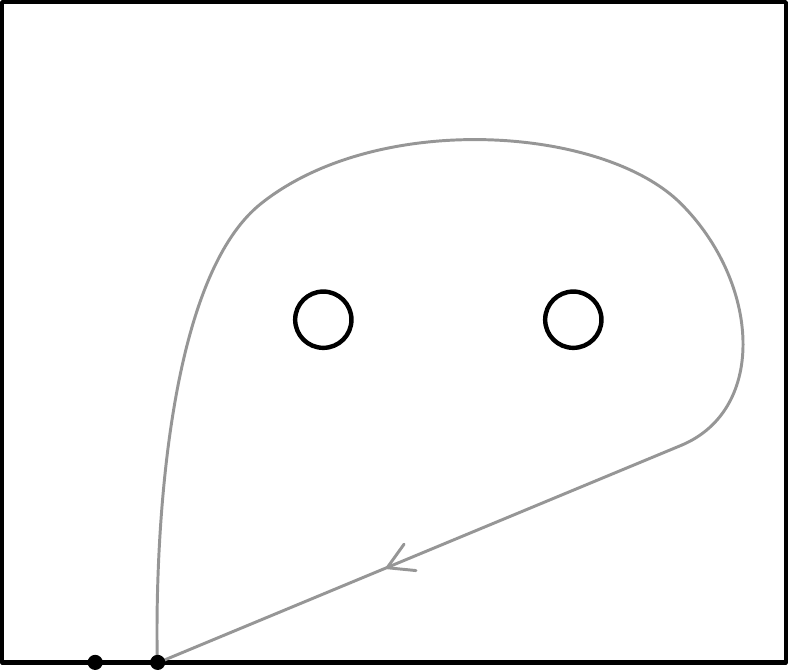}} \;\right) + \phi \left(\; \raisebox{-0.48\height}{\includegraphics[scale=0.3]{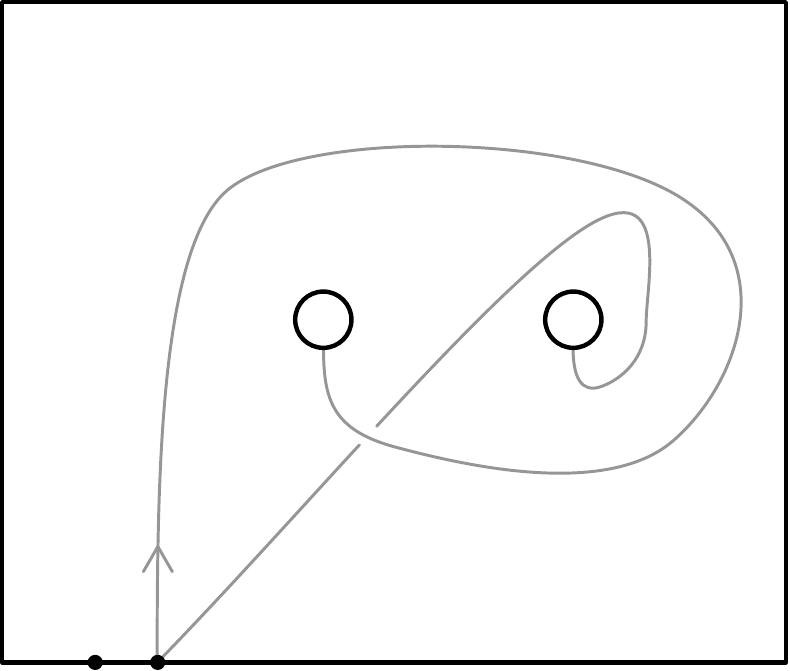}} \;\right) \\
&\phantom{=} - \phi \left(\; \raisebox{-0.48\height}{\includegraphics[scale=0.3]{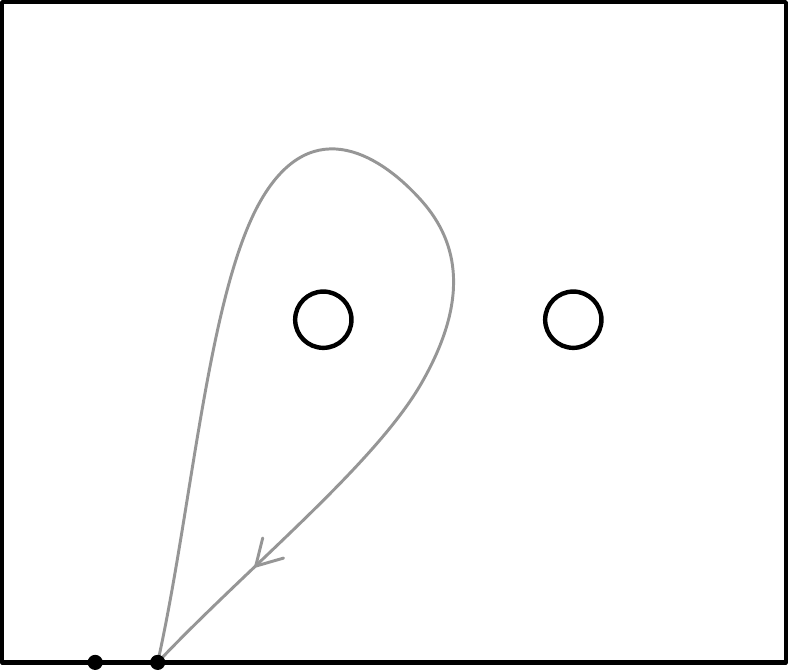}} \;\right) \\
&= \phi(\phantom{-}) - \phi(\sigma^{-1} \beta \alpha^{-1} \beta^{-1} \alpha \beta^{-1} \alpha^{-1} \beta \alpha \beta^{-1} \sigma) + \phi(\sigma^{-1} \alpha^{-1} \beta \alpha \beta^{-1} \sigma) \\
&\qquad + \phi(\sigma^{-1} \alpha \beta^{-1} \alpha^{-1} \beta \alpha \beta^{-1} \sigma) - \phi(\sigma^{-1} \beta \alpha \beta^{-1} \sigma) \\
&= 1 - b^{-1} + u^{-2} + u^{-2}ab^{-1} - u^{-2}a. \hspace{5.5cm} \rule{0.5em}{0.5em}
\end{align*}
\end{proofnobox}

\section*{Appendix A: signs in the intersection pairing formula}
\phantomsection
\addcontentsline{toc}{section}{Appendix A: signs in the intersection pairing formula}
\label{appendixA}

\begin{figure}[t]
\centering
\includegraphics[scale=0.5]{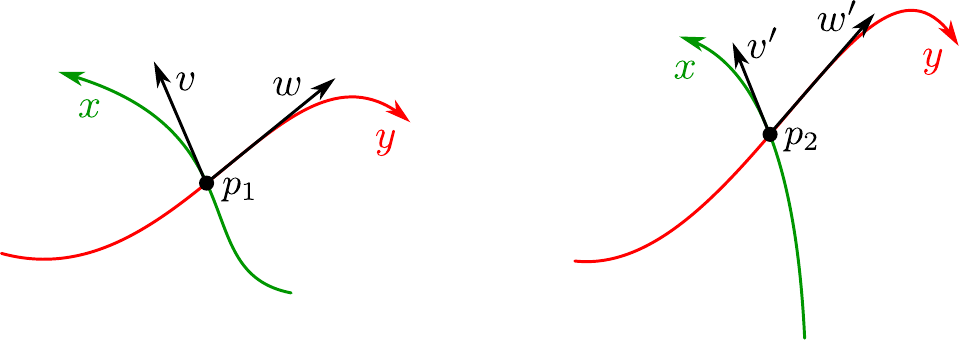}
\caption{Choices of tangent vectors from the computation of the sign of the intersection of $x$ and $y$ at $p = \{p_1,p_2\} \in \mathcal{C}_2(\Sigma)$.}
\label{fig:signs-pairing}
\end{figure}

Here we explain the signs appearing in the formula \eqref{eq:pairing-formula} for the intersection pairing on the homology of $2$-point configuration spaces, including the extra global $-1$ sign that was suppressed in \eqref{eq:pairing-formula} (see the comment in the paragraph below the formula).

We take the viewpoint that an orientation $o$ of a $d$-dimensional smooth manifold $M$ is given by a consistent choice of vector $o(p) \in \Lambda^d T_p M$ for all $p \in M$. We either choose a metric on the bundle $\Lambda^d TM$ and require $o(p)$ to be a unit vector with respect to this metric, or we consider $o(p)$ up to rescaling by positive real numbers.

Let us fix an orientation $o_\Sigma$ for the surface $\Sigma$. This determines an orientation $o_{\mathcal{C}_2(\Sigma)}$ of the configuration space $\mathcal{C}_2(\Sigma)$ by setting
\[
o_{\mathcal{C}_2(\Sigma)}(\{p_1,p_2\}) = o_\Sigma(p_1) \wedge o_\Sigma(p_2).
\]
Recall that we have $2$-dimensional submanifolds $x$ and $y$ of $\mathcal{C}_2(\Sigma)$ that intersect transversely, and let $p = \{p_1,p_2\}$ be a point of $x \cap y$. Let $v,w$ be the tangent vectors at $p_1$ and let $v',w'$ be the tangent vectors at $p_2$ illustrated in Figure \ref{fig:signs-pairing}. We have
\begin{align*}
v \wedge w &= \mathrm{sgn}(p_1).o_\Sigma(p_1) \\
v' \wedge w' &= \mathrm{sgn}(p_2).o_\Sigma(p_2),
\end{align*}
where $\mathrm{sgn}(p_i)$ is the sign of the intersection of the arcs in $\Sigma$ underlying $x$ and $y$ at $p_i$. Similarly, we have
\[
o_x(p) \wedge o_y(p) = \mathrm{sgn}(p).o_{\mathcal{C}_2(\Sigma)}(p),
\]
where $\mathrm{sgn}(p)$ is the sign that we are trying to compute: the sign of the intersection of $x$ and $y$ in the configuration space. The orientations of $x$ and $y$ depend on the tethers $t_x$, $t_y$ that have been chosen. Precisely, we have
\begin{align*}
o_x(p) &= \left\lbrace \begin{matrix} v \wedge v' & (*) \\ v' \wedge v & (\dagger) \end{matrix} \right\rbrace & o_y(p) &= \left\lbrace \begin{matrix} w \wedge w' & (*) \\ w' \wedge w & (\dagger) \end{matrix} \right\rbrace ,
\end{align*}
where the possibilities $((*),(*))$ or $((\dagger),(\dagger))$ occur if $\mathrm{sgn}(\ell_p) = +1$ and the possibilities $((*),(\dagger))$ or $((\dagger),(*))$ occur if $\mathrm{sgn}(\ell_p) = -1$. We therefore have
\[
o_x(p) \wedge o_y(p) = \mathrm{sgn}(\ell_p).(v \wedge v') \wedge (w \wedge w').
\]
Putting this together with the formulas above, we obtain
\begin{align*}
(v \wedge w) \wedge (v' \wedge w') &= \mathrm{sgn}(p_1).\mathrm{sgn}(p_2).o_\Sigma(p_1) \wedge o_\Sigma(p_2) \\
&= \mathrm{sgn}(p_1).\mathrm{sgn}(p_2).o_{\mathcal{C}_2(\Sigma)}(p) \\
&= \mathrm{sgn}(p_1).\mathrm{sgn}(p_2).\mathrm{sgn}(p).o_x(p) \wedge o_y(p) \\
&= \mathrm{sgn}(p_1).\mathrm{sgn}(p_2).\mathrm{sgn}(p).\mathrm{sgn}(\ell_p).(v \wedge v') \wedge (w \wedge w') \\
&= -\mathrm{sgn}(p_1).\mathrm{sgn}(p_2).\mathrm{sgn}(p).\mathrm{sgn}(\ell_p).(v \wedge w) \wedge (v' \wedge w'),
\end{align*}
and hence we have
\[
\mathrm{sgn}(p) = -\mathrm{sgn}(p_1).\mathrm{sgn}(p_2).\mathrm{sgn}(\ell_p).
\]

\section*{Appendix B: Universal coefficient spectral sequence arguments}
\phantomsection
\addcontentsline{toc}{section}{Appendix B: Universal coefficient spectral sequence arguments}
\label{appendixB}

As noted in Remark \ref{rmk:alternative-argument}, the general case of Theorem \ref{basis} (both for Borel--Moore homology and for compactly-supported cohomology) may be deduced from the result for Borel--Moore homology for a specific choice of $V$. The purpose of this appendix is to explain precisely how this may be done, using some slightly delicate universal coefficient spectral sequence arguments that are abstracted in Lemma \ref{lem:universal-coefficients} below, and which are similar to the argument of \cite[Appendix~A]{BigelowMartel}.

In this lemma, we interpret local coefficient systems on a space $X$ as actions of $\pi_1(X)$ on (bi)modules, or equivalently as (bi)modules over the group ring of $\pi_1(X)$.
The $\pi_1(X)$-action should be on the left for homology local coefficients and on the right for cohomology local coefficients. Thus each $(R[\pi_1(X)],S)$-module $V$ determines $(R,S)$-bimodules $H_k^{BM}(X;V)$ and each $(R,S[\pi_1(X)])$-module $W$ determines $(R,S)$-bimodules $H^k_{c}(X;W)$. In particular, we note that $H_k^{BM}(X;R[\pi_1(X)])$ has the structure of an $(R,R[\pi_1(X)])$-bimodule, since $V = R[\pi_1(X)]$ is both a left and right module over itself.

\begin{lemma}
\label{lem:universal-coefficients}
Let $X$ be a based, path-connected space admitting a universal cover. Suppose that there is a sequence of natural numbers $d_k$ such that, for any unital ring $R$ and each $k\geq 0$, there are isomorphisms
\begin{equation}
\label{eq:universal-coefficients-condition}
H_k^{BM}(X;R[\pi_1(X)]) \cong R[\pi_1(X)]^{\oplus d_k}
\end{equation}
of $(R,R[\pi_1(X)])$-bimodules. Then we have isomorphisms
\begin{equation}
\label{eq:universal-coefficients-conclusion}
H_k^{BM}(X;V) \cong V^{\oplus d_k} \qquad\text{and}\qquad H^k_c(X;W) \cong W^{\oplus d_k}
\end{equation}
of $(R,S)$-bimodules for any unital rings $R,S$, any $(R[\pi_1(X)],S)$-bimodule $V$ and $(R,S[\pi_1(X)])$-bimodule $W$. The same is true for Borel--Moore homology and compactly-supported cohomology relative to a subspace $A \subseteq X$.
\end{lemma}

\begin{proof}
Twisted Borel--Moore homology and twisted compactly-supported cohomology may both be computed using the complex of \emph{horizontally locally finite chains} $\mathcal{S}_*^{\hlf}(\widetilde{X};R)$, where $\pi \colon \widetilde{X} \to X$ is the universal covering of $X$. A $k$-dimensional horizontally locally finite chain is a formal sum $\sum_i \lambda_i s_i$ with $\lambda_i \in R$ and $s_i \colon \Delta^k \to \widetilde{X}$ such that each point $x \in X$ has an open neighbourhood $U$ such that $\pi^{-1}(U)$ intersects only finitely many of the $s_i(\Delta^k)$. Note that if $\pi_1(X)$ is finite (in other words $\pi \colon \widetilde{X} \to X$ is a finite-sheeted covering), this is the same as the complex of locally finite chains on $\widetilde{X}$, but when $\pi_1(X)$ is infinite it is a proper subcomplex. For any $(R[\pi_1(X)],S)$-bimodule $V$ and $(R,S[\pi_1(X)])$-bimodule $W$, we have $(R,S)$-bimodule isomorphisms:
\begin{align*}
H_*^{BM}(X;V) &\cong H_* \left( \mathcal{S}_*^{\hlf}(\widetilde{X};R) \otimes_{R[\pi_1(X)]} V \right) \\
H^*_c(X;W) &\cong H_* \left( \Hom_{S[\pi_1(X)]} \bigl( \mathcal{S}_*^{\hlf}(\widetilde{X};S) , W \bigr) \right) .
\end{align*}
Since $\pi_1(X)$ acts freely on $\widetilde{X}$, the chain complex $\mathcal{S}_*^{\hlf}(\widetilde{X};R)$ is free over $R[\pi_1(X)]$ in each degree. Applying the algebraic universal coefficient theorem (see \cite[Ch.~XVII]{CartanEilenberg1956} or \cite[Thm.~2.3]{Levine1977}), we get spectral sequences of $(R,S)$-bimodules:
\begin{align*}
E_{p,q}^2 = \mathrm{Tor}_q^{R[\pi_1(X)]} \left( H_p^{BM}(X;R[\pi_1(X)]) , V \right) \;&\Longrightarrow\; H_*^{BM}(X;V) \\
E^{p,q}_2 = \mathrm{Ext}^q_{S[\pi_1(X)]} \left( H_p^{BM}(X;S[\pi_1(X)]) , W \right) \;&\Longrightarrow\; H^*_c(X;W)
\end{align*}
Assumption \eqref{eq:universal-coefficients-condition} implies that $E^2_{p,q} = 0 = E_2^{p,q}$ for $q>0$, so these spectral sequences degenerate to isomorphisms of $(R,S)$-bimodules:
\begin{align*}
H_k^{BM}(X;V) &\cong R[\pi_1(X)]^{\oplus d_k} \otimes_{R[\pi_1(X)]} V \cong V^{\oplus d_k} \\
H^k_c(X;W) &\cong \Hom_{S[\pi_1(X)]} \bigl( S[\pi_1(X)]^{\oplus d_k} , W \bigr) \cong W^{\oplus d_k}.
\end{align*}
The analogous result for Borel--Moore homology and compactly-supported cohomology relative to a subspace $A \subseteq X$ follows by exactly the same argument if we first replace the chain complex $\mathcal{S}_*^{\hlf}(\widetilde{X};R)$ with its quotient by the subcomplex consisting of those horizontally locally finite chains $\sum_i \lambda_i s_i$ with $s_i(\Delta^k) \subseteq \pi^{-1}(A)$.
\end{proof}

\begin{proof}[Proof of Theorem \ref{basis} assuming that it holds in a special case]
Let us assume that the Borel--Moore homology version of Theorem \ref{basis} holds when $V = S = R[\pi_1(\mathcal{C}_n(\Sigma))]$.\footnote{Recall from Remark \ref{rmk:more-general-coefficients} that Theorem \ref{basis} for Borel--Moore homology is true (with the same proof) when $V$ is \emph{any} left representation of $\mathbb{B}_n(\Sigma) = \pi_1(\mathcal{C}_n(\Sigma))$, not necessarily factoring through the quotient onto the Heisenberg group $\Heis$.} Lemma \ref{lem:universal-coefficients} then immediately implies the general result for Borel--Moore homology (resp.\ for compactly-supported cohomology) with coefficients in any left (resp.\ right) representation $V$ of $\pi_1(\mathcal{C}_n(\Sigma))$ (in particular, for those that factor through the quotient onto the Heisenberg group $\Heis$).
\end{proof}

\section*{Appendix C: Sage computations}
\phantomsection
\addcontentsline{toc}{section}{Appendix C: Sage computations}
\label{appendixC}

Here we give the worksheet of the Sage computations used in the calculation of the matrix $M_\partial$ displayed in Figure \ref{fig:bigmatrix} (cf.~Remark \ref{rmk:Dehn-twist-boundary} on page \pageref{rmk:Dehn-twist-boundary}).

\vspace{1em}

\input{sage-output.tex}

\clearpage
\bibliographystyle{plainitalic}
\bibliography{biblio}

\end{document}

%% file: sage-preamble.tex
\usepackage[breakable]{tcolorbox}

\usepackage{xcolor} 
\usepackage{enumerate} 
\usepackage{textcomp} 
\AtBeginDocument{%
}
\usepackage{upquote} 

\usepackage{iftex}
\ifPDFTeX
    \usepackage[T1]{fontenc}
    \IfFileExists{alphabeta.sty}{
          \usepackage{alphabeta}
      }{
          \usepackage[mathletters]{ucs}
          \usepackage[utf8x]{inputenc}
      }
\else
    \usepackage{fontspec}
    \usepackage{unicode-math}
\fi

\usepackage{fancyvrb} 

\DefineVerbatimEnvironment{Highlighting}{Verbatim}{commandchars=\\\{\}}

\let\Oldtex\TeX
\let\Oldlatex\LaTeX
\renewcommand{\TeX}{\textrm{\Oldtex}}
\renewcommand{\LaTeX}{\textrm{\Oldlatex}}

\makeatletter
\def\PY@reset{\let\PY@it=\relax \let\PY@bf=\relax%
    \let\PY@ul=\relax \let\PY@tc=\relax%
    \let\PY@bc=\relax \let\PY@ff=\relax}
\def\PY@tok#1{\csname PY@tok@#1\endcsname}
\def\PY@toks#1+{\ifx\relax#1\empty\else%
    \PY@tok{#1}\expandafter\PY@toks\fi}
\def\PY@do#1{\PY@bc{\PY@tc{\PY@ul{%
    \PY@it{\PY@bf{\PY@ff{#1}}}}}}}
\def\PY#1#2{\PY@reset\PY@toks#1+\relax+\PY@do{#2}}

\@namedef{PY@tok@w}{\def\PY@tc##1{\textcolor[rgb]{0.73,0.73,0.73}{##1}}}
\@namedef{PY@tok@c}{\let\PY@it=\textit\def\PY@tc##1{\textcolor[rgb]{0.24,0.48,0.48}{##1}}}
\@namedef{PY@tok@cp}{\def\PY@tc##1{\textcolor[rgb]{0.61,0.40,0.00}{##1}}}
\@namedef{PY@tok@k}{\let\PY@bf=\textbf\def\PY@tc##1{\textcolor[rgb]{0.00,0.50,0.00}{##1}}}
\@namedef{PY@tok@kp}{\def\PY@tc##1{\textcolor[rgb]{0.00,0.50,0.00}{##1}}}
\@namedef{PY@tok@kt}{\def\PY@tc##1{\textcolor[rgb]{0.69,0.00,0.25}{##1}}}
\@namedef{PY@tok@o}{\def\PY@tc##1{\textcolor[rgb]{0.40,0.40,0.40}{##1}}}
\@namedef{PY@tok@ow}{\let\PY@bf=\textbf\def\PY@tc##1{\textcolor[rgb]{0.67,0.13,1.00}{##1}}}
\@namedef{PY@tok@nb}{\def\PY@tc##1{\textcolor[rgb]{0.00,0.50,0.00}{##1}}}
\@namedef{PY@tok@nf}{\def\PY@tc##1{\textcolor[rgb]{0.00,0.00,1.00}{##1}}}
\@namedef{PY@tok@nc}{\let\PY@bf=\textbf\def\PY@tc##1{\textcolor[rgb]{0.00,0.00,1.00}{##1}}}
\@namedef{PY@tok@nn}{\let\PY@bf=\textbf\def\PY@tc##1{\textcolor[rgb]{0.00,0.00,1.00}{##1}}}
\@namedef{PY@tok@ne}{\let\PY@bf=\textbf\def\PY@tc##1{\textcolor[rgb]{0.80,0.25,0.22}{##1}}}
\@namedef{PY@tok@nv}{\def\PY@tc##1{\textcolor[rgb]{0.10,0.09,0.49}{##1}}}
\@namedef{PY@tok@no}{\def\PY@tc##1{\textcolor[rgb]{0.53,0.00,0.00}{##1}}}
\@namedef{PY@tok@nl}{\def\PY@tc##1{\textcolor[rgb]{0.46,0.46,0.00}{##1}}}
\@namedef{PY@tok@ni}{\let\PY@bf=\textbf\def\PY@tc##1{\textcolor[rgb]{0.44,0.44,0.44}{##1}}}
\@namedef{PY@tok@na}{\def\PY@tc##1{\textcolor[rgb]{0.41,0.47,0.13}{##1}}}
\@namedef{PY@tok@nt}{\let\PY@bf=\textbf\def\PY@tc##1{\textcolor[rgb]{0.00,0.50,0.00}{##1}}}
\@namedef{PY@tok@nd}{\def\PY@tc##1{\textcolor[rgb]{0.67,0.13,1.00}{##1}}}
\@namedef{PY@tok@s}{\def\PY@tc##1{\textcolor[rgb]{0.73,0.13,0.13}{##1}}}
\@namedef{PY@tok@sd}{\let\PY@it=\textit\def\PY@tc##1{\textcolor[rgb]{0.73,0.13,0.13}{##1}}}
\@namedef{PY@tok@si}{\let\PY@bf=\textbf\def\PY@tc##1{\textcolor[rgb]{0.64,0.35,0.47}{##1}}}
\@namedef{PY@tok@se}{\let\PY@bf=\textbf\def\PY@tc##1{\textcolor[rgb]{0.67,0.36,0.12}{##1}}}
\@namedef{PY@tok@sr}{\def\PY@tc##1{\textcolor[rgb]{0.64,0.35,0.47}{##1}}}
\@namedef{PY@tok@ss}{\def\PY@tc##1{\textcolor[rgb]{0.10,0.09,0.49}{##1}}}
\@namedef{PY@tok@sx}{\def\PY@tc##1{\textcolor[rgb]{0.00,0.50,0.00}{##1}}}
\@namedef{PY@tok@m}{\def\PY@tc##1{\textcolor[rgb]{0.40,0.40,0.40}{##1}}}
\@namedef{PY@tok@gh}{\let\PY@bf=\textbf\def\PY@tc##1{\textcolor[rgb]{0.00,0.00,0.50}{##1}}}
\@namedef{PY@tok@gu}{\let\PY@bf=\textbf\def\PY@tc##1{\textcolor[rgb]{0.50,0.00,0.50}{##1}}}
\@namedef{PY@tok@gd}{\def\PY@tc##1{\textcolor[rgb]{0.63,0.00,0.00}{##1}}}
\@namedef{PY@tok@gi}{\def\PY@tc##1{\textcolor[rgb]{0.00,0.52,0.00}{##1}}}
\@namedef{PY@tok@gr}{\def\PY@tc##1{\textcolor[rgb]{0.89,0.00,0.00}{##1}}}
\@namedef{PY@tok@ge}{\let\PY@it=\textit}
\@namedef{PY@tok@gs}{\let\PY@bf=\textbf}
\@namedef{PY@tok@gp}{\let\PY@bf=\textbf\def\PY@tc##1{\textcolor[rgb]{0.00,0.00,0.50}{##1}}}
\@namedef{PY@tok@go}{\def\PY@tc##1{\textcolor[rgb]{0.44,0.44,0.44}{##1}}}
\@namedef{PY@tok@gt}{\def\PY@tc##1{\textcolor[rgb]{0.00,0.27,0.87}{##1}}}
\@namedef{PY@tok@err}{\def\PY@bc##1{{\setlength{\fboxsep}{\string -\fboxrule}\fcolorbox[rgb]{1.00,0.00,0.00}{1,1,1}{\strut ##1}}}}
\@namedef{PY@tok@kc}{\let\PY@bf=\textbf\def\PY@tc##1{\textcolor[rgb]{0.00,0.50,0.00}{##1}}}
\@namedef{PY@tok@kd}{\let\PY@bf=\textbf\def\PY@tc##1{\textcolor[rgb]{0.00,0.50,0.00}{##1}}}
\@namedef{PY@tok@kn}{\let\PY@bf=\textbf\def\PY@tc##1{\textcolor[rgb]{0.00,0.50,0.00}{##1}}}
\@namedef{PY@tok@kr}{\let\PY@bf=\textbf\def\PY@tc##1{\textcolor[rgb]{0.00,0.50,0.00}{##1}}}
\@namedef{PY@tok@bp}{\def\PY@tc##1{\textcolor[rgb]{0.00,0.50,0.00}{##1}}}
\@namedef{PY@tok@fm}{\def\PY@tc##1{\textcolor[rgb]{0.00,0.00,1.00}{##1}}}
\@namedef{PY@tok@vc}{\def\PY@tc##1{\textcolor[rgb]{0.10,0.09,0.49}{##1}}}
\@namedef{PY@tok@vg}{\def\PY@tc##1{\textcolor[rgb]{0.10,0.09,0.49}{##1}}}
\@namedef{PY@tok@vi}{\def\PY@tc##1{\textcolor[rgb]{0.10,0.09,0.49}{##1}}}
\@namedef{PY@tok@vm}{\def\PY@tc##1{\textcolor[rgb]{0.10,0.09,0.49}{##1}}}
\@namedef{PY@tok@sa}{\def\PY@tc##1{\textcolor[rgb]{0.73,0.13,0.13}{##1}}}
\@namedef{PY@tok@sb}{\def\PY@tc##1{\textcolor[rgb]{0.73,0.13,0.13}{##1}}}
\@namedef{PY@tok@sc}{\def\PY@tc##1{\textcolor[rgb]{0.73,0.13,0.13}{##1}}}
\@namedef{PY@tok@dl}{\def\PY@tc##1{\textcolor[rgb]{0.73,0.13,0.13}{##1}}}
\@namedef{PY@tok@s2}{\def\PY@tc##1{\textcolor[rgb]{0.73,0.13,0.13}{##1}}}
\@namedef{PY@tok@sh}{\def\PY@tc##1{\textcolor[rgb]{0.73,0.13,0.13}{##1}}}
\@namedef{PY@tok@s1}{\def\PY@tc##1{\textcolor[rgb]{0.73,0.13,0.13}{##1}}}
\@namedef{PY@tok@mb}{\def\PY@tc##1{\textcolor[rgb]{0.40,0.40,0.40}{##1}}}
\@namedef{PY@tok@mf}{\def\PY@tc##1{\textcolor[rgb]{0.40,0.40,0.40}{##1}}}
\@namedef{PY@tok@mh}{\def\PY@tc##1{\textcolor[rgb]{0.40,0.40,0.40}{##1}}}
\@namedef{PY@tok@mi}{\def\PY@tc##1{\textcolor[rgb]{0.40,0.40,0.40}{##1}}}
\@namedef{PY@tok@il}{\def\PY@tc##1{\textcolor[rgb]{0.40,0.40,0.40}{##1}}}
\@namedef{PY@tok@mo}{\def\PY@tc##1{\textcolor[rgb]{0.40,0.40,0.40}{##1}}}
\@namedef{PY@tok@ch}{\let\PY@it=\textit\def\PY@tc##1{\textcolor[rgb]{0.24,0.48,0.48}{##1}}}
\@namedef{PY@tok@cm}{\let\PY@it=\textit\def\PY@tc##1{\textcolor[rgb]{0.24,0.48,0.48}{##1}}}
\@namedef{PY@tok@cpf}{\let\PY@it=\textit\def\PY@tc##1{\textcolor[rgb]{0.24,0.48,0.48}{##1}}}
\@namedef{PY@tok@c1}{\let\PY@it=\textit\def\PY@tc##1{\textcolor[rgb]{0.24,0.48,0.48}{##1}}}
\@namedef{PY@tok@cs}{\let\PY@it=\textit\def\PY@tc##1{\textcolor[rgb]{0.24,0.48,0.48}{##1}}}

\def\PYZus{\char`\_}
\def\PYZob{\char`\{}
\def\PYZcb{\char`\}}
\def\PYZca{\char`\^}

\def\PYZlt{\char`\<}
\def\PYZgt{\char`\>}
\def\PYZsh{\char`\#}
\def\PYZpc{\char`\%}

\def\PYZhy{\char`\-}

\def\PYZdq{\char`\"}

\makeatother

\makeatletter
    \newbox\Wrappedcontinuationbox 
    \newbox\Wrappedvisiblespacebox 
    \newcommand*\Wrappedvisiblespace {\textcolor{red}{\textvisiblespace}} 
    \newcommand*\Wrappedcontinuationsymbol {\textcolor{red}{\llap{\tiny$\m@th\hookrightarrow$}}} 
    \newcommand*\Wrappedcontinuationindent {3ex } 
    \newcommand*\Wrappedafterbreak {\kern\Wrappedcontinuationindent\copy\Wrappedcontinuationbox} 
    \newcommand*\Wrappedbreaksatspecials {%
        \def\PYGZus{\discretionary{\char`\_}{\Wrappedafterbreak}{\char`\_}}%
        \def\PYGZob{\discretionary{}{\Wrappedafterbreak\char`\{}{\char`\{}}%
        \def\PYGZcb{\discretionary{\char`\}}{\Wrappedafterbreak}{\char`\}}}%
        \def\PYGZca{\discretionary{\char`\^}{\Wrappedafterbreak}{\char`\^}}%
        \def\PYGZam{\discretionary{\char`\&}{\Wrappedafterbreak}{\char`\&}}%
        \def\PYGZlt{\discretionary{}{\Wrappedafterbreak\char`\<}{\char`\<}}%
        \def\PYGZgt{\discretionary{\char`\>}{\Wrappedafterbreak}{\char`\>}}%
        \def\PYGZsh{\discretionary{}{\Wrappedafterbreak\char`\#}{\char`\#}}%
        \def\PYGZpc{\discretionary{}{\Wrappedafterbreak\char`\%}{\char`\%}}%
        \def\PYGZdl{\discretionary{}{\Wrappedafterbreak\char`\$}{\char`\$}}%
        \def\PYGZhy{\discretionary{\char`\-}{\Wrappedafterbreak}{\char`\-}}%
        \def\PYGZsq{\discretionary{}{\Wrappedafterbreak\textquotesingle}{\textquotesingle}}%
        \def\PYGZdq{\discretionary{}{\Wrappedafterbreak\char`\"}{\char`\"}}%
        \def\PYGZti{\discretionary{\char`\~}{\Wrappedafterbreak}{\char`\~}}%
    } 
    \newcommand*\Wrappedbreaksatpunct {%
        \lccode`\~`\.\lowercase{\def~}{\discretionary{\hbox{\char`\.}}{\Wrappedafterbreak}{\hbox{\char`\.}}}%
        \lccode`\~`\,\lowercase{\def~}{\discretionary{\hbox{\char`\,}}{\Wrappedafterbreak}{\hbox{\char`\,}}}%
        \lccode`\~`\;\lowercase{\def~}{\discretionary{\hbox{\char`\;}}{\Wrappedafterbreak}{\hbox{\char`\;}}}%
        \lccode`\~`\:\lowercase{\def~}{\discretionary{\hbox{\char`\:}}{\Wrappedafterbreak}{\hbox{\char`\:}}}%
        \lccode`\~`\?\lowercase{\def~}{\discretionary{\hbox{\char`\?}}{\Wrappedafterbreak}{\hbox{\char`\?}}}%
        \lccode`\~`\!\lowercase{\def~}{\discretionary{\hbox{\char`\!}}{\Wrappedafterbreak}{\hbox{\char`\!}}}%
        \lccode`\~`\/\lowercase{\def~}{\discretionary{\hbox{\char`\/}}{\Wrappedafterbreak}{\hbox{\char`\/}}}%
        \catcode`\.\active
        \catcode`\,\active 
        \catcode`\;\active
        \catcode`\:\active
        \catcode`\?\active
        \catcode`\!\active
        \catcode`\/\active 
        \lccode`\~`\~ 	
    }
\makeatother

\let\OriginalVerbatim=\Verbatim
\makeatletter
\renewcommand{\Verbatim}[1][1]{%
    \sbox\Wrappedcontinuationbox {\Wrappedcontinuationsymbol}%
    \sbox\Wrappedvisiblespacebox {\FV@SetupFont\Wrappedvisiblespace}%
    \def\FancyVerbFormatLine ##1{\hsize\linewidth
        \vtop{\raggedright\hyphenpenalty\z@\exhyphenpenalty\z@
            \doublehyphendemerits\z@\finalhyphendemerits\z@
            \strut ##1\strut}%
    }%
    \def\FV@Space {%
        \nobreak\hskip\z@ plus\fontdimen3\font minus\fontdimen4\font
        \discretionary{\copy\Wrappedvisiblespacebox}{\Wrappedafterbreak}
        {\kern\fontdimen2\font}%
    }%
    
    \Wrappedbreaksatspecials
    \OriginalVerbatim[#1,codes*=\Wrappedbreaksatpunct]%
}
\makeatother

\definecolor{incolor}{HTML}{303F9F}
\definecolor{outcolor}{HTML}{D84315}
\definecolor{cellborder}{HTML}{CFCFCF}
\definecolor{cellbackground}{HTML}{F7F7F7}

\makeatletter
\newcommand{\boxspacing}{\kern\kvtcb@left@rule\kern\kvtcb@boxsep}
\makeatother
\newcommand{\prompt}[4]{
    {\ttfamily\llap{{\color{#2}[#3]:\hspace{3pt}#4}}\vspace{-\baselineskip}}
}

%% file: sage-output.tex
    \begin{tcolorbox}[breakable, size=fbox, boxrule=1pt, pad at break*=1mm,colback=cellbackground, colframe=cellborder]
\prompt{In}{incolor}{1}{\boxspacing}
\begin{Verbatim}[commandchars=\\\{\}]
\PY{n}{load}\PY{p}{(}\PY{l+s+s2}{\PYZdq{}}\PY{l+s+s2}{HeisLatex\PYZus{}.sage}\PY{l+s+s2}{\PYZdq{}}\PY{p}{)} \PY{c+c1}{\PYZsh{}available on demand }
\end{Verbatim}
\end{tcolorbox}

    \begin{tcolorbox}[breakable, size=fbox, boxrule=1pt, pad at break*=1mm,colback=cellbackground, colframe=cellborder]
\prompt{In}{incolor}{2}{\boxspacing}
\begin{Verbatim}[commandchars=\\\{\}]
\PY{c+c1}{\PYZsh{} R is the centre of Heisenberg group ring}
\PY{n}{R}\PY{o}{.}\PY{o}{\PYZlt{}}\PY{n}{u}\PY{o}{\PYZgt{}}\PY{o}{=} \PY{n}{LaurentPolynomialRing}\PY{p}{(}\PY{n}{ZZ}\PY{p}{,}\PY{l+m+mi}{1}\PY{p}{)}
\end{Verbatim}
\end{tcolorbox}

    \begin{tcolorbox}[breakable, size=fbox, boxrule=1pt, pad at break*=1mm,colback=cellbackground, colframe=cellborder]
\prompt{In}{incolor}{3}{\boxspacing}
\begin{Verbatim}[commandchars=\\\{\}]
\PY{c+c1}{\PYZsh{} H is Heisenberg group ring}
\PY{n}{H} \PY{o}{=} \PY{n}{Heis}\PY{p}{(}\PY{n}{base}\PY{o}{=}\PY{n}{R}\PY{p}{,} \PY{n}{category}\PY{o}{=}\PY{n}{Rings}\PY{p}{(}\PY{p}{)}\PY{p}{)}
\end{Verbatim}
\end{tcolorbox}

    \begin{tcolorbox}[breakable, size=fbox, boxrule=1pt, pad at break*=1mm,colback=cellbackground, colframe=cellborder]
\prompt{In}{incolor}{4}{\boxspacing}
\begin{Verbatim}[commandchars=\\\{\}]
\PY{n}{a}\PY{o}{=}\PY{n}{H}\PY{p}{(}\PY{n+nb}{dict}\PY{p}{(}\PY{p}{\PYZob{}}\PY{p}{(}\PY{l+m+mi}{1}\PY{p}{,}\PY{l+m+mi}{0}\PY{p}{)}\PY{p}{:}\PY{l+m+mi}{1}\PY{p}{\PYZcb{}}\PY{p}{)}\PY{p}{)}   \PY{c+c1}{\PYZsh{}generator (0,a)}
\PY{n}{b}\PY{o}{=}\PY{n}{H}\PY{p}{(}\PY{n+nb}{dict}\PY{p}{(}\PY{p}{\PYZob{}}\PY{p}{(}\PY{l+m+mi}{0}\PY{p}{,}\PY{l+m+mi}{1}\PY{p}{)}\PY{p}{:}\PY{l+m+mi}{1}\PY{p}{\PYZcb{}}\PY{p}{)}\PY{p}{)}
\PY{n}{am}\PY{o}{=}\PY{n}{H}\PY{p}{(}\PY{n+nb}{dict}\PY{p}{(}\PY{p}{\PYZob{}}\PY{p}{(}\PY{o}{\PYZhy{}}\PY{l+m+mi}{1}\PY{p}{,}\PY{l+m+mi}{0}\PY{p}{)}\PY{p}{:}\PY{l+m+mi}{1}\PY{p}{\PYZcb{}}\PY{p}{)}\PY{p}{)} \PY{c+c1}{\PYZsh{}inverse generators}
\PY{n}{bm}\PY{o}{=}\PY{n}{H}\PY{p}{(}\PY{n+nb}{dict}\PY{p}{(}\PY{p}{\PYZob{}}\PY{p}{(}\PY{l+m+mi}{0}\PY{p}{,}\PY{o}{\PYZhy{}}\PY{l+m+mi}{1}\PY{p}{)}\PY{p}{:}\PY{l+m+mi}{1}\PY{p}{\PYZcb{}}\PY{p}{)}\PY{p}{)}
\end{Verbatim}
\end{tcolorbox}

    \begin{tcolorbox}[breakable, size=fbox, boxrule=1pt, pad at break*=1mm,colback=cellbackground, colframe=cellborder]
\prompt{In}{incolor}{5}{\boxspacing}
\begin{Verbatim}[commandchars=\\\{\}]
\PY{n}{a}\PY{o}{*}\PY{n}{b}\PY{o}{\PYZhy{}}\PY{n}{u}\PY{o}{\PYZca{}}\PY{l+m+mi}{2}\PY{o}{*}\PY{n}{b}\PY{o}{*}\PY{n}{a} \PY{c+c1}{\PYZsh{}check relation}
\end{Verbatim}
\end{tcolorbox}

            \begin{tcolorbox}[breakable, size=fbox, boxrule=.5pt, pad at break*=1mm, opacityfill=0]
\prompt{Out}{outcolor}{5}{\boxspacing}
\begin{Verbatim}[commandchars=\\\{\}]
O
\end{Verbatim}
\end{tcolorbox}
        
    \begin{tcolorbox}[breakable, size=fbox, boxrule=1pt, pad at break*=1mm,colback=cellbackground, colframe=cellborder]
\prompt{In}{incolor}{6}{\boxspacing}
\begin{Verbatim}[commandchars=\\\{\}]
\PY{c+c1}{\PYZsh{} a\PYZhy{}\PYZgt{}a , b \PYZhy{}\PYZgt{} ba\PYZca{}\PYZhy{}1 (T\PYZus{}a action on H)}
\PY{k}{def} \PY{n+nf}{Ha}\PY{p}{(}\PY{n}{h}\PY{p}{:}\PY{n}{HeisEl}\PY{p}{)}\PY{p}{:}
    \PY{n}{d0}\PY{o}{=}\PY{n}{h}\PY{o}{.}\PY{n}{d}
    \PY{n}{h1}\PY{o}{=}\PY{n}{H}\PY{p}{(}\PY{p}{)}
    \PY{k}{for} \PY{n}{k} \PY{o+ow}{in} \PY{n}{d0}\PY{p}{:}
        \PY{n}{i}\PY{o}{=}\PY{n}{k}\PY{p}{[}\PY{l+m+mi}{0}\PY{p}{]}
        \PY{n}{j}\PY{o}{=}\PY{n}{k}\PY{p}{[}\PY{l+m+mi}{1}\PY{p}{]}
        \PY{n}{h1}\PY{o}{+}\PY{o}{=} \PY{n}{H}\PY{p}{(}\PY{p}{\PYZob{}}\PY{p}{(}\PY{n}{i}\PY{o}{\PYZhy{}}\PY{n}{j}\PY{p}{,}\PY{n}{j}\PY{p}{)}\PY{p}{:}\PY{n}{d0}\PY{p}{[}\PY{n}{k}\PY{p}{]}\PY{o}{*}\PY{n}{u}\PY{o}{\PYZca{}}\PY{p}{(}\PY{n}{j}\PY{o}{*}\PY{p}{(}\PY{n}{j}\PY{o}{+}\PY{l+m+mi}{1}\PY{p}{)}\PY{p}{)}\PY{p}{\PYZcb{}}\PY{p}{)}   
    \PY{k}{return} \PY{n}{h1}
\PY{k}{def} \PY{n+nf}{MHa}\PY{p}{(}\PY{n}{M}\PY{p}{)}\PY{p}{:} \PY{c+c1}{\PYZsh{} same on matrices}
    \PY{n}{M1}\PY{o}{=}\PY{n}{matrix}\PY{p}{(}\PY{n}{H}\PY{p}{,}\PY{l+m+mi}{3}\PY{p}{)}
    \PY{k}{for} \PY{n}{i} \PY{o+ow}{in} \PY{n+nb}{range}\PY{p}{(}\PY{l+m+mi}{3}\PY{p}{)}\PY{p}{:}
        \PY{k}{for} \PY{n}{j} \PY{o+ow}{in} \PY{n+nb}{range}\PY{p}{(}\PY{l+m+mi}{3}\PY{p}{)}\PY{p}{:}
            \PY{n}{M1}\PY{p}{[}\PY{n}{i}\PY{p}{,}\PY{n}{j}\PY{p}{]}\PY{o}{=}\PY{n}{Ha}\PY{p}{(}\PY{n}{M}\PY{p}{[}\PY{n}{i}\PY{p}{,}\PY{n}{j}\PY{p}{]}\PY{p}{)}
    \PY{k}{return} \PY{n}{M1}
\end{Verbatim}
\end{tcolorbox}

    \begin{tcolorbox}[breakable, size=fbox, boxrule=1pt, pad at break*=1mm,colback=cellbackground, colframe=cellborder]
\prompt{In}{incolor}{7}{\boxspacing}
\begin{Verbatim}[commandchars=\\\{\}]
\PY{c+c1}{\PYZsh{} a\PYZhy{}\PYZgt{}a , b \PYZhy{}\PYZgt{} ba (T\PYZus{}a\PYZca{}\PYZhy{}1 action on H)}
\PY{k}{def} \PY{n+nf}{Ham}\PY{p}{(}\PY{n}{h}\PY{p}{:}\PY{n}{HeisEl}\PY{p}{)}\PY{p}{:}
    \PY{n}{d0}\PY{o}{=}\PY{n}{h}\PY{o}{.}\PY{n}{d}
    \PY{n}{h1}\PY{o}{=}\PY{n}{H}\PY{p}{(}\PY{p}{)}
    \PY{k}{for} \PY{n}{k} \PY{o+ow}{in} \PY{n}{d0}\PY{p}{:}
        \PY{n}{i}\PY{o}{=}\PY{n}{k}\PY{p}{[}\PY{l+m+mi}{0}\PY{p}{]}
        \PY{n}{j}\PY{o}{=}\PY{n}{k}\PY{p}{[}\PY{l+m+mi}{1}\PY{p}{]}
        \PY{n}{h1}\PY{o}{+}\PY{o}{=} \PY{n}{H}\PY{p}{(}\PY{p}{\PYZob{}}\PY{p}{(}\PY{n}{i}\PY{o}{+}\PY{n}{j}\PY{p}{,}\PY{n}{j}\PY{p}{)}\PY{p}{:}\PY{n}{d0}\PY{p}{[}\PY{n}{k}\PY{p}{]}\PY{o}{*}\PY{n}{u}\PY{o}{\PYZca{}}\PY{p}{(}\PY{o}{\PYZhy{}}\PY{n}{j}\PY{o}{*}\PY{p}{(}\PY{n}{j}\PY{o}{+}\PY{l+m+mi}{1}\PY{p}{)}\PY{p}{)}\PY{p}{\PYZcb{}}\PY{p}{)}   
    \PY{k}{return} \PY{n}{h1}
\PY{k}{def} \PY{n+nf}{MHam}\PY{p}{(}\PY{n}{M}\PY{p}{)}\PY{p}{:} \PY{c+c1}{\PYZsh{} same on matrices}
    \PY{n}{M1}\PY{o}{=}\PY{n}{matrix}\PY{p}{(}\PY{n}{H}\PY{p}{,}\PY{l+m+mi}{3}\PY{p}{)}
    \PY{k}{for} \PY{n}{i} \PY{o+ow}{in} \PY{n+nb}{range}\PY{p}{(}\PY{l+m+mi}{3}\PY{p}{)}\PY{p}{:}
        \PY{k}{for} \PY{n}{j} \PY{o+ow}{in} \PY{n+nb}{range}\PY{p}{(}\PY{l+m+mi}{3}\PY{p}{)}\PY{p}{:}
            \PY{n}{M1}\PY{p}{[}\PY{n}{i}\PY{p}{,}\PY{n}{j}\PY{p}{]}\PY{o}{=}\PY{n}{Ham}\PY{p}{(}\PY{n}{M}\PY{p}{[}\PY{n}{i}\PY{p}{,}\PY{n}{j}\PY{p}{]}\PY{p}{)}
    \PY{k}{return} \PY{n}{M1}
\end{Verbatim}
\end{tcolorbox}

    \begin{tcolorbox}[breakable, size=fbox, boxrule=1pt, pad at break*=1mm,colback=cellbackground, colframe=cellborder]
\prompt{In}{incolor}{8}{\boxspacing}
\begin{Verbatim}[commandchars=\\\{\}]
\PY{n}{Ha}\PY{p}{(}\PY{n}{b}\PY{p}{)}
\end{Verbatim}
\end{tcolorbox}

            \begin{tcolorbox}[breakable, size=fbox, boxrule=.5pt, pad at break*=1mm, opacityfill=0]
\prompt{Out}{outcolor}{8}{\boxspacing}
\begin{Verbatim}[commandchars=\\\{\}]
(u\^{}2)a\^{}-1b\^{}1
\end{Verbatim}
\end{tcolorbox}
        
    \begin{tcolorbox}[breakable, size=fbox, boxrule=1pt, pad at break*=1mm,colback=cellbackground, colframe=cellborder]
\prompt{In}{incolor}{9}{\boxspacing}
\begin{Verbatim}[commandchars=\\\{\}]
\PY{n}{Ham}\PY{p}{(}\PY{n}{Ha}\PY{p}{(}\PY{n}{am}\PY{o}{\PYZca{}}\PY{l+m+mi}{2}\PY{o}{*}\PY{n}{b}\PY{o}{\PYZca{}}\PY{l+m+mi}{3}\PY{p}{)}\PY{p}{)}
\end{Verbatim}
\end{tcolorbox}

            \begin{tcolorbox}[breakable, size=fbox, boxrule=.5pt, pad at break*=1mm, opacityfill=0]
\prompt{Out}{outcolor}{9}{\boxspacing}
\begin{Verbatim}[commandchars=\\\{\}]
(1)a\^{}-2b\^{}3
\end{Verbatim}
\end{tcolorbox}
        
    \begin{tcolorbox}[breakable, size=fbox, boxrule=1pt, pad at break*=1mm,colback=cellbackground, colframe=cellborder]
\prompt{In}{incolor}{10}{\boxspacing}
\begin{Verbatim}[commandchars=\\\{\}]
\PY{c+c1}{\PYZsh{} a\PYZhy{}\PYZgt{}ba , b \PYZhy{}\PYZgt{} b (T\PYZus{}b action on H)}
\PY{k}{def} \PY{n+nf}{Hb}\PY{p}{(}\PY{n}{h}\PY{p}{:}\PY{n}{HeisEl}\PY{p}{)}\PY{p}{:} 
    \PY{n}{d0}\PY{o}{=}\PY{n}{h}\PY{o}{.}\PY{n}{d}
    \PY{n}{h1}\PY{o}{=}\PY{n}{H}\PY{p}{(}\PY{p}{)}
    \PY{k}{for} \PY{n}{k} \PY{o+ow}{in} \PY{n}{d0}\PY{p}{:}
        \PY{n}{i}\PY{o}{=}\PY{n}{k}\PY{p}{[}\PY{l+m+mi}{0}\PY{p}{]}
        \PY{n}{j}\PY{o}{=}\PY{n}{k}\PY{p}{[}\PY{l+m+mi}{1}\PY{p}{]}
        \PY{n}{h1}\PY{o}{+}\PY{o}{=} \PY{n}{H}\PY{p}{(}\PY{p}{\PYZob{}}\PY{p}{(}\PY{n}{i}\PY{p}{,}\PY{n}{i}\PY{o}{+}\PY{n}{j}\PY{p}{)}\PY{p}{:}\PY{n}{d0}\PY{p}{[}\PY{n}{k}\PY{p}{]}\PY{o}{*}\PY{n}{u}\PY{o}{\PYZca{}}\PY{p}{(}\PY{o}{\PYZhy{}}\PY{n}{i}\PY{o}{*}\PY{p}{(}\PY{n}{i}\PY{o}{+}\PY{l+m+mi}{1}\PY{p}{)}\PY{p}{)}\PY{p}{\PYZcb{}}\PY{p}{)}   
    \PY{k}{return} \PY{n}{h1}
\PY{k}{def} \PY{n+nf}{MHb}\PY{p}{(}\PY{n}{M}\PY{p}{)}\PY{p}{:} \PY{c+c1}{\PYZsh{} same on matrices}
    \PY{n}{M1}\PY{o}{=}\PY{n}{matrix}\PY{p}{(}\PY{n}{H}\PY{p}{,}\PY{l+m+mi}{3}\PY{p}{)}
    \PY{k}{for} \PY{n}{i} \PY{o+ow}{in} \PY{n+nb}{range}\PY{p}{(}\PY{l+m+mi}{3}\PY{p}{)}\PY{p}{:}
        \PY{k}{for} \PY{n}{j} \PY{o+ow}{in} \PY{n+nb}{range}\PY{p}{(}\PY{l+m+mi}{3}\PY{p}{)}\PY{p}{:}
            \PY{n}{M1}\PY{p}{[}\PY{n}{i}\PY{p}{,}\PY{n}{j}\PY{p}{]}\PY{o}{=}\PY{n}{Hb}\PY{p}{(}\PY{n}{M}\PY{p}{[}\PY{n}{i}\PY{p}{,}\PY{n}{j}\PY{p}{]}\PY{p}{)}
    \PY{k}{return} \PY{n}{M1}
\end{Verbatim}
\end{tcolorbox}

    \begin{tcolorbox}[breakable, size=fbox, boxrule=1pt, pad at break*=1mm,colback=cellbackground, colframe=cellborder]
\prompt{In}{incolor}{11}{\boxspacing}
\begin{Verbatim}[commandchars=\\\{\}]
\PY{k}{def} \PY{n+nf}{Hbm}\PY{p}{(}\PY{n}{h}\PY{p}{:}\PY{n}{HeisEl}\PY{p}{)}\PY{p}{:} \PY{c+c1}{\PYZsh{}T\PYZus{}b\PYZca{}\PYZhy{}1 action on H}
    \PY{n}{d0}\PY{o}{=}\PY{n}{h}\PY{o}{.}\PY{n}{d}
    \PY{n}{h1}\PY{o}{=}\PY{n}{H}\PY{p}{(}\PY{p}{)}
    \PY{k}{for} \PY{n}{k} \PY{o+ow}{in} \PY{n}{d0}\PY{p}{:}
        \PY{n}{i}\PY{o}{=}\PY{n}{k}\PY{p}{[}\PY{l+m+mi}{0}\PY{p}{]}
        \PY{n}{j}\PY{o}{=}\PY{n}{k}\PY{p}{[}\PY{l+m+mi}{1}\PY{p}{]}
        \PY{n}{h1}\PY{o}{+}\PY{o}{=} \PY{n}{H}\PY{p}{(}\PY{p}{\PYZob{}}\PY{p}{(}\PY{n}{i}\PY{p}{,}\PY{o}{\PYZhy{}}\PY{n}{i}\PY{o}{+}\PY{n}{j}\PY{p}{)}\PY{p}{:}\PY{n}{d0}\PY{p}{[}\PY{n}{k}\PY{p}{]}\PY{o}{*}\PY{n}{u}\PY{o}{\PYZca{}}\PY{p}{(}\PY{n}{i}\PY{o}{*}\PY{p}{(}\PY{n}{i}\PY{o}{+}\PY{l+m+mi}{1}\PY{p}{)}\PY{p}{)}\PY{p}{\PYZcb{}}\PY{p}{)}   
    \PY{k}{return} \PY{n}{h1}
\PY{k}{def} \PY{n+nf}{MHbm}\PY{p}{(}\PY{n}{M}\PY{p}{)}\PY{p}{:} \PY{c+c1}{\PYZsh{} same on matrices}
    \PY{n}{M1}\PY{o}{=}\PY{n}{matrix}\PY{p}{(}\PY{n}{H}\PY{p}{,}\PY{l+m+mi}{3}\PY{p}{)}
    \PY{k}{for} \PY{n}{i} \PY{o+ow}{in} \PY{n+nb}{range}\PY{p}{(}\PY{l+m+mi}{3}\PY{p}{)}\PY{p}{:}
        \PY{k}{for} \PY{n}{j} \PY{o+ow}{in} \PY{n+nb}{range}\PY{p}{(}\PY{l+m+mi}{3}\PY{p}{)}\PY{p}{:}
            \PY{n}{M1}\PY{p}{[}\PY{n}{i}\PY{p}{,}\PY{n}{j}\PY{p}{]}\PY{o}{=}\PY{n}{Hbm}\PY{p}{(}\PY{n}{M}\PY{p}{[}\PY{n}{i}\PY{p}{,}\PY{n}{j}\PY{p}{]}\PY{p}{)}
    \PY{k}{return} \PY{n}{M1}
\end{Verbatim}
\end{tcolorbox}

    \begin{tcolorbox}[breakable, size=fbox, boxrule=1pt, pad at break*=1mm,colback=cellbackground, colframe=cellborder]
\prompt{In}{incolor}{12}{\boxspacing}
\begin{Verbatim}[commandchars=\\\{\}]
\PY{n}{Hb}\PY{p}{(}\PY{n}{a}\PY{p}{)}
\end{Verbatim}
\end{tcolorbox}

            \begin{tcolorbox}[breakable, size=fbox, boxrule=.5pt, pad at break*=1mm, opacityfill=0]
\prompt{Out}{outcolor}{12}{\boxspacing}
\begin{Verbatim}[commandchars=\\\{\}]
(u\^{}-2)a\^{}1b\^{}1
\end{Verbatim}
\end{tcolorbox}
        
    \begin{tcolorbox}[breakable, size=fbox, boxrule=1pt, pad at break*=1mm,colback=cellbackground, colframe=cellborder]
\prompt{In}{incolor}{13}{\boxspacing}
\begin{Verbatim}[commandchars=\\\{\}]
\PY{n}{Hbm}\PY{p}{(}\PY{n}{Hb}\PY{p}{(}\PY{n}{am}\PY{o}{\PYZca{}}\PY{l+m+mi}{3}\PY{o}{*}\PY{n}{b}\PY{o}{\PYZca{}}\PY{l+m+mi}{2}\PY{p}{)}\PY{p}{)}
\end{Verbatim}
\end{tcolorbox}

            \begin{tcolorbox}[breakable, size=fbox, boxrule=.5pt, pad at break*=1mm, opacityfill=0]
\prompt{Out}{outcolor}{13}{\boxspacing}
\begin{Verbatim}[commandchars=\\\{\}]
(1)a\^{}-3b\^{}2
\end{Verbatim}
\end{tcolorbox}
        
    \begin{tcolorbox}[breakable, size=fbox, boxrule=1pt, pad at break*=1mm,colback=cellbackground, colframe=cellborder]
\prompt{In}{incolor}{14}{\boxspacing}
\begin{Verbatim}[commandchars=\\\{\}]
\PY{k}{def} \PY{n+nf}{Hab}\PY{p}{(}\PY{n}{h}\PY{p}{)}\PY{p}{:}       \PY{c+c1}{\PYZsh{}other actions}
    \PY{k}{return} \PY{n}{Ha}\PY{p}{(}\PY{n}{Hb}\PY{p}{(}\PY{n}{h}\PY{p}{)}\PY{p}{)}
\PY{k}{def} \PY{n+nf}{Hba}\PY{p}{(}\PY{n}{h}\PY{p}{)}\PY{p}{:}
    \PY{k}{return} \PY{n}{Hb}\PY{p}{(}\PY{n}{Ha}\PY{p}{(}\PY{n}{h}\PY{p}{)}\PY{p}{)}
\PY{k}{def} \PY{n+nf}{Haba}\PY{p}{(}\PY{n}{h}\PY{p}{)}\PY{p}{:}
    \PY{k}{return} \PY{n}{Ha}\PY{p}{(}\PY{n}{Hba}\PY{p}{(}\PY{n}{h}\PY{p}{)}\PY{p}{)}
\PY{k}{def} \PY{n+nf}{Hbab}\PY{p}{(}\PY{n}{h}\PY{p}{)}\PY{p}{:}
    \PY{k}{return} \PY{n}{Hb}\PY{p}{(}\PY{n}{Hab}\PY{p}{(}\PY{n}{h}\PY{p}{)}\PY{p}{)}
\PY{k}{def} \PY{n+nf}{Hs}\PY{p}{(}\PY{n}{h}\PY{p}{)}\PY{p}{:}
    \PY{k}{return}\PY{p}{(}\PY{n}{Haba}\PY{p}{(}\PY{n}{Haba}\PY{p}{(}\PY{n}{h}\PY{p}{)}\PY{p}{)}\PY{p}{)}
\end{Verbatim}
\end{tcolorbox}

    \begin{tcolorbox}[breakable, size=fbox, boxrule=1pt, pad at break*=1mm,colback=cellbackground, colframe=cellborder]
\prompt{In}{incolor}{15}{\boxspacing}
\begin{Verbatim}[commandchars=\\\{\}]
\PY{k}{def} \PY{n+nf}{MHab}\PY{p}{(}\PY{n}{M}\PY{p}{)}\PY{p}{:}    \PY{c+c1}{\PYZsh{}same on matrices}
    \PY{k}{return} \PY{n}{MHa}\PY{p}{(}\PY{n}{MHb}\PY{p}{(}\PY{n}{M}\PY{p}{)}\PY{p}{)}
\PY{k}{def} \PY{n+nf}{MHba}\PY{p}{(}\PY{n}{M}\PY{p}{)}\PY{p}{:}
    \PY{k}{return} \PY{n}{MHb}\PY{p}{(}\PY{n}{MHa}\PY{p}{(}\PY{n}{M}\PY{p}{)}\PY{p}{)}
\PY{k}{def} \PY{n+nf}{MHaba}\PY{p}{(}\PY{n}{M}\PY{p}{)}\PY{p}{:}
    \PY{k}{return} \PY{n}{MHa}\PY{p}{(}\PY{n}{MHba}\PY{p}{(}\PY{n}{M}\PY{p}{)}\PY{p}{)}
\PY{k}{def} \PY{n+nf}{MHbab}\PY{p}{(}\PY{n}{M}\PY{p}{)}\PY{p}{:}
    \PY{k}{return} \PY{n}{MHb}\PY{p}{(}\PY{n}{MHab}\PY{p}{(}\PY{n}{M}\PY{p}{)}\PY{p}{)}
\PY{k}{def} \PY{n+nf}{MHs}\PY{p}{(}\PY{n}{M}\PY{p}{)}\PY{p}{:}
    \PY{k}{return}\PY{p}{(}\PY{n}{MHaba}\PY{p}{(}\PY{n}{MHaba}\PY{p}{(}\PY{n}{M}\PY{p}{)}\PY{p}{)}\PY{p}{)}
\end{Verbatim}
\end{tcolorbox}

    \begin{tcolorbox}[breakable, size=fbox, boxrule=1pt, pad at break*=1mm,colback=cellbackground, colframe=cellborder]
\prompt{In}{incolor}{16}{\boxspacing}
\begin{Verbatim}[commandchars=\\\{\}]
\PY{k}{def} \PY{n+nf}{Habm}\PY{p}{(}\PY{n}{h}\PY{p}{)}\PY{p}{:}       \PY{c+c1}{\PYZsh{}other actions}
    \PY{k}{return} \PY{n}{Hbm}\PY{p}{(}\PY{n}{Ham}\PY{p}{(}\PY{n}{h}\PY{p}{)}\PY{p}{)}
\PY{k}{def} \PY{n+nf}{Hbam}\PY{p}{(}\PY{n}{h}\PY{p}{)}\PY{p}{:}
    \PY{k}{return} \PY{n}{Ham}\PY{p}{(}\PY{n}{Hbm}\PY{p}{(}\PY{n}{h}\PY{p}{)}\PY{p}{)}
\PY{k}{def} \PY{n+nf}{Habam}\PY{p}{(}\PY{n}{h}\PY{p}{)}\PY{p}{:}
    \PY{k}{return} \PY{n}{Ham}\PY{p}{(}\PY{n}{Habm}\PY{p}{(}\PY{n}{h}\PY{p}{)}\PY{p}{)}
\PY{k}{def} \PY{n+nf}{Hbabm}\PY{p}{(}\PY{n}{h}\PY{p}{)}\PY{p}{:}
    \PY{k}{return} \PY{n}{Hbm}\PY{p}{(}\PY{n}{Hbam}\PY{p}{(}\PY{n}{h}\PY{p}{)}\PY{p}{)}
\PY{k}{def} \PY{n+nf}{Hsm}\PY{p}{(}\PY{n}{h}\PY{p}{)}\PY{p}{:}
    \PY{k}{return}\PY{p}{(}\PY{n}{Habam}\PY{p}{(}\PY{n}{Habam}\PY{p}{(}\PY{n}{h}\PY{p}{)}\PY{p}{)}\PY{p}{)}
\end{Verbatim}
\end{tcolorbox}

    \begin{tcolorbox}[breakable, size=fbox, boxrule=1pt, pad at break*=1mm,colback=cellbackground, colframe=cellborder]
\prompt{In}{incolor}{17}{\boxspacing}
\begin{Verbatim}[commandchars=\\\{\}]
\PY{k}{def} \PY{n+nf}{MHabm}\PY{p}{(}\PY{n}{M}\PY{p}{)}\PY{p}{:}    \PY{c+c1}{\PYZsh{}same on matrices}
    \PY{k}{return} \PY{n}{MHbm}\PY{p}{(}\PY{n}{MHam}\PY{p}{(}\PY{n}{M}\PY{p}{)}\PY{p}{)}
\PY{k}{def} \PY{n+nf}{MHbam}\PY{p}{(}\PY{n}{M}\PY{p}{)}\PY{p}{:}
    \PY{k}{return} \PY{n}{MHam}\PY{p}{(}\PY{n}{MHbm}\PY{p}{(}\PY{n}{M}\PY{p}{)}\PY{p}{)}
\PY{k}{def} \PY{n+nf}{MHabam}\PY{p}{(}\PY{n}{M}\PY{p}{)}\PY{p}{:}
    \PY{k}{return} \PY{n}{MHam}\PY{p}{(}\PY{n}{MHabm}\PY{p}{(}\PY{n}{M}\PY{p}{)}\PY{p}{)}
\PY{k}{def} \PY{n+nf}{MHbabm}\PY{p}{(}\PY{n}{M}\PY{p}{)}\PY{p}{:}
    \PY{k}{return} \PY{n}{MHbm}\PY{p}{(}\PY{n}{MHbam}\PY{p}{(}\PY{n}{M}\PY{p}{)}\PY{p}{)}
\PY{k}{def} \PY{n+nf}{MHsm}\PY{p}{(}\PY{n}{M}\PY{p}{)}\PY{p}{:}
    \PY{k}{return}\PY{p}{(}\PY{n}{MHabam}\PY{p}{(}\PY{n}{MHabam}\PY{p}{(}\PY{n}{M}\PY{p}{)}\PY{p}{)}\PY{p}{)}
\end{Verbatim}
\end{tcolorbox}

    \begin{tcolorbox}[breakable, size=fbox, boxrule=1pt, pad at break*=1mm,colback=cellbackground, colframe=cellborder]
\prompt{In}{incolor}{18}{\boxspacing}
\begin{Verbatim}[commandchars=\\\{\}]
\PY{n}{Ma}\PY{o}{=}\PY{n}{matrix}\PY{p}{(}\PY{p}{[}\PY{p}{[}\PY{n}{H}\PY{p}{(}\PY{l+m+mi}{1}\PY{p}{)}\PY{p}{,}\PY{n}{u}\PY{o}{\PYZca{}}\PY{l+m+mi}{2}\PY{o}{*}\PY{n}{a}\PY{o}{\PYZca{}}\PY{p}{(}\PY{l+m+mi}{2}\PY{p}{)}\PY{o}{*}\PY{n}{bm}\PY{o}{\PYZca{}}\PY{l+m+mi}{2}\PY{p}{,}\PY{p}{(}\PY{n}{H}\PY{p}{(}\PY{n}{u}\PY{o}{\PYZca{}}\PY{p}{(}\PY{o}{\PYZhy{}}\PY{l+m+mi}{1}\PY{p}{)}\PY{p}{)}\PY{o}{\PYZhy{}}\PY{n}{H}\PY{p}{(}\PY{l+m+mi}{1}\PY{p}{)}\PY{p}{)}\PY{o}{*}\PY{n}{a}\PY{o}{*}\PY{n}{bm}\PY{p}{]}\PY{p}{,}
\qquad \PY{p}{[}\PY{n}{H}\PY{p}{(}\PY{l+m+mi}{0}\PY{p}{)}\PY{p}{,}\PY{n}{H}\PY{p}{(}\PY{l+m+mi}{1}\PY{p}{)}\PY{p}{,}\PY{n}{H}\PY{p}{(}\PY{l+m+mi}{0}\PY{p}{)}\PY{p}{]}\PY{p}{,}\PY{p}{[}\PY{n}{H}\PY{p}{(}\PY{l+m+mi}{0}\PY{p}{)}\PY{p}{,}\PY{n}{H}\PY{p}{(}\PY{o}{\PYZhy{}}\PY{l+m+mi}{1}\PY{p}{)}\PY{o}{*}\PY{n}{a}\PY{o}{*}\PY{n}{bm}\PY{p}{,}\PY{n}{H}\PY{p}{(}\PY{l+m+mi}{1}\PY{p}{)}\PY{p}{]}\PY{p}{]}\PY{p}{)}
\end{Verbatim}
\end{tcolorbox}

    \begin{tcolorbox}[breakable, size=fbox, boxrule=1pt, pad at break*=1mm,colback=cellbackground, colframe=cellborder]
\prompt{In}{incolor}{19}{\boxspacing}
\begin{Verbatim}[commandchars=\\\{\}]
\PY{o}{\PYZpc{}}\PY{k}{display} latex
\PY{n}{Ma} \PY{c+c1}{\PYZsh{}Ta action}
\end{Verbatim}
\end{tcolorbox}

\prompt{Out}{outcolor}{19}{}
    
    $$\newcommand{\Bold}[1]{\mathbf{#1}}\left(\begin{array}{rrr}
 1 &  u^{2} a^{ 2 }b^{ -2 } & ( -1 + u^{-1} )a^{ 1 }b^{ -1 } \\
0 &  1 & 0 \\
0 & -a^{ 1 }b^{ -1 } &  1
\end{array}\right)$$

    \begin{tcolorbox}[breakable, size=fbox, boxrule=1pt, pad at break*=1mm,colback=cellbackground, colframe=cellborder]
\prompt{In}{incolor}{20}{\boxspacing}
\begin{Verbatim}[commandchars=\\\{\}]
\PY{n}{Mb}\PY{o}{=}\PY{n}{matrix}\PY{p}{(}\PY{p}{[}\PY{p}{[}\PY{n}{H}\PY{p}{(}\PY{n}{u}\PY{o}{\PYZca{}}\PY{p}{(}\PY{o}{\PYZhy{}}\PY{l+m+mi}{2}\PY{p}{)}\PY{p}{)}\PY{o}{*}\PY{n}{bm}\PY{o}{\PYZca{}}\PY{l+m+mi}{2}\PY{p}{,}\PY{n}{H}\PY{p}{(}\PY{l+m+mi}{0}\PY{p}{)}\PY{p}{,}\PY{n}{H}\PY{p}{(}\PY{l+m+mi}{0}\PY{p}{)}\PY{p}{]}\PY{p}{,}\PY{p}{[}\PY{n}{H}\PY{p}{(}\PY{o}{\PYZhy{}}\PY{n}{u}\PY{o}{\PYZca{}}\PY{p}{(}\PY{o}{\PYZhy{}}\PY{l+m+mi}{1}\PY{p}{)}\PY{p}{)}\PY{p}{,}\PY{n}{H}\PY{p}{(}\PY{l+m+mi}{1}\PY{p}{)}\PY{p}{,}\PY{n}{H}\PY{p}{(}\PY{l+m+mi}{1}\PY{o}{\PYZhy{}}\PY{n}{u}\PY{o}{\PYZca{}}\PY{p}{(}\PY{o}{\PYZhy{}}\PY{l+m+mi}{1}\PY{p}{)}\PY{p}{)}\PY{p}{]}\PY{p}{,}
\qquad \PY{p}{[}\PY{n}{H}\PY{p}{(}\PY{o}{\PYZhy{}}\PY{n}{u}\PY{o}{\PYZca{}}\PY{p}{(}\PY{o}{\PYZhy{}}\PY{l+m+mi}{1}\PY{p}{)}\PY{p}{)}\PY{o}{*}\PY{n}{bm}\PY{p}{,}\PY{n}{H}\PY{p}{(}\PY{l+m+mi}{0}\PY{p}{)}\PY{p}{,}\PY{n}{bm}\PY{p}{]}\PY{p}{]}\PY{p}{)}
\end{Verbatim}
\end{tcolorbox}

    \begin{tcolorbox}[breakable, size=fbox, boxrule=1pt, pad at break*=1mm,colback=cellbackground, colframe=cellborder]
\prompt{In}{incolor}{21}{\boxspacing}
\begin{Verbatim}[commandchars=\\\{\}]
\PY{n}{Mb} \PY{c+c1}{\PYZsh{}Tb action}
\end{Verbatim}
\end{tcolorbox}

\prompt{Out}{outcolor}{21}{}
    
    $$\newcommand{\Bold}[1]{\mathbf{#1}}\left(\begin{array}{rrr}
 u^{-2} b^{ -2 } & 0 & 0 \\
 -u^{-1} &  1 &  1 - u^{-1} \\
 -u^{-1} b^{ -1 } & 0 & b^{ -1 }
\end{array}\right)$$

    \begin{tcolorbox}[breakable, size=fbox, boxrule=1pt, pad at break*=1mm,colback=cellbackground, colframe=cellborder]
\prompt{In}{incolor}{22}{\boxspacing}
\begin{Verbatim}[commandchars=\\\{\}]
\PY{n}{MHa}\PY{p}{(}\PY{n}{Mb}\PY{p}{)} \PY{c+c1}{\PYZsh{} Ta shifted action of Tb}
\end{Verbatim}
\end{tcolorbox}

\prompt{Out}{outcolor}{22}{}
    
    $$\newcommand{\Bold}[1]{\mathbf{#1}}\left(\begin{array}{rrr}
a^{ 2 }b^{ -2 } & 0 & 0 \\
 -u^{-1} &  1 &  1 - u^{-1} \\
 -u^{-1} a^{ 1 }b^{ -1 } & 0 & a^{ 1 }b^{ -1 }
\end{array}\right)$$

    \begin{tcolorbox}[breakable, size=fbox, boxrule=1pt, pad at break*=1mm,colback=cellbackground, colframe=cellborder]
\prompt{In}{incolor}{23}{\boxspacing}
\begin{Verbatim}[commandchars=\\\{\}]
\PY{n}{MHb}\PY{p}{(}\PY{n}{Ma}\PY{p}{)} \PY{c+c1}{\PYZsh{}Tb shifted action of Ta}
\end{Verbatim}
\end{tcolorbox}

\prompt{Out}{outcolor}{23}{}
    
    $$\newcommand{\Bold}[1]{\mathbf{#1}}\left(\begin{array}{rrr}
 1 &  u^{-4} a^{ 2 } & ( -u^{-2} + u^{-3} )a^{ 1 } \\
0 &  1 & 0 \\
0 &  -u^{-2} a^{ 1 } &  1
\end{array}\right)$$

    \begin{tcolorbox}[breakable, size=fbox, boxrule=1pt, pad at break*=1mm,colback=cellbackground, colframe=cellborder]
\prompt{In}{incolor}{24}{\boxspacing}
\begin{Verbatim}[commandchars=\\\{\}]
\PY{n}{MHab}\PY{p}{(}\PY{n}{Ma}\PY{p}{)} \PY{c+c1}{\PYZsh{}TaTb shifted action of Ta}
\end{Verbatim}
\end{tcolorbox}

\prompt{Out}{outcolor}{24}{}
    
    $$\newcommand{\Bold}[1]{\mathbf{#1}}\left(\begin{array}{rrr}
 1 &  u^{-4} a^{ 2 } & ( -u^{-2} + u^{-3} )a^{ 1 } \\
0 &  1 & 0 \\
0 &  -u^{-2} a^{ 1 } &  1
\end{array}\right)$$

    \begin{tcolorbox}[breakable, size=fbox, boxrule=1pt, pad at break*=1mm,colback=cellbackground, colframe=cellborder]
\prompt{In}{incolor}{25}{\boxspacing}
\begin{Verbatim}[commandchars=\\\{\}]
\PY{n}{MHba}\PY{p}{(}\PY{n}{Mb}\PY{p}{)} \PY{c+c1}{\PYZsh{}TbTa shifted action of Tb}
\end{Verbatim}
\end{tcolorbox}

\prompt{Out}{outcolor}{25}{}
    
    $$\newcommand{\Bold}[1]{\mathbf{#1}}\left(\begin{array}{rrr}
 u^{-6} a^{ 2 } & 0 & 0 \\
 -u^{-1} &  1 &  1 - u^{-1} \\
 -u^{-3} a^{ 1 } & 0 &  u^{-2} a^{ 1 }
\end{array}\right)$$

    \begin{tcolorbox}[breakable, size=fbox, boxrule=1pt, pad at break*=1mm,colback=cellbackground, colframe=cellborder]
\prompt{In}{incolor}{26}{\boxspacing}
\begin{Verbatim}[commandchars=\\\{\}]
\PY{n}{X}\PY{o}{=}\PY{n}{Ma}\PY{o}{*}\PY{n}{MHa}\PY{p}{(}\PY{n}{Mb}\PY{p}{)}\PY{o}{*}\PY{n}{MHab}\PY{p}{(}\PY{n}{Ma}\PY{p}{)} \PY{c+c1}{\PYZsh{}action of TaTbTa}
\PY{n}{X}
\end{Verbatim}
\end{tcolorbox}

\prompt{Out}{outcolor}{26}{}
    
    $$\newcommand{\Bold}[1]{\mathbf{#1}}\left(\begin{array}{rrr}
0 &  u^{2} a^{ 2 }b^{ -2 } & 0 \\
 -u^{-1} &  -u^{-5} a^{ 2 } + 1 +( -u^{-2} + u^{-3} )a^{ 1 } & ( u^{-3} - u^{-4} )a^{ 1 } + 1 - u^{-1} \\
0 & -a^{ 1 }b^{ -1 }  -u^{-1} a^{ 2 }b^{ -1 } &  u^{-1} a^{ 1 }b^{ -1 }
\end{array}\right)$$

    \begin{tcolorbox}[breakable, size=fbox, boxrule=1pt, pad at break*=1mm,colback=cellbackground, colframe=cellborder]
\prompt{In}{incolor}{27}{\boxspacing}
\begin{Verbatim}[commandchars=\\\{\}]
\PY{n}{Y}\PY{o}{=}\PY{n}{Mb}\PY{o}{*}\PY{n}{MHb}\PY{p}{(}\PY{n}{Ma}\PY{p}{)}\PY{o}{*}\PY{n}{MHba}\PY{p}{(}\PY{n}{Mb}\PY{p}{)} \PY{c+c1}{\PYZsh{}action of TbTaTb}
\PY{n}{Y}
\end{Verbatim}
\end{tcolorbox}

\prompt{Out}{outcolor}{27}{}
    
    $$\newcommand{\Bold}[1]{\mathbf{#1}}\left(\begin{array}{rrr}
0 &  u^{2} a^{ 2 }b^{ -2 } & 0 \\
 -u^{-1} &  -u^{-5} a^{ 2 } + 1 +( -u^{-2} + u^{-3} )a^{ 1 } &  1 - u^{-1} +( u^{-3} - u^{-4} )a^{ 1 } \\
0 &  -u^{-1} a^{ 2 }b^{ -1 } -a^{ 1 }b^{ -1 } &  u^{-1} a^{ 1 }b^{ -1 }
\end{array}\right)$$

    \begin{tcolorbox}[breakable, size=fbox, boxrule=1pt, pad at break*=1mm,colback=cellbackground, colframe=cellborder]
\prompt{In}{incolor}{28}{\boxspacing}
\begin{Verbatim}[commandchars=\\\{\}]
\PY{n}{X}\PY{o}{\PYZhy{}}\PY{n}{Y} \PY{c+c1}{\PYZsh{}check braid relation TaTbTa=TbTaTb}
\end{Verbatim}
\end{tcolorbox}

\prompt{Out}{outcolor}{28}{}
    
    $$\newcommand{\Bold}[1]{\mathbf{#1}}\left(\begin{array}{rrr}
0 & 0 & 0 \\
0 & 0 & 0 \\
0 & 0 & 0
\end{array}\right)$$

    \begin{tcolorbox}[breakable, size=fbox, boxrule=1pt, pad at break*=1mm,colback=cellbackground, colframe=cellborder]
\prompt{In}{incolor}{29}{\boxspacing}
\begin{Verbatim}[commandchars=\\\{\}]
\PY{n}{Z}\PY{o}{=}\PY{n}{X}\PY{o}{*}\PY{n}{MHaba}\PY{p}{(}\PY{n}{X}\PY{p}{)} \PY{c+c1}{\PYZsh{}action of (TaTbTa)\PYZca{}2}
\end{Verbatim}
\end{tcolorbox}

    \begin{tcolorbox}[breakable, size=fbox, boxrule=1pt, pad at break*=1mm,colback=cellbackground, colframe=cellborder]
\prompt{In}{incolor}{30}{\boxspacing}
\begin{Verbatim}[commandchars=\\\{\}]
\PY{n}{Z}\PY{p}{[}\PY{p}{:}\PY{p}{,}\PY{l+m+mi}{0}\PY{p}{]} \PY{c+c1}{\PYZsh{} first column}
\end{Verbatim}
\end{tcolorbox}

\prompt{Out}{outcolor}{30}{}
    
    $$\newcommand{\Bold}[1]{\mathbf{#1}}\left(\begin{array}{r}
 -u a^{ 2 }b^{ -2 } \\
 u^{-6} a^{ 2 } + -u^{-1} +( u^{-3} - u^{-4} )a^{ 1 } \\
 u^{-1} a^{ 1 }b^{ -1 } + u^{-2} a^{ 2 }b^{ -1 }
\end{array}\right)$$

    \begin{tcolorbox}[breakable, size=fbox, boxrule=1pt, pad at break*=1mm,colback=cellbackground, colframe=cellborder]
\prompt{In}{incolor}{31}{\boxspacing}
\begin{Verbatim}[commandchars=\\\{\}]
\PY{n}{Z}\PY{p}{[}\PY{p}{:}\PY{p}{,}\PY{l+m+mi}{1}\PY{p}{]}
\end{Verbatim}
\end{tcolorbox}

\prompt{Out}{outcolor}{31}{}
    
    $$\newcommand{\Bold}[1]{\mathbf{#1}}\left(\begin{array}{r}
 -u^{-3} a^{ 2 } + u^{2} a^{ 2 }b^{ -2 } +( -1 + u^{-1} )a^{ 2 }b^{ -1 } \\
 -u^{-5} b^{ 2 } + -u^{-5} a^{ 2 } + 1 +( -u^{-2} + u^{-3} )a^{ 1 } +( -u^{-2} + u^{-3} )b^{ 1 } +( -u^{-5} + u^{-6} )a^{ 1 }b^{ 1 } \\
 u^{-5} a^{ 1 }b^{ 1 } -a^{ 1 }b^{ -1 } + -u^{-1} a^{ 2 }b^{ -1 } +( u^{-2} - u^{-3} )a^{ 1 } + -u^{-4} a^{ 2 }
\end{array}\right)$$

    \begin{tcolorbox}[breakable, size=fbox, boxrule=1pt, pad at break*=1mm,colback=cellbackground, colframe=cellborder]
\prompt{In}{incolor}{32}{\boxspacing}
\begin{Verbatim}[commandchars=\\\{\}]
\PY{n}{Z}\PY{p}{[}\PY{p}{:}\PY{p}{,}\PY{l+m+mi}{2}\PY{p}{]}
\end{Verbatim}
\end{tcolorbox}

\prompt{Out}{outcolor}{32}{}
    
    $$\newcommand{\Bold}[1]{\mathbf{#1}}\left(\begin{array}{r}
( u^{-1} - u^{-2} )a^{ 2 }b^{ -1 } +( u^{2} - u )a^{ 2 }b^{ -2 } \\
( u^{-3} - u^{-4} )b^{ 1 } +( u^{-6} - u^{-7} )a^{ 1 }b^{ 1 } +( -u^{-5} + u^{-6} )a^{ 2 } + 1 - u^{-1} +( -u^{-2} + 2 u^{-3} - u^{-4} )a^{ 1 } \\
( -u^{-3} + u^{-4} )a^{ 1 } + u^{-5} a^{ 2 } +( -1 + u^{-1} )a^{ 1 }b^{ -1 } +( -u^{-1} + u^{-2} )a^{ 2 }b^{ -1 }
\end{array}\right)$$

    \begin{tcolorbox}[breakable, size=fbox, boxrule=1pt, pad at break*=1mm,colback=cellbackground, colframe=cellborder]
\prompt{In}{incolor}{33}{\boxspacing}
\begin{Verbatim}[commandchars=\\\{\}]
\PY{n}{ZZ}\PY{o}{=}\PY{n}{Z}\PY{o}{*}\PY{n}{MHs}\PY{p}{(}\PY{n}{Z}\PY{p}{)} \PY{c+c1}{\PYZsh{}action of Tc=(TaTbTa)\PYZca{}4}
\end{Verbatim}
\end{tcolorbox}

    \begin{tcolorbox}[breakable, size=fbox, boxrule=1pt, pad at break*=1mm,colback=cellbackground, colframe=cellborder]
\prompt{In}{incolor}{34}{\boxspacing}
\begin{Verbatim}[commandchars=\\\{\}]
\PY{n}{ZZ}\PY{p}{[}\PY{p}{:}\PY{p}{,}\PY{l+m+mi}{0}\PY{p}{]}
\end{Verbatim}
\end{tcolorbox}

\prompt{Out}{outcolor}{34}{}
    
    $$\tiny\newcommand{\Bold}[1]{\mathbf{#1}}\left(\begin{array}{r}
 u^{-8} b^{ -2 } + u^{-4} a^{ 2 } + -u a^{ 2 }b^{ -2 } +( u^{-1} - u^{-2} )a^{ 2 }b^{ -1 } +( u^{-3} - u^{-4} )a^{ 1 }b^{ -2 } +( u^{-4} - u^{-5} )a^{ 1 }b^{ -1 } \\[5pt]
 -u^{-1} - u^{-3} + 2 u^{-4} - u^{-5} - u^{-7} + u^{-2} a^{ -2 } +( u^{-1} - u^{-2} - u^{-4} + u^{-5} )a^{ -1 } + u^{-6} a^{ 2 } +( u^{-3} - u^{-4} - u^{-6} + u^{-7} )a^{ 1 } \\[5pt]
 -u^{-6} a^{ -1 }b^{ -1 } +( -u^{-3} + u^{-4} - u^{-7} )b^{ -1 } + -u^{-4} +( u^{-1} - u^{-4} + u^{-5} )a^{ 1 }b^{ -1 } + u^{-2} a^{ 2 }b^{ -1 } +( -u^{-3} + u^{-6} )a^{ 1 } + u^{-5} a^{ 2 }
\end{array}\right)$$

    \begin{tcolorbox}[breakable, size=fbox, boxrule=1pt, pad at break*=1mm,colback=cellbackground, colframe=cellborder]
\prompt{In}{incolor}{35}{\boxspacing}
\begin{Verbatim}[commandchars=\\\{\}]
\PY{n}{ZZ}\PY{p}{[}\PY{p}{:}\PY{p}{,}\PY{l+m+mi}{1}\PY{p}{]}
\end{Verbatim}
\end{tcolorbox}

\prompt{Out}{outcolor}{35}{}

    $$\tiny\newcommand{\Bold}[1]{\mathbf{#1}}\left(\begin{array}{r}
( u^{2} + 1 - 2 u^{-1} + u^{-2} + u^{-4} )a^{ 2 }b^{ -2 } + -u a^{ 2 }b^{ -4 } +( -u^{2} + u + u^{-1} - u^{-2} )a^{ 2 }b^{ -3 } + -u^{-3} a^{ 2 } +( -1 + u^{-1} + u^{-3} - u^{-4} )a^{ 2 }b^{ -1 } \\[5pt]
 1 + u^{-2} - u^{-3} + u^{-6} + u^{-6} a^{ 2 }b^{ -2 } + -u^{-1} b^{ -2 } +( u^{-3} - u^{-4} )a^{ 1 }b^{ -2 } +( -1 + u^{-1} + u^{-3} - u^{-4} )b^{ -1 } \\
+ ( u^{-2} - 2 u^{-3} + u^{-4} + u^{-6} - u^{-7} )a^{ 1 }b^{ -1 } + -u^{-5} a^{ 2 } +( -u^{-2} + u^{-3} + u^{-5} - u^{-6} )a^{ 1 } +( u^{-5} - u^{-6} )a^{ 2 }b^{ -1 } \\[5pt]
( -1 - u^{-2} + 2 u^{-3} - u^{-6} )a^{ 1 }b^{ -1 } + u^{-1} a^{ 1 }b^{ -3 } + u^{-2} a^{ 2 }b^{ -3 } +( 1 - u^{-1} - u^{-3} + u^{-4} )a^{ 1 }b^{ -2 } +( u^{-1} - u^{-2} + u^{-5} )a^{ 2 }b^{ -2 } \\
+( -u^{-1} + u^{-4} - u^{-5} )a^{ 2 }b^{ -1 } +( u^{-2} - u^{-5} )a^{ 1 } + -u^{-4} a^{ 2 }
\end{array}\right)$$

    \begin{tcolorbox}[breakable, size=fbox, boxrule=1pt, pad at break*=1mm,colback=cellbackground, colframe=cellborder]
\prompt{In}{incolor}{36}{\boxspacing}
\begin{Verbatim}[commandchars=\\\{\}]
\PY{n}{ZZ}\PY{p}{[}\PY{p}{:}\PY{p}{,}\PY{l+m+mi}{2}\PY{p}{]}
\end{Verbatim}
\end{tcolorbox}

\prompt{Out}{outcolor}{36}{}

    $$\tiny\newcommand{\Bold}[1]{\mathbf{#1}}\left(\begin{array}{r}
( -1 + 2 u^{-1} - u^{-2} - u^{-4} + u^{-5} )a^{ 2 }b^{ -1 } +( u - 1 )a^{ 2 }b^{ -3 } +( u^{2} - u - u^{-1} + 2 u^{-2} - u^{-3} )a^{ 2 }b^{ -2 } +( -u^{-3} + u^{-4} )a^{ 1 }b^{ -1 } \\
+( u^{-4} - u^{-5} )a^{ 1 }b^{ -3 } +( -u^{-2} + u^{-3} + u^{-5} - u^{-6} )a^{ 1 }b^{ -2 } +( -u^{-3} + u^{-4} )a^{ 2 } \\[5pt]
( -u^{-6} + u^{-7} )a^{ 2 }b^{ -1 } + ( u^{-1} - u^{-2} - u^{-4} + 2 u^{-5} - u^{-6} )b^{ -1 } +( -u^{-3} + 2 u^{-4} - u^{-5} - u^{-7} + u^{-8} )a^{ 1 }b^{ -1 } + 1 - u^{-1} + u^{-2} 
- 3 u^{-3} + 2 u^{-4} \\
+ u^{-6} - u^{-7} +( -u^{-2} + 2 u^{-3} - u^{-4} + u^{-5} - 2 u^{-6} + u^{-7} )a^{ 1 } +( u^{-2} - u^{-3} )a^{ -1 }b^{ -1 } +( -1 + u^{-1} + u^{-3} - u^{-4} )a^{ -1 } +( -u^{-5} + u^{-6} )a^{ 2 } \\[5pt]
 u^{-3} +( u^{-2} - u^{-3} - u^{-5} + u^{-6} )a^{ 1 } +( -u^{-1} + u^{-2} - u^{-5} + u^{-6} )a^{ 1 }b^{ -2 } +( -u^{-2} + u^{-3} )a^{ 2 }b^{ -2 }
  +( -1 + u^{-1} + 2 u^{-3} - 3 u^{-4} + u^{-7} )a^{ 1 }b^{ -1 }\\ +( -u^{-1} + u^{-2} - u^{-5} + u^{-6} )a^{ 2 }b^{ -1 } +( -u^{-4} + u^{-5} )b^{ -2 }
  +( u^{-2} - u^{-3} - u^{-5} + u^{-6} )b^{ -1 } +( -u^{-4} + u^{-5} )a^{ 2 }
\end{array}\right)$$

    \begin{tcolorbox}[breakable, size=fbox, boxrule=1pt, pad at break*=1mm,colback=cellbackground, colframe=cellborder]
\prompt{In}{incolor}{37}{\boxspacing}
\begin{Verbatim}[commandchars=\\\{\}]
\PY{n}{ZZ}\PY{o}{*}\PY{n}{Ma}\PY{o}{\PYZhy{}}\PY{n}{Ma}\PY{o}{*}\PY{n}{MHa}\PY{p}{(}\PY{n}{ZZ}\PY{p}{)} \PY{c+c1}{\PYZsh{} check that Tc is central}
\end{Verbatim}
\end{tcolorbox}

\prompt{Out}{outcolor}{37}{}
    
    $$\newcommand{\Bold}[1]{\mathbf{#1}}\left(\begin{array}{rrr}
0 & 0 & 0 \\
0 & 0 & 0 \\
0 & 0 & 0
\end{array}\right)$$

    \begin{tcolorbox}[breakable, size=fbox, boxrule=1pt, pad at break*=1mm,colback=cellbackground, colframe=cellborder]
\prompt{In}{incolor}{38}{\boxspacing}
\begin{Verbatim}[commandchars=\\\{\}]
\PY{n}{ZZ}\PY{o}{*}\PY{n}{Mb}\PY{o}{\PYZhy{}}\PY{n}{Mb}\PY{o}{*}\PY{n}{MHb}\PY{p}{(}\PY{n}{ZZ}\PY{p}{)}
\end{Verbatim}
\end{tcolorbox}

\prompt{Out}{outcolor}{38}{}
    
    $$\newcommand{\Bold}[1]{\mathbf{#1}}\left(\begin{array}{rrr}
0 & 0 & 0 \\
0 & 0 & 0 \\
0 & 0 & 0
\end{array}\right)$$

%% file: HeisenbergHomology.bbl
\begin{thebibliography}{10}

\bibitem{An}
Byung~Hee An and Ki~Hyoung Ko.
\newblock {\em A family of representations of braid groups on surfaces}.
\newblock Pacific J. Math., 247(2):257--282, 2010.

\bibitem{AP}
Cristina A.-M. Anghel and Martin Palmer.
\newblock {\em Lawrence-{B}igelow representations, bases and duality}.
\newblock ArXiv:\href{http://arxiv.org/abs/2011.02388}{2011.02388}.

\bibitem{BarNatan}
Dror Bar~Natan.
\newblock {\em A note on the unitarity property of the {G}assner invariant}.
\newblock Vestn. Chelyab. Gos. Univ. Mat. Mekh. Inform., 3:22--25, 2015.

\bibitem{Bellingeri}
Paolo Bellingeri.
\newblock {\em On presentations of surface braid groups}.
\newblock J. Algebra, 274(2):543--563, 2004.

\bibitem{B_al2008}
Paolo Bellingeri, Sylvain Gervais, and John Guaschi.
\newblock {\em Lower central series of {A}rtin-{T}its and surface braid
  groups}.
\newblock J. Algebra, 319(4):1409--1427, 2008.

\bibitem{BellingeriGodelle2007}
Paolo Bellingeri and Eddy Godelle.
\newblock {\em Positive presentations of surface braid groups}.
\newblock J. Knot Theory Ramifications, 16(9):1219--1233, 2007.

\bibitem{B_al2011}
Paolo Bellingeri, Eddy Godelle, and John Guaschi.
\newblock {\em Exact sequences, lower central series and representations of
  surface braid groups}.
\newblock ArXiv:\href{http://arxiv.org/abs/1106.4982}{1106.4982}.

\bibitem{BGG2017}
Paolo Bellingeri, Eddy Godelle, and John Guaschi.
\newblock {\em Abelian and metabelian quotient groups of surface braid groups}.
\newblock Glasg. Math. J., 59(1):119--142, 2017.

\bibitem{Bianchi2020}
Andrea Bianchi.
\newblock {\em Splitting of the homology of the punctured mapping class group}.
\newblock J. Topol., 13(3):1230--1260, 2020.

\bibitem{BianchiMillerWilson}
Andrea Bianchi, Jeremy Miller, and Jennifer C.~H. Wilson.
\newblock {\em Mapping class group actions on configuration spaces and the
  {J}ohnson filtration}.
\newblock Trans. Amer. Math. Soc., 375(8):5461--5489, 2022.

\bibitem{BianchiStavrou2022}
Andrea Bianchi and Andreas Stavrou.
\newblock {\em Non-trivial action of the {J}ohnson filtration on the homology
  of configuration spaces}.
\newblock ArXiv:\href{http://arxiv.org/abs/2208.01608}{2208.01608}.

\bibitem{Bigelow_Hecke}
Stephen Bigelow.
\newblock {\em Homological representations of the {I}wahori-{H}ecke algebra}.
\newblock In {\em Proceedings of the {C}asson {F}est}, volume~7 of {\em Geom.
  Topol. Monogr.}, pages 493--507. Geom. Topol. Publ., Coventry, 2004.

\bibitem{BigelowMartel}
Stephen Bigelow and Jules Martel.
\newblock {\em Quantum groups from homologies of configuration spaces}.
\newblock ArXiv:\href{http://arxiv.org/abs/2405.06982}{2405.06982}, 2024.

\bibitem{Bigelow2001}
Stephen~J. Bigelow.
\newblock {\em Braid groups are linear}.
\newblock J. Amer. Math. Soc., 14(2):471--486, 2001.

\bibitem{Bigelow-Budney}
Stephen~J. Bigelow and Ryan~D. Budney.
\newblock {\em The mapping class group of a genus two surface is linear}.
\newblock Algebr. Geom. Topol., 1:699--708, 2001.

\bibitem{BPS2}
Christian Blanchet, Martin Palmer, and Awais Shaukat.
\newblock {\em Action of subgroups of the mapping class group on {H}eisenberg
  homologies}.
\newblock ArXiv:\href{http://arxiv.org/abs/2306.08614}{2306.08614}. To appear
  in AMS Contemporary Mathematics, 2023.

\bibitem{Bredon}
Glen~E. Bredon.
\newblock {\em Sheaf Theory}, volume 170 of {\em Graduate Texts in
  Mathathematics}.
\newblock Springer, 1997.

\bibitem{CartanEilenberg1956}
Henri Cartan and Samuel Eilenberg.
\newblock {\em Homological algebra}.
\newblock Princeton University Press, Princeton, NJ, 1956.

\bibitem{Crivelli93}
M.~Crivelli, G.~Felder, and C.~Wieczerkowski.
\newblock {\em Generalized hypergeometric functions on the torus and the
  adjoint representation of {$U_q({\rm sl}_2)$}}.
\newblock Comm. Math. Phys., 154(1):1--23, 1993.

\bibitem{Crivelli94}
M.~Crivelli, G.~Felder, and C.~Wieczerkowski.
\newblock {\em Topological representations of {$U_q({\rm sl}_2(\mathbf C))$} on
  the torus and the mapping class group}.
\newblock Lett. Math. Phys., 30(1):71--85, 1994.

\bibitem{DPS}
Jacques Darn{\'e}, Martin Palmer, and Arthur Souli{\'e}.
\newblock {\em When the lower central series stops: a comprehensive study for
  braid groups and their relatives}.
\newblock ArXiv:\href{http://arxiv.org/abs/2201.03542}{2201.03542}. To appear
  in the Memoirs of the American Mathematical Society.

\bibitem{Davis}
James~F. Davis and Paul Kirk.
\newblock {\em Lecture notes in algebraic topology}, volume~35 of {\em Graduate
  Studies in Mathematics}.
\newblock American Mathematical Society, Providence, RI, 2001.

\bibitem{DeRenziMartel}
Marco {De Renzi} and Jules Martel.
\newblock {\em Homological Construction of Quantum Representations of Mapping
  Class Groups}.
\newblock ArXiv:\href{http://arxiv.org/abs/2212.10940}{2212.10940}, 2022.

\bibitem{FarbMargalit}
Benson Farb and Dan Margalit.
\newblock {\em A primer on mapping class groups}, volume~49 of {\em Princeton
  Mathematical Series}.
\newblock Princeton University Press, Princeton, NJ, 2012.

\bibitem{Gelca_Hamilton}
R\u{a}zvan Gelca and Alastair Hamilton.
\newblock {\em The topological quantum field theory of {R}iemann's theta
  functions}.
\newblock J. Geom. Phys., 98:242--261, 2015.

\bibitem{Gelca_Uribe_TQFT}
R\u{a}zvan Gelca and Alejandro Uribe.
\newblock {\em From classical theta functions to topological quantum field
  theory}.
\newblock In {\em The influence of {S}olomon {L}efschetz in geometry and
  topology}, volume 621 of {\em Contemp. Math.}, pages 35--68. Amer. Math.
  Soc., Providence, RI, 2014.

\bibitem{Gelca_Uribe}
R\u{a}zvan Gelca and Alejandro Uribe.
\newblock {\em Quantum mechanics and nonabelian theta functions for the gauge
  group {${\rm SU}(2)$}}.
\newblock Fund. Math., 228(2):97--137, 2015.

\bibitem{Harer1985}
John~L. Harer.
\newblock {\em Stability of the homology of the mapping class groups of
  orientable surfaces}.
\newblock Ann. of Math. (2), 121(2):215--249, 1985.

\bibitem{Hatcher}
Allen Hatcher.
\newblock {\em Algebraic topology}.
\newblock Cambridge University Press, Cambridge, 2002.

\bibitem{Heistad1967}
Egil Heistad.
\newblock {\em Excision in singular theory}.
\newblock Math. Scand., 20:61--64, 1967.

\bibitem{HBE}
Derek~F. Holt, Bettina Eick, and Eamonn~A. O'Brien.
\newblock {\em Handbook of computational group theory}.
\newblock Discrete Mathematics and its Applications (Boca Raton). Chapman \&
  Hall/CRC, Boca Raton, FL, 2005.

\bibitem{Johnson-survey}
Dennis Johnson.
\newblock {\em A survey of the {T}orelli group}.
\newblock In {\em Low-dimensional topology ({S}an {F}rancisco, {C}alif.,
  1981)}, volume~20 of {\em Contemp. Math.}, pages 165--179. Amer. Math. Soc.,
  Providence, RI, 1983.

\bibitem{Korkmaz}
Mustafa Korkmaz.
\newblock {\em Low-dimensional homology groups of mapping class groups: a
  survey}.
\newblock Turkish J. Math., 26(1):101--114, 2002.

\bibitem{Krammer2002}
Daan Krammer.
\newblock {\em Braid groups are linear}.
\newblock Ann. of Math. (2), 155(1):131--156, 2002.

\bibitem{Lawrence1990}
R.~J. Lawrence.
\newblock {\em Homological representations of the {H}ecke algebra}.
\newblock Comm. Math. Phys., 135(1):141--191, 1990.

\bibitem{Levine1977}
Jerome Levine.
\newblock {\em Knot modules. {I}}.
\newblock Trans. Amer. Math. Soc., 229:1--50, 1977.

\bibitem{LionVergne}
G\'{e}rard Lion and Mich\`ele Vergne.
\newblock {\em The {W}eil representation, {M}aslov index and theta series},
  volume~6 of {\em Progress in Mathematics}.
\newblock Birkh\"{a}user, Boston, Mass., 1980.

\bibitem{Martel}
Jules Martel.
\newblock {\em A homological model for {$U_q\mathfrak{sl}2$} {V}erma modules
  and their braid representations}.
\newblock Geom. Topol., 26(3):1225--1289, 2022.

\bibitem{Morita1989}
Shigeyuki Morita.
\newblock {\em Families of {J}acobian manifolds and characteristic classes of
  surface bundles. {I}}.
\newblock Ann. Inst. Fourier (Grenoble), 39(3):777--810, 1989.

\bibitem{Moriyama}
Tetsuhiro Moriyama.
\newblock {\em The mapping class group action on the homology of the
  configuration spaces of surfaces}.
\newblock J. Lond. Math. Soc. (2), 76(2):451--466, 2007.

\bibitem{PS}
Martin Palmer and Arthur Souli{\'e}.
\newblock {\em On the kernels of {M}oriyama-type representations of the mapping
  class groups}.
\newblock In preparation.

\bibitem{Spanier}
Edwin Spanier.
\newblock {\em Singular homology and cohomology with local coefficients and
  duality for manifolds}.
\newblock Pacific J. Math., 160(1):165--200, 1993.

\bibitem{Squier}
Craig~C. Squier.
\newblock {\em The {B}urau representation is unitary}.
\newblock Proc. Amer. Math. Soc, 90(2):199--202, 1984.

\bibitem{Stavrou2023}
Andreas Stavrou.
\newblock {\em Cohomology of configuration spaces of surfaces as mapping class
  group representations}.
\newblock Trans. Amer. Math. Soc., 376(4):2821--2852, 2023.

\bibitem{Suzuki2003}
Masaaki Suzuki.
\newblock {\em Irreducible decomposition of the {M}agnus representation of the
  {T}orelli group}.
\newblock Bull. Austral. Math. Soc., 67(1):1--14, 2003.

\bibitem{Suzuki2005}
Masaaki Suzuki.
\newblock {\em Geometric interpretation of the {M}agnus representation of the
  mapping class group}.
\newblock Kobe J. Math., 22(1-2):39--47, 2005.

\bibitem{Wahl2013}
Nathalie Wahl.
\newblock {\em Homological stability for mapping class groups of surfaces}.
\newblock In {\em Handbook of moduli. {V}ol. {III}}, volume~26 of {\em Adv.
  Lect. Math. (ALM)}, pages 547--583. Int. Press, Somerville, MA, 2013.

\bibitem{Weibel}
Charles~A. Weibel.
\newblock {\em An introduction to homological algebra}, volume~38 of {\em
  Cambridge Studies in Advanced Mathematics}.
\newblock Cambridge University Press, Cambridge, 1994.

\end{thebibliography}
